\documentclass[a4paper,11pt]{amsart}

\usepackage{booktabs}
\usepackage{latexsym}
\usepackage{amssymb}
\usepackage{mathrsfs,mathbbol}
\usepackage{color}
\usepackage[
frame,cmtip,arrow,matrix,line,graph]{xy}
\usepackage{comment}

\let\oldtocsection=\tocsection

\let\oldtocsubsection=\tocsubsection

\let\oldtocsubsubsection=\tocsubsubsection

\renewcommand{\tocsection}[2]{\hspace{0em}\oldtocsection{#1}{#2}}
\renewcommand{\tocsubsection}[2]{\hspace{1em}\oldtocsubsection{#1}{#2}}
\renewcommand{\tocsubsubsection}[2]{\hspace{2em}\oldtocsubsubsection{#1}{#2}}

\DeclareRobustCommand{\SkipTocEntry}[5]{}

\definecolor{darkgreen}{RGB}{47,139,79}
\definecolor{darkblue}{RGB}{36,24,130}

\usepackage{hyperref}
\hypersetup{
    colorlinks=true, 
    linktoc=all,     
    linktocpage,
    citecolor=darkgreen,
    filecolor=black,
    linkcolor=darkblue,
    urlcolor=black
}

\input xy
\xyoption{all}

\usepackage{amsmath}
\usepackage{amsthm}
\usepackage{amsfonts}
\usepackage{amscd}

\usepackage{epsfig}
\usepackage{lpic}
\usepackage{enumerate}

\newtheorem{thm}{Theorem}[section]
\newtheorem{lem}[thm]{Lemma}
\newtheorem{prop}[thm]{Proposition}
\newtheorem{cor}[thm]{Corollary}
\newtheorem{ind}[thm]{Inductive hypothesis}
\newtheorem{Th}{Theorem}
\newtheorem{Co}[Th]{Corollary}
\newtheorem{Coj}[Th]{Conjecture}
\theoremstyle{definition}
\newtheorem{Def}[thm]{Definition}
\newtheorem{con}[thm]{Conjecture}

\newtheorem{ex}[thm]{Example}
\newtheorem{exprop}[thm]{Example and Proposition}
\newtheorem{rem}[thm]{Remark}
\numberwithin{equation}{section}

\newcommand{\A}{\mathcal{A}}

\newcommand{\C}{\mathcal{C}}

\newcommand{\e}{\mathcal{E}}

\newcommand{\G}{\mathcal{G}}

\newcommand{\M}{\mathcal{M}}
\newcommand{\N}{\mathbb{N}}

\newcommand{\Z}{\mathbb{Z}}

\newcommand{\al}{\alpha}
\newcommand{\be}{\beta}

\newcommand{\Ga}{\Gamma}

\newcommand{\De}{\Delta}

\newcommand{\la}{\lambda}

\newcommand{\s}{\sigma}
\newcommand{\Si}{\Sigma}

\newcommand{\Dif}{\operatorname{Diff}}
\newcommand{\Hom}{\operatorname{Hom}}
\newcommand{\coker}{\operatorname{coker}}

\newcommand{\surj}{\twoheadrightarrow}

\newcommand{\rar}{\longrightarrow}
\newcommand{\inc}{\hookrightarrow}
\newcommand{\sta}{\stackrel}
\newcommand{\arsim}{\sta{\sim}{\rar}}
\newcommand{\minus}{\backslash}

\newcommand{\x}{\times}

\newcommand{\ot}{\otimes}
\newcommand{\op}{\oplus}

\newcommand{\lgl}{\langle}
\newcommand{\rgl}{\rangle}

\newcommand{\del}{\partial}

\newcommand{\emp}{\varnothing}

\newcommand{\note}[1]{}
\newcommand{\orw}[1]{}

\newcommand{\Aut}{\operatorname{Aut}}
\newcommand{\Emb}{\operatorname{Emb}}

\newcommand{\colim}{\operatorname{colim}}

\newcommand{\link}{\operatorname{Link}}
\newcommand{\st}{\operatorname{Star}}

\newcommand{\Fix}{\operatorname{Fix}}
\newcommand{\St}{\operatorname{Stab}}

\newcounter{samcounter}

\newcommand{\wt}[1]{\widetilde{#1}}

\newcommand{\bF}{\mathbf{F}}

\newcommand{\bN}{\mathbb{N}}

\newcommand{\bQ}{\mathbb{Q}}
\newcommand{\bR}{\mathbb{R}}

\newcommand{\bZ}{\mathbb{Z}}

\newcommand{\cB}{\mathcal{B}}

\newcommand\lra{\longrightarrow}

\newcommand\U{U\!}

\newcommand{\EI}[1]{$({\bf E}_{I}{#1})$}
\newcommand{\Ei}[1]{$({\bf E}_{#1})$}

\newcommand{\II}[1]{$({\bf I}_{I}{#1})$}
\newcommand{\Ii}[1]{$({\bf I}_{#1})$}

\begin{document}

\title{Homological stability for automorphism groups}

\author{Oscar Randal-Williams}
\email{o.randal-williams@dpmms.cam.ac.uk}
\address{Centre for Mathematical Sciences\\
Wilberforce Road\\
Cambridge CB3 0WB\\
UK}

\author{Nathalie Wahl}
\email{wahl@math.ku.dk}
\address{Department of Mathematical Sciences\\University of Copenhagen\\Universitetsparken 5\\DK-2100 Copenhagen\\Denmark }

\date{\today}

\begin{abstract}
Given a family of groups admitting a braided monoidal structure (satisfying mild assumptions) 
we construct a family of spaces on which the groups act and whose connectivity yields, via a classical argument of Quillen, 
homological stability for the family of groups. We show that stability also holds with both polynomial and abelian twisted coefficients, with no further assumptions. 
This new construction of a family of spaces from a family of groups recovers known spaces in the classical examples 
of stable families of groups, such as the symmetric groups, general linear groups and mapping class groups.    
By making systematic the proofs of classical stability results, we show that they  all hold with the same type of coefficient systems, obtaining in particular  
without any further work 
new stability theorems with twisted coefficients for the symmetric groups, braid groups, automorphisms of free groups, unitary groups,
mapping class groups of non-orientable surfaces and mapping class
groups of 3-manifolds. Our construction can also be applied to families of groups not considered before in the context of homological stability.

As a byproduct of our work, we construct the braided analogue of the
category $FI$ of finite sets and injections relevant to the present context, 
and define polynomiality for functors in the context of pre-braided monoidal categories. 
\end{abstract}

\maketitle

A family of groups 
$$G_1\inc G_2\inc \cdots \inc G_n\inc \cdots$$
is said to satisfy {\em homological stability} if the induced maps 
$$H_i(G_n)\rar H_i(G_{n+1})$$
are isomorphisms in a range $0\le i\le f(n)$ increasing with $n$. In this paper, we prove that homological stability always holds if there is a monoidal category $\C$ satisfying a certain hypothesis and a pair of objects $A$ and $X$ in $\C$,  such that $G_n$ is the group of automorphisms of  $A\op X^{\op n}$ in $\C$. We show that stability holds not just for constant coefficients, but also for both {\em polynomial} and {\em abelian} coefficients, without any further assumption on $\C$. 

The polynomial coefficient systems considered here are functors $F: \C\to \Z\operatorname{-Mod}$ satisfying a finite degree condition.  They are generalisations of polynomial functors in the sense of functor homology, classically considered in homological stability, and include new examples such as the Burau representation of braid groups. Abelian coefficients are given by  functors $F:\C\to \Z G_\infty^{ab}\operatorname{-Mod}$ for $G_\infty^{ab}$ the abelianisation of the limit group $G_\infty$, satisfying the same finite degree condition. These include coefficients such as the sign representation, or determinant-twisted polynomial functors. Such coefficient systems are newer to the subject. One consequence of stability with abelian coefficients is that stability with polynomial coefficients also holds (under the same conditions) for the commutator subgroups $G_n' \leq G_n$. 

Our theorem applies to all the classical examples and gives new stability results with twisted coefficients in particular for 
symmetric groups, alternating groups, unitary groups, braid groups, mapping class groups of non-orientable surfaces, automorphisms and symmetric automorphisms of free groups, and  it proves stability for these groups in a unified way. Our framework also applies to new families of groups (see e.g.~\cite{GanWah15,PatWu15,SzyWah}).

We now describe our results and techniques in more detail. 

\medskip

A family of groups $\{G_n\}_{n\ge 1}$ can be thought of as a groupoid with objects the natural numbers, with $\Aut(0)=\{id\}$ and
$\Aut(n)=G_n$ for $n\ge 1$. Families of groups such as the symmetric groups, braid groups, and general linear groups have the additional
feature that they admit a ``block sum'' operation
$$G_n\x G_m \sta{\op}\rar G_{n+m}.$$ 
On the level of groupoids, these maps define in each case a symmetric or braided monoidal structure. 

Given a braided monoidal groupoid $(\G,\op,0)$ such that 
\begin{enumerate}[(i)]
\item the monoid of isomorphism classes of objects under the operation $\op$ satisfies cancellation and has no zero divisors\footnote{i.e.\ $U\op V\cong 0$ in $\G$ implies that $U\cong 0\cong V$.}, and

\item $\Aut(A)\to \Aut(A\op B)$ is injective for all objects $A,B \in \mathcal{C}$,
\end{enumerate}
our Theorem~\ref{universal} constructs what we call a {\em homogeneous category} $(U\G,\op,0)$. 
A homogeneous category (see Definition~\ref{homdef}) is a monoidal category $(\C,\op,0)$ in which 0 is initial, such that
\begin{enumerate}[(i)]
\item  for all pairs of objects $A,B \in \C$ the group
$\Aut(B)$ acts transitively on the set $\Hom(A,B)$, and

\item the natural map $\Aut(A)\to \Aut(A\op B)$ is injective with image those automorphisms fixing the canonical morphism $B\to A\op B$.
\end{enumerate}

For the groupoid of finite sets and their bijections, the homogeneous category we associate is the category $FI$ of finite sets and injections, much studied recently
by Church, Ellenberg, Farb, and others in the context of representation stability
(see e.g.\ \cite{CEF12}). For the groupoid of finitely generated $R$-modules and their isomorphisms, the construction yields the category with objects the finitely generated $R$-modules and morphisms the injective homomorphisms equipped with a choice of splitting, which arises in some form already in Charney's work on stability phenomena for congruence subgroups \cite{Cha84}. And, for mapping class groups of surfaces of increasing genus, we recover a version of Ivanov's category of decorated surfaces, which he used to prove stability with twisted coefficients for these groups \cite{Iva93}. The construction appears to be new for braid groups, where it yields a {\em pre-braided} homogeneous category, which can be thought of as the braided analogue of $FI$. This new category allows us to define a notion of polynomial functors associated to braid groups, which is of independent interest. (The property of being pre-braided is weaker than that of being braided, see Definition~\ref{prebraidDef}. See also Section~\ref{prebraidcal} for a graphical description of the pre-braided category associated to the braid groups.)

\medskip

Given a homogeneous category $(\C,\op,0)$ and objects $A$ and $X$ in $\C$, we consider the question of homological stability for the groups 
$$G_n=\Aut(A\op X^{\op n})$$
with coefficients in the $\bZ[G_n]$-module
$$F_n=F(A\op X^{\op n})$$
for $F:\C\to \Z\operatorname{-Mod}$ a functor to the category of $\Z$-modules. One should think of $A$ as a ``starting point'' and $X$ as a ``stabilisation direction''. Often $\C$ will have objects the natural numbers and we will take $A=0$ and $X=1$, so that $A\op X^{\op n}$ will just be the object $n$ in $\C$, but in general $\C$ may have many interesting starting points and stabilisation directions, for example when $\C$ is a category whose objects are manifolds or groups. 

We say, inductively, that a functor $F$ as above defines a {\em coefficient system of degree $r$} if the kernel of $F(-) \to F(X\op -)$ vanishes and the cokernel defines a coeffcient system of degree $(r-1)$; the zero coefficient system has degree $-1$. More generally we say that $F$ defines a {\em coefficient system of degree $r$ at $N$} if the above holds when evaluated at the objects $A\op X^{\op n}$ for $n\ge N$ (see Definition~\ref{fdegcoef} for full details). 
As detailed in Example~\ref{FIex} and Remark~\ref{polrem}, this notion is closely related to the classical notion of polynomial functors and, when $\C=FI$, to that of being presented in finite degree.

Let $G_\infty$ denote the colimit of the groups $G_n$ and $G_\infty^{ab}$ its abelianisation. We will also consider functors $F:\C\to \Z G_\infty^{ab}\operatorname{-Mod}$ with target the category of $G_\infty^{ab}$-modules. In this case, we have that $F_n=F(A\op X^{\op n})$ comes equipped with two commuting actions of $G_n$, one from the functoriality of $F$, and one via the composition $G_n\to G_\infty\to G_\infty^{ab}$ and its $G_\infty^{ab}$-module structure. For such functors, we let $F_n^\circ$ denote the $G_n$-module $F(A\op X^{\op n})$ considered with the diagonal action. Our stability theorem, given below, will allow both $F_n$ and $F_n^\circ$ as coefficients. 

\smallskip
 
To each triple $(\C,A,X)$, we associate a family of semi-simplicial sets $W_n(A,X)_\bullet$ for $n\geq 1$, built out of the morphism sets in $\C$ (see Definition~\ref{simpdef}). The semi-simplicial set $W_n(A,X)_\bullet$ admits an action of $G_n$, and can be thought of as a canonical ``space of destabilisations'' for $G_n$ associated to the family of groups.  Our main theorem is the following:

\begin{Th}\label{main}
Let $(\C,\op,0)$ be a pre-braided homogeneous category, and $A$ and $X$ be two objects of $\C$. 
Assume that there is a $k\ge 2$ such that for each $n\ge 1$, the semi-simplicial set $W_n(A,X)_\bullet$ mentioned above is at least $(\frac{n-2}{k})$-connected. 
Let $F:\C\to \Z\operatorname{-Mod}$ be a coefficient system of degree $r$ at $N$,  and $G_n=\Aut(A\op X^{\op n})$ and $F_n=F(A\op X^{\op n})$. Then 
$$H_i(G_n;F_n) \rar H_i(G_{n+1};F_{n+1}) $$ 
is an epimorphism for $i\le \frac{n}{k}-r$ and an isomorphism for  $i\le \frac{n}{k}-r-1$ for all $n>N$.

Moreover, for $F:\C\to \Z G_\infty^{ab}\operatorname{-Mod}$ and assuming now that $k\ge 3$, we have that 
$$H_i(G_n;F^\circ_n) \rar H_i(G_{n+1};F^\circ_{n+1}) $$ 
is an epimorphism for $i\le \frac{n+2}{k}-r-1$ and an isomorphism for  $i\le \frac{n+2}{k}-r-2$ for all $n>2N$. 
\end{Th}

As we will make explicit in the more technical Theorems \ref{stabthm}, \ref{abstabthm}, and~\ref{twistrange}, the stability ranges can be improved when the coefficient system is constant or split. A consequence of Theorem \ref{main} is the following: 

\begin{Co}\label{cor:B}
Let $\C$, $A$, $X$ and $F:\C\to \Z\operatorname{-Mod}$ be as in Theorem \ref{main} and assume that $W_n(A,X)_\bullet$ is at least $(\frac{n-2}{k})$-connected for some $k\ge 3$. Let $G_n'$ denote the commutator subgroup of $G_n=\Aut(A\op X^{\op n})$. Then 
$$H_i(G_n';F_n) \rar H_i(G_{n+1}';F_{n+1}) $$ 
is an epimorphism for $i\le \frac{n+2}{k}-r-1$ and an isomorphism for  $i\le \frac{n+2}{k}-r-2$ for all $n>2N$. 
\end{Co}

The assumption on the stability slope $k\ge 2$ in the first part of Theorem~\ref{main} and $k\ge 3$ in the two other statements is known to be best possible in this level of generality as shown for example by the computations \cite{Nak60} in the case of symmetric groups with constant coefficients and \cite{Hausmann} in the case of their commutator subgroups, the alternating groups, also with constant coefficients.

In Section ~\ref{sec:StabHomology}, we briefly explain how the method of group completion can be used to compute the stable homology in the case of constant and constant abelian coefficients. 

\medskip

To apply the above result to a given family of groups $G_1\inc G_2\inc\cdots$, we must first identify the family as a family of automorphism groups in a monoidal groupoid. If the groupoid admits a braiding and satisfies cancellation, we can apply our Theorem~\ref{universal} to obtain an associated homogeneous category and a family of semi-simplicial sets $W_n(A,X)_\bullet$. Homological stability will follow if 
these semi-simplicial sets 
are highly connected. So there are three main assumptions on the groups for stability to hold: the existence of a {\em braiding}, the {\em cancellation} property, and the {\em connectivity} property. 

The braiding assumption rules out many examples that are known not to satisfy stability, such as $G_n=\Z^n$ or $G_n=P\be_n$ is the pure braid group on $n$ strands. Cancellation is not a strong assumption: we give in Definition~\ref{LCDef} a local version of cancellation which suffices, and the stability theorems in the paper are actually stated using a local version of the axioms given in this introduction. If the groupoid does not satisfy cancellation at all, it can always be replaced by one that does (though at the expense of changing the semi-simplicial sets $W_n(A,X)_\bullet$). The assumption which usually requires work to prove is that of connectivity. Indeed, there are many examples of
families of groups that fit this framework, but we do not always know whether or not the connectivity hypothesis is satisfied. We conjecture in fact that stability holds if and only if connectivity holds: 

\begin{Coj}\label{conjstab}
Let $(\C,\op,0)$ be a homogeneous category and $A$ and $X$ objects of $\C$.  If  $G_n=\Aut(A\op X^{\op n})$ satisfies homological stability for all twisted coefficient systems as in Theorem~\ref{main}, then the connectivity of the $W_n(A,X)_\bullet$ tends to infinity with $n$.
\end{Coj}

A conjecture of the same flavour has been made by Thomas Church.

\smallskip

To be able to check connectivity of the semi-simplicial sets $W_n(A,X)_\bullet$, we define a family of simplicial complexes $S_n(A,X)$ and show that, under favourable circumstances, the $W_n(A,X)_\bullet$'s are highly connected if and only if the associated $S_n(A,X)$'s are highly connected. (See Theorem~\ref{caseA}.) In all the classical examples, the complexes $S_n(A,X)$ have previously appeared in the literature (at least essentially), and their connectivity is known, meaning that no further work is required. Combining our main theorem with these connectivity results, we obtain a number of stability results, many of which had not yet appeared in the literature. We will shortly give a sample of such theorems and indicate which other stability theorems likewise follow from our work. 

Our set-up, that of {\em homogeneous categories}, is inspired by a set-up developed by Djament--Vespa to compute stable homology with twisted coefficients \cite{DjaVes10}.
The proofs of the stability theorems given here are generalisations of work of Charney \cite{Cha87}, Dwyer \cite{Dwy80}, van der Kallen \cite{vdK80}, and Galatius and the first author \cite{GalRW14}.

A version of Theorem~\ref{main} in the case of symmetric monoidal homogeneous categories and for polynomial coefficients was earlier sketched by Djament
\cite{Dja12b}. The methods we use are also very similar to those used by Putman and Sam in their independent work \cite{PutSam14}, where they prove representation stability of
certain congruence subgroups. In particular, Putman--Sam introduce a notion of {\em complemented categories} which are particular cases of homogeneous categories.

\addtocontents{toc}{\SkipTocEntry}
\subsection*{Examples}

We now give a selection of the results which follow from Theorem \ref{main} and Corollary \ref{cor:B} when they are applied to particular families of groups. We shall not give every possible result that can be deduced from these results, but rather a representative sample. These examples will be explained in detail in Section \ref{examples}, where we shall also give more technical statements improving the stability ranges under certain conditions.

\addtocontents{toc}{\SkipTocEntry}
\subsection*{Braid groups and symmetric groups}

Let $S$ be a surface with one 
boundary and let us write $\be_n^S= \pi_1\operatorname{Conf}(\mathrm{int}S,n)$ for the corresponding surface braid group (so $\be_n^{D^2}=\be_n$ is the usual braid group). Just as for the usual braid group, there is a map $\be_n^S\to \Si_n$ and we denote by 
$G\wr \be_n^S$ the corresponding wreath product. 

\begin{Th}\label{thm:D}
Let $G$ be a group, $S$ a surface with one boundary, and let 
$$G_n=G\wr \Si_n\ \ \textrm{or}\ \ \ G\wr \be_n^S$$
with $G_n'$ denoting its commutator subgroup. Let $F:U(\sqcup_n G_n)\to \Z\operatorname{-Mod}$ be a coefficient system of degree $r$ at $N$. Then for all $n>N$, $H_i(G_n;F_n)\to H_i(G_{n+1};F_{n+1})$
is an  epimorphism for $i \leq \tfrac{n}{2}-r$ and an isomorphism for $i \leq \tfrac{n-2}{2}-r$, and for all $n>2N$, 
$H_i(G'_n;F_n)\to H_i(G'_{n+1};F_{n+1})$
an epimorphism for $i \leq \tfrac{n-1}{3}-r$ and an isomorphism for $i \leq \tfrac{n-4}{3}-r$. 
\end{Th}

In the case of symmetric groups, $G_n=\Si_n$ and $G_n'=A_n$, the alternating group. Then $U(\sqcup_n G_n)$ is equivalent to the category $FI$ of finite sets and injections. If $F : FI \to R\operatorname{-Mod}$ is an $FI$-module which is generated in degrees $\leq k$ and related in degrees $\leq d$, then we shall show in Example~\ref{FIex} that as a coefficient system it has degree $k$ at $d+\min(k,d)$. This implies the following: 

\begin{Co}\label{cor:E}
Let $F:FI\to \Z\operatorname{-Mod}$ be an $FI$-module which is generated in degrees $\leq k$ and related in degrees $\leq d$. Then for all $n>d+\min(k,d)-1$, the map
$H_i(\Si_n;F_n)\to H_i(\Si_{n+1};F_{n+1})$
is an  epimorphism for $i \leq \tfrac{n}{2}-k$ and an isomorphism for $i \leq \tfrac{n-2}{2}-k$, and for all $n>2(d+\min(k,d)-1)$, the map
$H_i(A_n;F_n)\to H_i(A_{n+1};F_{n+1})$
an epimorphism for $i \leq \tfrac{n-1}{3}-k$ and an isomorphism for $i \leq \tfrac{n-4}{3}-k$. 
\end{Co}

For braid groups, Theorem \ref{thm:D} in particular gives the following: 

\begin{Co}\label{cor:F}
Let $B_n=\Z[t,t^{-1}]^n$ denote the Burau representation of the braid group $\be_n$. Then 
$H_i(\be_n;B_n)\to H_i(\be_{n+1};B_{n+1})$
is an  epimorphism for $i \leq \tfrac{n-2}{2}$ and an isomorphism for $i \leq \tfrac{n-4}{2}$, and  
$H_i(\be'_n;B_n)\to H_i(\be'_{n+1};B_{n+1})$
an epimorphism for $i \leq \tfrac{n-4}{3}$ and an isomorphism for $i \leq \tfrac{n-7}{3}$. 
\end{Co}

Previously known results were: Nakaoka \cite{Nak60} ($\Si_n$ with constant coefficients), Betley \cite{Bet02} ($\Si_n$ with more restrictive polynomial coefficients) and Hausmann \cite{Hausmann} ($A_n$ with constant coefficients), Arnold \cite{Arn70} ($\be_n$ with constant coefficients) and Church--Farb \cite{ChuFar13} ($\be_n$ more restrictive polynomial coefficients),  Segal \cite{Seg79}  ($\be_n^S$ with constant coefficients),  Frenkel and Callegaro \cite{Frenkel, Callegaro} (complete computation for $\be_n'$ with  constant coefficients), and Hatcher--Wahl \cite{HatWah10} ($G\wr \Si_n$ and $G\wr\be_n^S$ with constant coefficients).

\addtocontents{toc}{\SkipTocEntry}
\subsection*{Automorphism groups of free groups}

Let $\Aut(F_n)$ denote the automorphism group of the free group on $n$ generators, and $\Si\Aut(F_n)$ the subgroup of symmetric automorphisms (those generated by conjugations, permutations and taking inverses of generators). 

\begin{Th}\label{thm:G}
Let $$G_n=\Aut(F_n)\  \textrm{or}\  \Si\Aut(F_n)$$ with $G_n'$ denoting its commutator subgroup. Let $F:U(\sqcup_n G_n)\to \Z\operatorname{-Mod}$ be a coefficient system of degree $r$ at $N$. Then for all $n>N$, $H_i(G_n;F_n)\to H_i(G_{n+1};F_{n+1})$
is an epimorphism for $i \leq \tfrac{n-1}{2}-r$ and an isomorphism for $i \leq \tfrac{n-3}{2}-r$, and for all $n>2N$, 
 $H_i(G'_n;F_n)\to H_i(G'_{n+1};F_{n+1})$
an epimorphism for $i \leq \tfrac{n-2}{3}-r$ and an isomorphism for $i \leq \tfrac{n-5}{3}-r$. 
\end{Th}

Note that $\Aut(F_n)'$ identifies with the index 2 subgroup $S\Aut(F_n)$ of $\Aut(F_n)$ of those automorphisms which induce an isomorphism of determinant one on $H_1(F_n;\bZ)$. The category $U(\sqcup_n\Aut(F_n))$ is equivalent to the category of free groups and split injective homomorphisms (with respect to free product), that is pairs consisting of an injective homomorphism and a choice of splitting as free factor for its image. 
The first part of the theorem was conjectured by the first author for $\Aut(F_n)$ in \cite{Ran10}. Previously known results were:  Hatcher--Vogtmann \cite{HatVog-cerf} ($\Aut(F_n)$ with constant coefficients) and Hatcher--Wahl \cite{HatWah10}  ($\Si\Aut(F_n)$ with constant coefficients). A generalisation of the above theorem to automorphism groups of certain free products is given in Theorem~\ref{freetwist}. A further generalisation to certain families of subgroups as described in \cite[Cor.~1.3]{HatWah10} can likewise be carried out.

\addtocontents{toc}{\SkipTocEntry}
\subsection*{General linear and unitary groups} 

For a ring $R$ satisfying one of Bass' stable range conditions, we obtain homological stability for the groups $GL_n(R)$, with both twisted and abelian coefficients. However, the stability range obtained is not as good as that of van der Kallen \cite{vdK80}, as there are some tricks particular to general linear groups which do not fit in to our general framework (though we shall explain them in Section \ref{GLnsec}). On the other hand, we do not require coefficient systems to be split, so our stability theorem with twisted coefficients may be applied more broadly than \cite{vdK80}. 

If $R$ is a ring with anti-involution $\overline{\phantom{a}}$, $\epsilon \in R$ is central satisfying $\epsilon \overline{\epsilon}=1$, and $\Lambda$ is a form parameter in the sense of Bak \cite{Bak}, then there are defined (hyperbolic) unitary groups $U^\epsilon_n(R,\Lambda)$. This general construction recovers many interesting groups: for trivial involution (with $R$ commutative), $U^1_n(R, 0)$ is the orthogonal group $O_{n,n}(R)$ and $U^{-1}_n(R,R)$ is the symplectic group $Sp_{2n}(R)$; for non-trivial involution and $\epsilon=-1$, $U_n^{-1}(R,\Lambda)$ is the classical unitary group for $\Lambda$ the fixed points of the involution. But more elaborate constructions are possible: for example, $U^{-1}_n(\bZ, 2\bZ)$ is the subgroup of $Sp_{2n}(\bZ)$ of those matrices fixing an even theta-characteristic.

Mirzaii and van der Kallen \cite{MvdK} have defined a notion of unitary stable rank for a tuple $(R, \epsilon, \Lambda)$, which we shorten to $usr(R)$. In Section \ref{sec:Unitary} we shall show that for $n \geq usr(R)+1$ the commutator subgroup $U^\epsilon_n(R,\Lambda)'$ agrees with the elementary subgroup $EU^\epsilon_n(R,\Lambda)$. 

\begin{Th}\label{thm:H}
Let $F : U(\sqcup_n U^\epsilon_n(R,\Lambda)) \to \Z\operatorname{-Mod}$ be a coefficient system of degree $r$ at $N$. Then for all $n>N$, $$H_i(U^\epsilon_n(R,\Lambda);F_n)\lra H_i(U^\epsilon_{n+1}(R,\Lambda);F_{n+1})$$
is an epimorphism for $i\le \frac{n-usr(R)-1}{2}-r$ and an isomorphism for $i\le \frac{n-usr(R)-3}{2}-r$.

If $F$ is a coefficient system as above then for all $n > 2N$ the map
$$H_i(EU^\epsilon_n(R,\Lambda);F_n)\lra H_i(EU^\epsilon_{n+1}(R,\Lambda);F_{n+1})$$
is an epimorphism for $i\le \frac{n-usr(R)-2}{3}-r$ and an isomorphism for $i\le \frac{n-usr(R)-5}{3}-r$.
\end{Th}

Previously known results were: Vogtmann \cite{Vog79} ($O_{n,n}$ for fields except $\bF_2$), Betley \cite{Betley} ($O_{n,n}$ for semi-local rings), Charney \cite{Cha87} ($Sp_{2n}$ and $O_{n,n}$ for Dedekind domains) and Mirzaii--van der Kallen \cite{MvdK} (for constant coefficients). In particular, the results for twisted coefficients and for  $EU^\epsilon_{n}(R,\Lambda)$ seem to be new; the case of $EU^\epsilon_{n}(R,\Lambda)$ and constant coefficients implies that the unstable unitary $K$-groups $\pi_i(BU_n^\epsilon(R,\Lambda)^+)$ stabilise.

\addtocontents{toc}{\SkipTocEntry}
\subsection*{Non-orientable mapping class groups}

Let $F_{n,b}$ denote a non-orientable surface of genus $n \geq 0$ with $b \geq 1$ boundary components, and write $\pi_0\Dif(F_{n,b}\, \textrm{rel}\ \del_0 F_{n,b})$ for the group of isotopy classes of diffeomorphisms of $F_{n,b}$ which fix a boundary component $\del_0 F_{n,b}$ of $F_{n,b}$, the mapping class group of $(F_{n,b}, \partial_0 F_{n,b})$. 
For $n \geq 7$ in the case $b=1$, the commutator subgroup $\pi_0\Dif(F_{n,1}\, \textrm{rel}\ \del_0 F_{n,1})'$ coincides with the subgroup $\mathcal{T}_{n,1}$ generated by Dehn twists \cite{Stu09}. 

\begin{Th}\label{thm:I}
Let $F:U(\sqcup_{n,b} \pi_0\Dif(F_{n,b}\, \textrm{rel}\ \del_0 F_{n,b}))\to \Z\operatorname{-Mod}$ be a coefficient system of degree $r$ at $N$. Then the map 
$$H_i(\pi_0\Dif(F_{n,b}\, \textrm{rel}\ \del_0 F_{n,b});F(F_{n,b}))\rar  H_i(\pi_0\Dif(F_{n+1,b}\, \textrm{rel}\ \del_0 F_{n,b});F(F_{n+1,b}))$$
is an epimorphism for $i\le \frac{n-3}{3}-r$ and an isomorphism for $i\le \frac{n-6}{3}-r$.

If $F$ is a coefficient system as above then for $n \geq 7$ the map 
$$H_i( \mathcal{T}_{n,b};F(F_{n,1}))\rar  H_i(\mathcal{T}_{n+1,1};F(F_{n+1,1}))$$
is an epimorphism for $i\le \frac{n-3}{3}-r$ and an isomorphism for $i\le \frac{n-6}{3}-r$.
\end{Th}

For constant coefficients, homological stability for the groups $\pi_0\Dif(F_{n,b}\, \textrm{rel}\ \del_0 F_{n,b})$ was proved by the second author in \cite{Wah08} and with an improved range by the first author in \cite{RW09}, but for twisted coefficients or for the subgroups $\mathcal{T}_{n,b}$ it is new.

\addtocontents{toc}{\SkipTocEntry}
\subsection*{Mapping class groups of  3-manifolds} 

Let $M$ and $N$ be compact, connected, oriented 3-manifolds with boundary, and let $D$ and $D'$ be chosen discs in boundary components $\del_0M$ and $\del_0N$ of $M$ and $N$. Denote by $M\natural N$ the {\em boundary connected sum} of $M$ and $N$, i.e.~the manifold obtained from $M$ and $N$ by identifying $D$ with $D'$. 
(Note that if $N=\overline N\minus D^3$ for some 3-manifold $\overline N$ and $\del_0N=\del D^3$, then $M\natural N$ is the usual connected sum of $M$ and $\overline N$.) 
Write $\pi_0\Dif(M \natural N^{\natural n}\ \textrm{rel}\ D)$ for the mapping class group of $M\natural N^{\natural n}$, the group of components of the diffeomorphisms of $M\natural N^{\natural n}$ which restrict to the identity on $D$. 

\begin{Th}\label{thm:J}
Let $(M,D)$ and $(N,D')$ be compact, connected, oriented 3-manifolds equipped with discs in their boundaries as above. If $\del_0N\not \cong S^2$, assume that $M$ is irreducible. Let 
$$G_n = \pi_0\Dif(M \natural N^{\natural n} \ \textrm{rel}\ \partial_0 M)$$
and $G_n'$ be its commutator subgroup. Let $\M_3^+$ denote the groupoid of pairs $(M,D)$ and isotopy classes of diffeomorphisms fixed on the discs. 
Let $F:U(\M_3^+)\to \Z\operatorname{-Mod}$ be a coefficient system of degree $r$ at $N$. Then for all $n>N$, $H_i(G_n;F_n)\to H_i(G_{n+1};F_{n+1})$
is an epimorphism for $i \leq \tfrac{n-1}{2}-r$ and an isomorphism for $i \leq \tfrac{n-3}{2}-r$, and for all $n>2N$, 
 $H_i(G'_n;F_n)\to H_i(G'_{n+1};F_{n+1})$
an epimorphism for $i \leq \tfrac{n-2}{3}-r$ and an isomorphism for $i \leq \tfrac{n-5}{3}-r$. 
\end{Th}

For $G_n$ with constant coefficients, this was first proved in \cite{HatWah10}. An example of a group $G_n$ in the theorem is for $M=D^3$ and $N=S^1\x D^2$, where $G_n$ is the handlebody mapping class group of a surface of genus $n$.

\addtocontents{toc}{\SkipTocEntry}
\subsection*{Outline}

The paper is organised as follows. Section~\ref{homcatsec} introduces (locally) homogeneous categories and pre-braided categories, and gives some of their basic properties; it also explains how to obtain a homogeneous category from a braided monoidal groupoid, by a construction of Quillen (see Theorem~\ref{universal}). In Section~\ref{simpsec}, we define semi-simplicial sets $W_n(A,X)_\bullet$ associated to a pair of objects in a monoidal category having an initial unit. We then relate these to certain simplicial complexes $S_n(A,X)$, which are easier to manipulate: these are not necessary for the proofs of our main theorems, but are useful in establishing examples. In Section~\ref{cstsec}, we prove a preliminary homological stability theorem for families of automorphism groups in a homogeneous category with constant coefficients (Theorem~\ref{stabthm}) and with abelian coefficients (Theorem~\ref{abstabthm}). These are not necessary for the proof of our general homological stability theorem, but the argument we give yields a slightly better stability range. In Section \ref{sec:twist} we first introduce (internalised) coefficient systems and their basic properties, and then formulate and prove our general homological stability theorem for families of automorphism groups in a homogeneous category, which implies Theorem \ref{main} and Corollary \ref{cor:B}. Finally, Section~\ref{examples} gives examples of homogeneous categories and resulting stability theorems, and in particular establishes Theorems~\ref{thm:D}--\ref{thm:J}.

\addtocontents{toc}{\SkipTocEntry}
\subsection*{Acknowledgements}
NW: This paper was inspired by a workshop in Copenhagen on homological stability in August 2013, and I benefited from conversations with many people at this occasion, as well as subsequently. I would in particular like to thank Aur\'elien Djament and Christine Vespa for numerous helpful discussions on polynomial functors,  Julie Bergner and Emily Riehl for categorical help,  and Giovanni Gandini, Manuel Krannich, Martin Palmer, Peter Patzt, Christian Schlichtkrull and Arthur Souli\'e for helpful comments, suggestions and corrections.  Much of this paper was written during research stays at MSRI, MPI and the HIM, and I would like to thank these institutes for providing such wonderful working conditions. I was supported by the Danish National Sciences Research Council (DNSRC) and the European Research Council (ERC), as well as by the Danish National Research Foundation through the Centre for Symmetry and Deformation (DNRF92). 

ORW: This paper is a revision of one originally written by NW, in which many of the main ideas already appeared. NW suggested that I write an appendix on abelian coefficient systems, but it fit so well with her framework that we were led to push it as far as possible: this revision is the result. I am grateful to her for suggesting that we pursue this project together. I am also grateful to Nina Friedrich for her useful comments on a draft of this paper.

\tableofcontents

\section{Homogeneous and locally homogeneous categories}\label{homcatsec}

In this section we define a categorical framework based on one used by Djament--Vespa in the context of functor homology, and presented by Christine Vespa in her course in Copenhagen \cite{DjaVes10,Ves13}. We introduce the notions of (locally) homogeneous categories, and of pre-braided monoidal categories. The main result of the section, Theorem~\ref{universal}, shows that Quillen's bracket construction associates a pre-braided monoidal homogeneous category to any braided monoidal groupoid satisfying some mild conditions. 

\bigskip

Let $(\C,\oplus,0)$ be a (strict) monoidal category in which the unit 0 is initial. For a morphism $g : X \to Y$ in $\C$ we will write
$$\Fix(g):=\{f\in \Aut(Y)\ |\ f\circ g=g\}.$$
In particular, for every pair of objects $A$ and $B$ in such a category we have a preferred morphism 
$$\iota_A\op B:B= 0\op B \lra A\op B$$
where $\iota_A:0\to A$ denotes the unique morphism in $\C$ from the initial object 0 to $A$, and we have hence defined $\Fix(\iota_A\op B)$.

\begin{rem}
Here, and in the rest of the paper, we adopt the convention of identifying objects of $\mathcal{C}$ with their identity morphisms.
\end{rem}

\begin{Def}\label{lochomdef}
A monoidal category $(\C,\op,0)$ is {\em locally homogeneous} at a pair of objects $(A,X)$ if 0 is initial in $\C$ and if it satisfies the following two axioms: 
\begin{itemize}
\item[{\bf LH1}] For all $0 \leq p < n$, $\Hom(X^{\op p+1},A \op X^{\op n})$ is a transitive $\Aut(A \op X^{\op n})$-set under postcomposition. 
\item[{\bf LH2}] For all $0 \leq p < n$, the map $\Aut(A \op X^{\op n-p-1})\to \Aut(A \op X^{\op n})$ taking $f$ to $f\op X^{\op p+1}$ is injective with image $\Fix(\iota_{A \op X^{\op n-p-1}} \op X^{\op p+1})$.
\end{itemize}
\end{Def}

In many cases---but not all---these axioms will hold for all objects $A$ and $X$ of $\mathcal{C}$ because the category satisfies the following stronger, global, version of LH1 and LH2.

\begin{Def}\label{homdef}
A monoidal category $(\C,\op,0)$ is {\em homogeneous} if 0 is initial in $\C$ and if it satisfies the following two axioms for all objects $A, B \in \C$: 
\begin{itemize}
\item[{\bf H1}] $\Hom(A,B)$ is a transitive $\Aut(B)$-set under postcomposition. 
\item[{\bf H2}] The map $\Aut(A)\to \Aut(A\oplus B)$ taking $f$ to $f\op B$ is injective with image $\Fix(\iota_A \op B)$.
\end{itemize}
\end{Def}
These axioms are a generalisation of a simplification of the set-up of \cite[Sec.~1.2]{DjaVes10}, and can be found in that form in \cite[Sec.~3]{Ves13} and essentially in \cite[Sec.~2.4]{Dja12}. Djament--Vespa in these papers work only with symmetric monoidal categories, a condition we shall weaken to a notion of {\em pre-braiding} (see Definition~\ref{prebraidDef} below).

\medskip

The most basic example of a (non-trivial) homogeneous category is the category $FI$ of finite sets and injections, much studied by Church, Farb, Ellenberg, and others (see \cite{CEF12} for an introduction). The monoidal structure on $FI$ is the disjoint union of sets, with the empty set as unit. The automorphism group of a set is the group of permutations of its elements. It is easy to see that $\Aut(B)$ acts transitively on $\Hom(A,B)=\operatorname{Inj}(A,B)$. Likewise,
$\Aut(A)\to \Aut(A\sqcup B)$ is injective with image the permutations that fix the canonical injection $B\inc A\sqcup B$. 

A more ``typical'' example of a homogeneous category is obtained by fixing a field $k$ and considering the category $\mathcal{V}_k$ whose objects are finite dimensional $k$-vector spaces, and whose morphisms 
from $V$ to $W$ consist of a pair $(f,H)$ where $f:V\to W$ is an injective homomorphism and $H\le W$ is a subspace such that $W=H\op f(V)$. The monoidal structure is direct sum and the unit is the 0 vector space. The automorphism group of $W$ is its associated general linear group $GL(W)$. The fact that $GL(W)$ acts transitively on $\Hom(V,W)$ uses a cancellation property, namely the fact that complements of isomorphic subspaces of $W$ are isomorphic. We have that $GL(V)\to GL(V\op W)$ is injective with image $\Fix(\iota_V\op W)$ because the morphism $\iota_V\op W=(i_W,V):W\to V\op W$ is the pair of the canonical inclusion $i_W$ of $W$ in $V\op W$ with the canonical copy of $V$ 
as chosen complement, so that an automorphism of $V\op W$ fixing this map needs to fix $V$ setwise and $W$ pointwise, and thus can only come from an automorphism of $V$. 

The {\em complemented categories with generator}   of Putman--Sam in \cite{PutSam14} are particular examples of homogeneous categories (see Lemmas 3.1 and 3.2 in that paper). Homogeneous categories are also closely related to the categories satisfying {\em transitivity}, which is our H1 above, and {\em bijectivity} of Gan and Li \cite{GanLi14}. 

As we will see, there are many examples of homogeneous categories. We give below a way to construct a homogeneous category from certain braided monoidal groupoids, the above two examples being constructible this way by starting from the groupoid which is the union of the symmetric groups (see Section~\ref{setex}), and the groupoid which is the union of the general linear groups $GL(V)$ respectively (see Section~\ref{GLnsec}). 

\begin{rem}
For a homogeneous category $(\C,\op ,0)$,  conditions H1 and H2 together imply that $$\Hom(B,A\op B)\cong \Aut(A\op B)/\Aut(A)$$
with the subgroup $\Aut(A)\inc \Aut(A\op B)$ acting by precomposition.  
In particular, given that the unit 0 is initial and hence has no non-trivial automorphisms, we must have 
$$\operatorname{End}(A):=\Hom(A,A)=\Aut(A)$$ for each object $A$, i.e.~endomorphisms are isomorphisms. So homogeneous categories are in particular
EI-categories, as studied in \cite{Luc89}. 
This implies in particular that 
$$\Hom(A,B)\neq\emp \ \ \textrm{and} \ \ \Hom(B,A)\neq \emp \ \ \ \ \Longrightarrow \ \ \ A\cong B$$ 
so that isomorphism classes of objects naturally form a poset with $[A]\le [B]$ if and only if $\Hom(A,B)\neq \emp$. 
\end{rem}

\bigskip

Let $(\mathcal{C}, \op, 0)$ be a monoidal category which is locally homogeneous at $(A,X)$, and consider the groups 
 $$G_n:= \Aut(A\op X^{\oplus n}).$$ 
There is a canonical map $\Si^X:G_n\to G_{n+1}$ induced by writing $A\op X^{\oplus n+1}=(A\op X^{\oplus n})\oplus X$ and taking $f\in G_n$ to $f\oplus X\in G_{n+1}$. By LH2 the map $\Si^X$ is injective, so we obtain a sequence of groups and preferred maps between them 
$$G_1\inc G_2\inc \cdots \inc G_n\inc \cdots.$$ 
The category $\mathcal{C}$ however encodes more information about how these groups are related.

We shall study homological stability for sequences of groups arising as above. When we study stability with twisted coefficients, we shall consider a certain ``lower stabilisation map" $\Si_X : G_n \to G_{n+1}$ in addition to the
``upper stabilisation map" $\Si^X:G_n\to G_{n+1}$ described above (see
Section~\ref{sec:LowerSusp}). In the case $A=0$ this is simply given by $X \oplus -$, but for $A \neq 0$ a more complicated formula is needed, which makes use of the following definition.

\begin{Def}\label{prebraidDef}
Let $(\C,\op,0)$ be a monoidal category with 0 initial. 
We say that $\C$ is  {\em pre-braided} if its underlying groupoid is braided and for each pair of objects $A$ and $B$ in $\C$, the groupoid braiding $b_{A,B}:A\op B\to B\op A$ satisfies 
$$b_{A,B}\circ (A\op \iota_B)=\iota_B\op A:A\lra B\op A.$$
\end{Def}

Note that braided monoidal categories are examples of pre-braided categories. All the examples of categories $\C$ in which we shall study homological stability will be pre-braided, and in fact will be built from an underlying braided monoidal groupoid. Note however that the category $\C$ itself need {\em not} in general be a braided monoidal category (see Remark~\ref{prebraidrem} for an example).

In Section~\ref{sec:twist}, we will use that the following symmetric version of H2 also holds in a pre-braided homogeneous category.

\begin{prop}\label{H2sym}
Let $\C$ be a pre-braided homogeneous category. Then 
the map 
$$\Aut(A\op C)\lra \Aut(A\oplus B\op C)$$
taking $f$ to $c_{A\op b^{-1}_{B,C}}(f\op B)$ has image $\Fix(\iota_A\op B\op \iota_C)$. 
In particular, the map 
$$\Aut(C)\lra \Aut(B\oplus C)$$
taking $f$ to $B\op f$ has image $\Fix(B\op \iota_C)$. 
\end{prop}

\begin{proof} The second statement follows from the first and the braided monoidal axioms of the underlying braided groupoid, so we are left to prove the first statement. By axiom H2 the map $f \mapsto f \op B : \Aut(A\op C) \to \Aut(A\op C \op B)$ is injective with image $\Fix(\iota_{A \op C} \op B)$, so it remains to show that $c_{A\op b^{-1}_{B,C}} : \Aut(A\op C \op B) \overset{\sim}\to \Aut(A\op B \op C)$ takes $\Fix(\iota_{A \op C} \op B)$ isomorphically to $\Fix(\iota_A\op B\op \iota_C)$. For $f \in \Fix(\iota_{A \op C} \op B)$ we have
\begin{align*}
(A\op b_{B,C}^{-1})\circ f&\circ (A\op b_{B,C})\circ (\iota_A\op B\op \iota_C)\\
&= (A\op b_{B,C}^{-1})\circ f\circ (\iota_A\op \iota_C\op B)\\
& =  (A\op b_{B,C}^{-1})\circ  (\iota_A\op \iota_C\op B)\\
& =  (\iota_A\op B\op \iota_C),
\end{align*}
where we used the axioms of  the pre-braided axiom for the first equality, the definition of $\Fix(\iota_A\op B\op \iota_C)$ for the second, and finally the pre-braided axiom again for the last equality, so $c_{A\op b^{-1}_{B,C}}(f) \in \Fix(\iota_A\op B\op \iota_C)$. Similarly, for $g \in \Fix(\iota_A\op B\op \iota_C)$ the same reasoning gives
\begin{align*}
(A\op b_{B,C})\circ g&\circ (A\op b_{B,C}^{-1})\circ (\iota_A\op \iota_C\op B)\\
&= (A\op b_{B,C}^{-1})\circ g\circ (\iota_A\op B\op \iota_C)\\
& =  (A\op b_{B,C}^{-1})\circ  (\iota_A\op B\op \iota_C)\\
& =  (\iota_A\op \iota_C\op B).\qedhere
\end{align*}
\end{proof}

\subsection{(Locally) homogeneous categories from groupoids}

Let $(\G,\op,0)$ be a monoidal groupoid. We recall from \cite[p.~219]{Gra76} Quillen's construction of a category $\langle\G,\G\rangle$, denoted here $U\G$ for brevity. (Quillen considers the more general case of a monoidal category $S$ acting on a category $X$: here we take $S=\G=X$.)  The category $U\G$ has the same objects as $\G$, and a morphism in $U\G$ from $A$ to $B$ is an equivalence class of pairs $(X,f)$ where  $X$ is an
object of $\G$ and $f:X\op A\to B$ is a morphism in $\G$, and where $(X,f)\sim (X',f')$ if there exists an isomorphism $g:X\to X'$ in $\G$ making
the diagram 
$$\xymatrix{X\op A \ar[d]_{g\op A}\ar[r]^f & B\\
X'\op A\ar[ur]_{f'} &
}$$
commute. We write $[X,f]$ for such an equivalence class. If $[X,f]:A\to B$ and $[Y,g]:B\to C$ in $U\G$, their composition is defined as 
$[Y,g]\circ [X,f]=[Y\op X,g\circ (Y\op f)]$. 

An alternative description of the morphisms in this category, used in \cite[Sec.~3]{DjaVes13}, is $$\Hom_{U\G}(A,B)=\colim_{\G}\Hom_{\G}(- \, \op A,B).$$
From this description in particular, one can see that there is a well-defined map from the set of morphisms in $U\G$ to the set of isomorphism classes of objects in $\G$, which associates to a morphism $[X,f]$ its {\em complement} $[X]$.

When applying Quillen's construction, we will be interested in the relationship between the automorphism groups in $\G$ we start with, and those in $U\G$. 
We have a functor 
$$I:\G\lra \mathrm{Iso}(U\G)$$ 
taking an isomorphism $f:A\to B$ to the pair $[0,f]$. 
The following result gives a condition on $\G$ under which the functor $I$ is faithful (resp.\ full). 

\begin{prop}\label{underlying}
Let $(\G,\op,0)$ be a monoidal groupoid and  $U\G :=\lgl \G,\G\rgl$. Then the following holds: 
\begin{enumerate}[(i)]
\item If $\Aut_{\G}(0)=\{id\}$, the functor $I:\G\to \mathrm{Iso}(U\G)$ is faithful;
\item If $\G$ has no {\em zero divisors}, i.e.\ if $U\op V\cong 0$ in $\G$ implies $U,V\cong 0$, then  the functor $I:\G\to \mathrm{Iso}(U\G)$ is full. 
\end{enumerate}
In particular, $\G$ is the underlying groupoid of $U\G$ when both conditions hold. 
\end{prop}

\begin{proof} 
Suppose first that  $I(f):=[0,f]=[0,g]=:I(g)$ in $\Hom_{U\G}(A,B)$. This means that there is an isomorphism $\phi:0\to 0$ in $\G$ such that $g=f\circ (\phi\op A)$. If $\Aut_\G(0)=\{id\}$, we must have that $\phi=id$ and hence that $f=g$, which proves (i). 

For (ii), suppose that  $[X,f] \in \Hom_{U\G}(A,B)$ is an isomorphism.  Then $[X,f]$ has an inverse $[Y,g] \in \Hom_{U\G}(B,A)$. This means that
$[Y,g]\circ [X,f]=[Y\op X,g\circ (Y\op f)]=[0,id]$. So in particular we must have $ Y \op X\cong 0$ in $\G$. If $\G$ has no zero divisors, this implies that  $X\cong 0\cong Y$ in $\G$. Choosing an isomorphism $\phi:0\to X$, we get $[X,f]=[0,f\circ (\phi \op A)]$, which shows that $[X,f]$ is in the image of $\G$ in $U\G$, proving (ii). 
\end{proof}

As already noted in \cite[Ex.~5.11]{SchSag12}, Quillen's construction applied to the groupoid of finitely generated free modules over a ring $R$ yields the category with objects the finitely generated free $R$-modules and morphisms the ``free split injections", that is split injections $f : M \to N$ equipped with a choice of free submodule $F \leq N$ such that $N = \mathrm{Im}(f) \oplus F$. This is an example of a homogeneous category, already mentioned after Definition~\ref{homdef}  in the case where $R$ is a field. We will show in the present section that this construction often yields homogeneous categories. This will be done by analysing how mild assumptions on the groupoid $\G$ yield the properties (L)H1, (L)H2 and the pre-braidedness. We start with the latter. 

\begin{prop}\label{braidandsym} Let $(\G,\op,0)$ be a monoidal groupoid, and $U\G :=\lgl \G,\G\rgl$. Then: 
\begin{enumerate}[(i)]
\item  $0$ is initial in $U\G$;
\item  if $\mathcal{G}$ is braided monoidal  with no zero divisors then $U\G$ is a pre-braided monoidal category;
\item  if $\mathcal{G}$ is symmetric monoidal then $U\G$ is a symmetric monoidal category. 
\end{enumerate}
Moreover, in the latter two cases, the monoidal structure of $U\G$ is such that the map $\G\to U\G$ taking an isomorphism $f$ to $[0,f]$ is monoidal.
\end{prop}

Note that in the case when $\G$ is symmetric monoidal the construction can be repeated on $U\G^{op}$ to create a symmetric monoidal category in which $0$ is now a null object. This is what the construction is used for in \cite[Sec.~3]{DjaVes13}. 

\begin{proof}
We first check that the unit 0 is initial in $U\G$. Indeed, if $[X,f]$ and $[Y,g]$ are two elements of $\Hom_{U\G}(0,A)$, then $g^{-1}\circ f :X\to Y$ is an isomorphism exhibiting that the two morphisms actually represent the same element of $\Hom_{U\G}(0,A)$. 
Note now that there always exists a morphism from 0 to $A$ in $U\G$, namely the morphism $\iota_A=[A,id_A]$. This proves (i). 

\medskip

Assuming now that $\G$ is braided monoidal, we define the monoidal structure on $U\G$ as follows: given $[X,f]\in \Hom(A,B)$ and $[Y,g]\in\Hom(C,D)$, we let 
$$[X,f]\op [Y,g]:= [X\op Y,(f\op g)\circ (X\op b_{A,Y}^{-1}\op C)] \in \Hom(A\op C,B\op D).$$
This is compatible with the monoidal structure of $\G$ because of the compatibility between the braiding and the unit in a braided monoidal category. The associativity of this monoidal product follows from that of the monoidal product in $\G$ and the braid relations. 

The functor $I : \G \to Iso(U\G)$ is a bijection on objects and, given that $\G$ has no zero divisors, Proposition~\ref{underlying} shows that it is also full. The braided monoidal structure on $\G$ therefore induces one on $Iso(U\G)$. The fact that $U\G$ is pre-braided follows from the computation
\begin{align*}
b_{A,B}\circ (A\op \iota_B)&= [0,b_{A,B}]\circ \big([0,id_A]\op [B,id_B]\big)\\
& = [0,b_{A,B}]\circ [B,id_{A\op B}\circ b_{A,B}^{-1}] = [B,b_{A,B}\circ b_{A,B}^{-1}]=[B,id_{A\op B}]\\
\iota_B\op A& = [B,id_B]\op [0,id_A]=[B,id_{A\op B}\circ b_{B,0}^{-1}]=[B,id_{A\op B}]
\end{align*} 
which proves (ii). 

\medskip

To prove (iii), we need to check that for any  $[X,f]\in \Hom_{U\G}(A,B)$ and  $[Y,g]\in \Hom_{U\G}(C,D)$, we have 
$$([Y,g]\op [X,f])\circ [0,b_{A,C}]=[0,b_{B,D}]\circ ([X,f]\op [Y,g]).$$
Explicitly, the left-hand side is 
$$[Y\op X,(g\op f)\circ (Y\op b_{C,X}^{-1}\op A)\circ (Y\op X\op b_{A,C})]$$
and the right-hand side is 
$$[X\op Y,b_{B,D}\circ (f\op g)\circ (X\op b_{A,Y}^{-1}\op C)].$$
Now $b_{X,Y}:X\op Y\to Y\op X$ defines an isomorphism between the complements of these two morphisms. The fact that they represent the same morphism corresponds to the commutativity of the following diagram: 
$$\xymatrix{X\op Y\op A\op C \ar[rr]^-{b_{X,Y}\op A\op C}\ar[d]_{X\op b^{-1}_{A,Y}\op C} && Y\op X\op A\op C \ar[d]^{Y\op ((b_{C,X}^{-1}\op A)\circ (X\op b_{A,C}))} \\
X\op A\op Y\op C \ar[d]_{f\op g}\ar[rr]^{b_{X\op A,Y\op C}} && Y\op C\op X\op A \ar[d]^{g\op f}\\
B\op D \ar[rr]^-{b_{B,D}}  && D\op B.
}$$
The bottom square commutes because $b$ is a braiding in $\G$. One then checks that the top square commutes under the additional assumption that $b$ is a symmetry, i.e. that $b_{A,B}^{-1}=b_{B,A}$. (This last commutation fails if $b$ is only a braiding.) 
\end{proof}

Property (L)H1 on $U\G$ will require a cancellation property on $\G$.

\begin{Def}\label{LCDef}
For a pair of objects $(A,X)$ in a monoidal groupoid $(\G,\op,0)$ we say that $\G$ satisfies \emph{local cancellation at $(A,X)$} if it satisfies
\begin{itemize}
\item[{\bf LC}] For all $0 \leq p < n$, if $Y \in \G$ is such that $Y \op X^{\op p+1} \cong A \op X^{\op n}$ then $Y \cong A \op X^{\op n-p-1}$.
\end{itemize}
We say that $\G$ satisfies \emph{cancellation} if it satisfies
\begin{itemize}
\item[{\bf C}] For all objects $A,B,C \in \G$, if $A\op C\cong B\op C$ then $A\cong B$.
\end{itemize}
\end{Def}
The axiom LC implies in particular that morphisms $X^{\op p+1} \to A \op X^{\op n}$ in $\lgl \G,\G\rgl$ have isomorphic complements, and this is one of the main ways in which we shall use it.

\medskip

Our main result in this section is the following: 

\begin{thm}\label{universal}\mbox{}  Let $(\G,\op,0)$ be a braided monoidal groupoid with no zero divisors, and  $U\G :=\lgl \G,\G\rgl$ its associated pre-braided category. 
\begin{enumerate}[(a)]
\item $U\G$ satisfies LH1 at $(A,X)$ if and only if $\G$ satisfies   LC at $(A,X)$.
\item If the map $\Aut_{\G}(A \op X^{\op n-p-1})\to \Aut_{\G}(A \op X^{\op n})$ taking $f$ to $f\op X^{\op p+1}$ is injective  for all $0 \leq p < n$, then  $U\G$ satisfies LH2 at $(A,X)$.
\end{enumerate}
In particular, if (a) and (b) are both satisfied, then $U\G$ is locally homogeneous at $(A,X)$.
\begin{enumerate}[(a)]
\setcounter{enumi}{2}
\item $U\G$ satisfies H1 if and only if $\G$ satisfies  C.
\item If the map $\Aut_{\G}(A)\to\Aut_{\G}(A\op B)$ taking $f$ to $f\oplus B$ is injective for every $A,B$ in $\G$, then $U\G$ satisfies H2.
\end{enumerate}
In particular, if (c) and (d) are both satisfied, then $U\G$ is homogeneous.
\end{thm}

\begin{rem}
If cancellation does not hold  in a (small) monoidal groupoid $\G$, it can be forced by replacing $\G$ with the groupoid $\tilde \G$ whose set of objects is the free monoid on the objects of $\G$, and with automorphisms of these objects as only morphisms---that is we free up the monoidal structure and forget in $\tilde \G$ that certain sums of objects were isomorphic in $\G$. This construction may seem rather unnatural, but it can be relevant in our context. Indeed, it does not change the automorphism groups of objects, which are the groups we are interested in, and only affects which morphisms we want to allow between different objects. This method is applied in \cite{SzyWah} to prove homological stability for the Higman--Thompson groups. 
\end{rem}

\begin{proof}
We first prove (a) and (c). 
Suppose that $\G$ satisfies LC at $(A,X)$ and let $[U,f],[V,g]\in\Hom_{U\G}(X^{\op p+1},A \op X^{\op n})$. Then
$$f: U \op X^{\op p+1} \lra A \op X^{\op n} \text{ and } g: V \op X^{\op p+1} \lra A \op X^{\op n}$$
are isomorphisms, so by two applications of LC we have that $U$ and $V$ are isomorphic: choose a $\phi : U \overset{\sim}\to V$. Then $[V, g] = [U, g \circ (\phi \op X^{\op p+1})]$, but this clearly differs from $[U, f]$ by postcomposition by $[0, f \circ (g \circ (\phi \op X^{\op p+1}))^{-1}]$, which proves LH1.

Likewise, if $\G$ satisfies C and $[U,f],[V,g]\in\Hom_{U\G}(A,B)$ are two morphisms, we have $U\op A\cong B\cong V\op A$ giving $U\cong V$ and the same proof shows that $U\G$ satisfies H1. 

\medskip

Conversely, assume that $U\G$ satisfies LH1 and  that we have an isomorphism $\phi:Y\op X^{\op p+1}\sta{\cong}\rar A\op X^{\op n}$. We can consider $[Y,\phi]$ as an element of $\Hom(X^{\op p+1}, A\op X^{\op n})$. Now $[A\op X^{\op n-p-1},id]$ is also an element in that set. By LH1, there exists an automorphism $[U,\psi]$ of $A\op X^{\op n}$ such that $[U,\psi] \circ [Y,\phi]=[A\op X^{\op n-p-1},id]$. 
Given that $\G$ has no zero divisors, $[U,\psi]=[0,\psi']$ for some $\psi'\in \Aut_\G(A\op X^{\op n})$. It follows that $[Y,\psi'\circ \phi]=[A\op X^{\op n-p-1},id]$. 
But this can only happen in $U\G$ if $Y\cong A\op X^{\op n-p-1}$, proving LC, and finishing the proof of (a). 

Again, if we now assume that $U\G$ satisfies H1, and $Y\op A\cong B$ is an isomorphism, the same proof, using now H1 on $\Hom(A,B)$, shows that $\G$ satisfies C, which finishes the proof of (c). 

\medskip

For (b), we must show that the map
$$-\op X^{\op p+1}:\Aut_{U\G}(A\op X^{\op n-p-1})\lra \Aut_{U\G}(A\op X^{\op n})$$
is injective and identify its image. To check injectivity, suppose that $[V,f]\in\Aut_{U\G}(A\op X^{\op n-p-1})$ is such that $[V,f]\op X^{\op p+1}$ is the identity in $\Aut_{U\G}(A\op X^{\op n})$. This means that there exists an isomorphism $\phi:V\to 0$ in $\G$  such that
$$f\op X^{\op p+1}=\phi\op A\op X^{\op n} \in \Hom_\G(V\op A\op X^{\op n},A\op X^{\op n}).$$
Using the assumption, we have that $f=\phi\op A\op X^{\op n-p-1}$. But this means that $[V,f]$ is the identity in $\Aut_{U\G}(A\op X^{\op n-p-1})$, as required. We are left to show that the image of the map  $-\op X^{\op p+1}$ is $\Fix(\iota_{A \op X^{\op n-p-1}} \op X^{\op p+1})$. Letting $U = A \op X^{\op n-p-1}$ and $V=X^{\op p+1}$, we compute (using the no zero divisor assumption for simplicity)
\begin{align*}
\Fix(\iota_U \op V)&=\{[0,\phi]\in\Aut_{U\G}(U \op V)\ |\ [0,\phi]\circ (\iota_U\op V)=\iota_U\op V\}\\
& = \{[0,\phi]\in\Aut_{U\G}(U \op V)\ |\ [U,\phi]=[U,id_{U\op V}]\}.
\end{align*}
The last equality is equivalent to saying that $\Fix(\iota_U \op V)$ consists of the morphisms $[0,\phi]$ such that there is an isomorphism $\psi:U\to U$ in $\G$ satisfying that $\phi=\psi\op V$, which is exactly saying that $[0,\phi]$ is the image of 
$[0,\psi]\in \Aut_{U\G}(U)$.

For (d), the proof is identical. 
\end{proof}

There are many examples of braided monoidal groupoids satisfying the hypotheses of Theorem \ref{universal}. In Section~\ref{examples}, we will study the homogeneous categories associated to braided monoidal groupoids defined from symmetric groups, braid groups, automorphism groups of free groups, general linear groups, and mapping class groups. 

\subsection{A graphical calculus for pre-braided categories}\label{prebraidcal}

Let $\beta := \sqcup_{n \geq 0} \beta_n$ be the braided monoidal groupoid with objects the natural numbers, and with the automorphisms of $n$ given by the braid group on $n$ strands. By Proposition \ref{braidandsym}, $\U\beta$ is a pre-braided monoidal category. Just as $\beta$ is the free braided monoidal category on one object, one has that $\U\beta$ is the free pre-braided monoidal category on one object. (See \cite[Prop.~2.2]{JoyStr93} for the braided case.) We give here a graphical calculus that may be useful to the reader in following the subsequent proofs, though it is not mathematically necessary. We give an indication of why it holds. 

It follows from our construction of $\U\beta$ that $\mathrm{Hom}_{\U\beta}(n, m+n) = \beta_{m+n}/\beta_m$, where $\beta_m \subset \beta_{m+n}$ is considered as the subgroup of braids on the first $m$ strands and the quotient is by pre-composition. This set can alternatively be described as follows: morphisms from $n$ to $m+n$ are given by diagrams of braids with $m+n$ strands, $m$ of which have ``free'' ends, which are allowed to pass under, but not over, any other strand (Fig.~\ref{prebfig}). 
\begin{figure}[h]
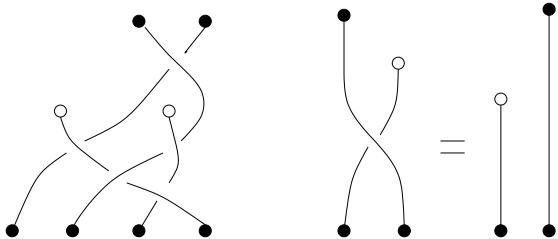

\begin{lpic}{prebraid.2(0.5,0.5)}
\end{lpic}
\caption{Morphism from 2 to 4 and the pre-braid relation in $\U\beta$.}\label{prebfig}
\end{figure}
Indeed, given a coset $\beta_m x \in \beta_{m+n}/\beta_m$, we can represent it by the braid $x$ on $m+n$ strands, the first $m$ of which start at a ``free'' end point. As free ends are allowed to pass under strands, we see that the diagram so obtained is independent of the coset representative $x$. Conversely, given a braid diagram with free ends, we can pass all the free ends leftwards under all the other strands, obtaining an element of $\beta_{m+n}$.  This is well-defined up to precomposition with elements of the subgroup $\beta_m$.

\section{The space of destabilisations associated to a pair $(A,X)$}\label{simpsec}

In this section we shall associate a sequence of semi-simplicial sets to any pair of objects in a monoidal category. We will use these to define an additional {\em connectivity} property for pairs of objects in a monoidal category, which will play a central role in our homological stability theorems.

\begin{Def}\label{simpdef} Let $(\C,\op,0)$ be a monoidal category with 0 initial and $(A,X)$ a pair of objects in $\C$. 
Define $W_n(A,X)_\bullet$ to be the semi-simplicial set with set of $p$-simplices 
$$W_n(A,X)_p :=\ \Hom_\C(X^{\op p+1},A\op X^{\op n})$$ and with face map 
$$d_i\colon \Hom_\C(X^{\op p+1},A\op X^{\op n})\rar \Hom_\C(X^{\op p},A\op X^{\op n})$$ defined by precomposing with ${X^{\op i}}\op \iota_X\op {X^{\op p-i}}$.  

Postcomposition in $\C$ defines a simplicial action of the group $\Aut(A\op X^{\op n})$ on  $W_n(A,X)_\bullet$.
\end{Def}

\medskip

Recall that a non-empty space $Y$ is called $m$-connected if $\pi_i(Y)=0$ for all $i\le m$ and all basepoints; we adopt the convention that $m$ may be a rational number, for which this definition also makes sense. We say that a non-empty space is at least $(-1)$-connected, and that the empty space is $(-2)$-connected. For a pair of objects $(A,X)$ as above, we make one final connectivity axiom, which depends on a parameter $k \in \bN$. 

\begin{Def} Let $(\C,\op,0)$ be a monoidal category and $(A,X)$ a pair of objects in $\C$. We say that $\C$ satisfies {\em $LH3$ at $(A,X)$ with slope $k$} if 
\begin{itemize}
\item[{\bf LH3}] For all $n \geq 1$, $\vert W_n(A,X)_\bullet \vert$ is $(\tfrac{n-2}{k})$-connected.
\end{itemize}
\end{Def}
The following lemma shows that when $\C=U\G$, LH3 essentially implies LH1, which we saw in Theorem~\ref{universal} is equivalent to local cancellation.  

\begin{lem}
Let $(\G,\op,0)$ be a braided monoidal groupoid, and suppose that $U\G$ satisfies LH3 at $(A,X)$ with some slope $k \geq 1$. Then $U\G$ satisfies LH1 at $(A\op X,X)$.
\end{lem}
\begin{proof} 
By Theorem~\ref{universal} (a), it is enough to check that $\G$ satisfies LC at $(A\op X,X)$ (noting that the no zero divisors assumption is not used for that direction). So  
let $Y \in \G$ and suppose that $\phi : Y \op X^{\op p+1} \to (A \op X) \op X^{\op n}$ is an isomorphism in $\G$, with $0 \leq p < n$, where we may assume $n\ge 1$ (otherwise there is nothing to prove). If we can show that $Y \op X^{\op p} \cong (A \op X) \op X^{\op n-1}$ then, by induction on $p$, we are done.

If $[U, f]$ and $[V, g]$ are 0-simplices in $W_{n+1}(A,X)_\bullet$ which span a 1-simplex $[W, h]$, then
$$h : W \op X^{\op 2} \lra A \op X^{\op n+1}$$
is an isomorphism in $\G$, and it follows from the definition of precomposition with the morphisms $\iota_X \op X$ and $X \op \iota_X$ in $U\G$ that $U \cong W \op X \cong V$. Hence adjacent vertices in $W_{n+1}(A,X)_\bullet$ represent morphisms having isomorphic complements.

Under our assumption, the space $\vert W_{n+1}(A,X)_\bullet \vert$ is $0$-connected as $n+1\ge 2$. To $\phi$ we can associate the morphism $[Y \op X^{\op p}, \phi]$ in $U\G$, which represents a 0-simplex of $W_{n+1}(A,X)_\bullet$, as does $[(A \op X) \op X^{\op n-1}, id]$. As the space is connected there must be a sequence of 1-simplices interpolating between these two vertices, and so the complements they define, $Y \op X^{\op p}$ and $(A \op X) \op X^{\op n-1}$, must be isomorphic.
\end{proof}

Given a homogeneous category $\C$ and a pair of objects $(A,X)$ in $\C$, 
we shall obtain a homological stability theorem for the groups $G_n := \Aut(A \op X^{\op n})$ whenever we are able to show that $\C$ satisfies LH3 at $(A,X)$ for some slope $k \geq 2$. In proving stability theorems, it is usually in establishing LH3 that the difficulty arises, and to provide examples we must be able to prove that the $W_n(A,X)_\bullet$ are highly-connected. The remainder of this section develops tools to reduce this problem to proving high-connectivity of closely related simplicial complexes, which experience shows are often easier to work with. 

The reader may wish to ignore the rest of this section on a first reading: its results are not necessary to develop the general theory, and will only be used in Section \ref{examples} when we come to examples.

\subsection{Simplicial complexes versus semi-simplicial sets}\label{sec:SxCxSSets}

A simplicial complex is defined from a set of vertices by giving a collection of subsets of the vertices closed under taking subsets and containing all singletons. A $p$-simplex is then a subset of cardinality $(p+1)$ in the collection. So in a simplicial complex, $p$-simplices have an unordered set of $(p+1)$ distinct vertices which determine the simplex. 
In contrast, in a semi-simplicial set, $p$-simplices have an {\em ordered list} of $(p+1)$ {\em not necessarily distinct} vertices (obtained from applying the boundary maps repeatedly) and two distinct simplices may have the same set of vertices, i.e.~the vertices do not necessarily determine the simplex. 
In all examples we will work with in the present paper, simplices of $W_n(A,X)_\bullet$ will be determined by their ordered list of vertices, which will be distinct. This is, however, not an automatic feature of these simplicial sets, as shown by the following pathological example (pointed out to us by Peter Patzt): 

\begin{ex}
Let $\e=\coprod_{n\ge 0}\{e\}$ be the groupoid consisting of a copy of the trivial group for each natural number. This groupoid has a monoidal structure defined by addition on the objects, and the identity map defines a symmetry for $\e$. The category $U\e$ has objects the natural numbers with a unique morphism from $m$ to $n$ whenever $m\le n$: in other words, it is the poset $(\N, \leq)$. Take $X=1$ and $A=0$. Then the semi-simplicial set $W_n(A,X)_\bullet$ has a single $p$-simplex for every $p\le n-1$, and $\vert W_n(A,X)_\bullet\vert$ is an $(n-1)$-dimensional ``Dunce hat'' (which is nonetheless $(n-2)$-connected).
\end{ex}

To avoid such pathological examples, we make the following {\em standardness} assumption throughout this section: 

\begin{Def}
Let $(\C,\op,0)$ be a homogeneous category and $(A,X)$ a pair of objects in $\C$. We say that $\C$ is {\em locally standard at $(A,X)$} if 
\begin{itemize} 
\item[{\bf LS1}] The morphisms $\iota_A\op X\op \iota_X$ and $\iota_{A\op X}\op X$ are distinct in $\Hom(X,A\op X^{\op 2})$;
\item[{\bf LS2}] For all $n\ge 1$, the map $\Hom(X,A\op X^{\op n-1})\to \Hom(X,A\op X^{\op n})$ taking $f$ to $f\op \iota_X$ is injective. 
\end{itemize}
\end{Def}

\begin{prop}\label{standardprop}
Let $(\C,\op,0)$ be locally homogeneous at $(A,X)$. Then $\C$ is locally standard at $(A,X)$ if and only if all simplices of $W_n(A,X)_\bullet$ for all $n$ are determined by their vertices and their vertices are all distinct. 
\end{prop}

Recall that the group $G_n=\Aut(A\op X^{\op n})$ acts on $W_n(A,X)_\bullet$ by post-composition. For a simplex $f\in W_n(A,X)_p$,  we denote by $\St(f) \le G_n$ its stabiliser. Towards the proof of the above proposition, we give different equivalent characterisations of LS2, one of them in terms of this action: 

\begin{lem}\label{LH4equiv}
Let $(\C,\op,0)$ be locally homogeneous at $(A,X)$. Then the following are equivalent: 
\begin{enumerate}[(i)]
\item $\C$ satisfies LS2 at $(A,X)$;
\item For every $n$ and every edge $f\in W_n(A,X)_1$, $\St(f)=\St(d_0f)\cap \St(d_1f)$;
\item For every $n$ and $p$, if $f,f'\in W_n(A,X)_p$ have the same ordered set of vertices, i.e.~if $d_0^id_{i+1}d_{i+2}\dots d_pf=d_0^id_{i+1}d_{i+2}\dots d_pf'$ for all $0\le i\le p$, then $f=f'$.  
\end{enumerate}
\end{lem}

\begin{proof}
We first show that (iii) $\Rightarrow$ (ii). Let  $f\in W_n(A,X)_1$. We always have that $\St(f)\subset \St(d_0f)\cap \St(d_1f)$. Suppose $\phi\in \St(d_0f)\cap \St(d_1f)$. Then $\phi\circ f$ is a new 1-simplex in $W_n$ with the same vertices as $f$. By (iii), this means that $\phi\circ f=f$, i.e.~that $\phi\in \St(f)$.  This proves (ii). 

\medskip

We now show that (ii) $\Rightarrow$ (i). Suppose we have $f, g : X \to A \op X^{\op n-1}$ such that $f \op \iota_X = g \op \iota_X$. By LH1,  we may assume without loss of generality that $f = \iota_{A \op X^{\op n-2}} \op X$ and that $g = \phi \circ f$ for some $\phi \in \Aut(A \op X^{\op n-1})$.

Thus $f \op \iota_X = (\phi \circ f) \op \iota_X = (\phi \op X) \circ (f \op \iota_X)$ and so $\phi \op X$ stabilises $f \op \iota_X = \iota_{A \op X^{\op n-2}} \op X\op \iota_X$. Now $\phi \op X$ also stabilises $\iota_{A \op X^{\op n-1}} \op X$, so by (ii) it stabilises the edge $\iota_{A \op X^{\op n-2}} \op X^{\op 2}$ in $W_n(A,X)_\bullet$. By LH2, it follows that $\phi \op X = \phi' \op X^{\op 2}$. By a further application of LH2, it follows that $\phi = \phi' \op X$. But then $\phi$ fixes $f$, so $g=f$.

\medskip

We are left to show that  (i) $\Rightarrow$ (iii).
Let $f, g : X^{\op p+1} \to A \op X^{\op n}$ be two $p$-simplices having the same sets of ordered vertices: that is, they agree when precomposed with any
$$i_j=\iota_{X^{\op j}}\op X\op \iota_{X^{\op p-j}}\colon X \lra X^{\op p+1}.$$
By LH1, we may suppose that $f = \iota_{A \op X^{\op n-p-1}} \op X^{\op p+1}$, and we may also find a $\phi \in \Aut(A \op X^{\op n})$ such that $\phi \circ g = f$. Now $f$ and $g$ precomposed with $i_p$ both give $\iota_{A \op X^{\op n-1}} \op X$, and so $\phi \in \Fix(\iota_{A \op X^{\op n-1}} \op X)$ and hence by LH2 we have $\phi = \phi' \op X$. Similarly, $f$ and $g$ precomposed with $i_{p-1}$ both give $\iota_{A \op X^{\op n-2}} \op X \op \iota_X$. Hence $(\phi'\op X)\circ (\iota_{A \op X^{\op n-2}} \op X \op \iota_X ) = (\phi'\circ (\iota_{A \op X^{\op n-2}} \op X)) \op \iota_X=\iota_{A \op X^{\op n-2}} \op X \op \iota_X$. Now by (i), this means that $\phi' \in \Fix(\iota_{A \op X^{\op n-2}} \op X)$ and hence $\phi' = \phi'' \op X$. Continuing in this way, we find that $\phi = \psi \op X^{\op p+1}$, but then $g = (\psi^{-1} \op X^{\op p+1}) \circ (\iota_{A \op X^{\op n-p-1}} \op X^{\op p+1}) = \iota_{A \op X^{\op n-p-1}} \op X^{\op p+1}$ as required.
\end{proof}

\begin{proof}[Proof of Proposition~\ref{standardprop}]
Suppose first that simplices of $W_n(A,X)_\bullet$ for all $n$ are determined by their vertices which are distinct. In particular, this holds for the 1-simplex $\iota_A\op X^{\op 2}$ whose vertices are $\iota_A\op X\op \iota_X$ and $\iota_{A\op X}\op X$ which immediately gives LS1. Moreover, by Lemma \ref{LH4equiv}, we have that $\C$ also satisfies LS2. 

Conversely, if $\C$ satisfies LS1 and LS2 then by Lemma \ref{LH4equiv} we know that simplices of $W_n(A,X)_\bullet$ are determined by their vertices. Now every $p$-simplex of $W_n(A,X)_\bullet$ has distinct vertices if and only if every edge of $W_n(A,X)_\bullet$ has distinct vertices. By LH1, this holds if and only if the edge $\iota_A\op X^{\op 2}\op \iota_{X^{\op n-2}}$ has distinct vertices, which by LS2 holds if and only if LS1 holds. 
\end{proof}

We associate a simplicial complex to the pair $(A,X)$, which is closely related to the semi-simplicial set $W_n(A,X)_\bullet$.

\begin{Def} Let $(\C,\op,0)$ be a monoidal category with 0 initial and $(A,X)$ a pair of objects in $\C$. 
Define $S_n(A,X)$ to be the simplicial complex with vertices $W_n(A,X)_0$, and where a set of vertices spans a simplex if and only if there is a simplex of $W_n(A,X)_\bullet$ having them as vertices.
\end{Def}

When $\C$ is locally standard at $(A,X)$, there is a surjective map on the set of $p$-simplices
$$\pi_p:W_n(A,X)_p\surj S_n(A,X)_p$$
for each $p$. The maps $\pi_p$ together define a map on geometric realisation 
 $\pi: \vert W_n(A,X)_\bullet \vert \to \vert S_n(A,X) \vert.$
The goal of this section is to describe two situations in which the high connectivity of the spaces $\vert S_n(A,X)\vert$ can be leveraged to deduce the high-connectivity of the spaces $\vert W_n(A,X)_\bullet\vert$. 

There are two situations which arise in examples, namely the two extreme ones:

\begin{enumerate}[(A)]
\item For any $p$-simplex $\{ v_0, v_1,\ldots, v_p\}$ of $ S_n(A,X)$ and any permutation $\la \in \Sigma_{p+1}$, there exists a simplex of $W_n(A,X)_p$ whose vertices are $(v_{\la(0)}, v_{\la(1)}, \ldots, v_{\la(p)})$. In particular, there are exactly $(p+1)!$ elements in $\pi_p^{-1}(\s)$ for any $p$-simplex $\s\in S_n(A,X)_p$. 

\item There is a single element in $\pi_p^{-1}(\s)$ for any $p$-simplex $\s\in S_n(A,X)_p$. 
\end{enumerate}

In case (B), under the locally standard assumption, we have that $\vert W_n(A,X)_\bullet \vert$ and $\vert S_n(A,X) \vert$ are homeomorphic, so it is equivalent to work with the simplicial complex $S_n(A,X)$. However, case (A) is the most common situation; it occurs whenever the category is symmetric monoidal: 

\begin{prop}\label{prop:SymMonBuilding}
Suppose $(\C, \op, 0)$ is a symmetric monoidal category, locally homogeneous and locally standard at $(A,X)$. Then the semi-simplicial set $W_n(A,X)_\bullet$ satisfies condition {\rm (A)}. 
\end{prop}

\begin{proof}
If $\s$ is a $p$-simplex of $S_n(A,X)$, there must exist a $p$--simplex $f$ of $W_n(A,X)_\bullet$ such that $\pi(f)=\s$. We have  $f: X^{\op p+1}\to A\op X^{\op n}$. Now let $\la\in \Si_{p+1}$ be a permutation of the set $\{0,\dots,p\}$. Then the symmetry of $\C$ defines a morphism $\la:X^{\op p+1}\to X^{\op p+1}$ such that $\la\circ i_j=i_{\la(j)}$ for $i_j= \iota_{X^{\op j}}\op X\op \iota_{X^{\op p-j}}:X\to X^{\op p+1}$. Hence the simplex $f\circ \la$ in $W_n(A,X)_\bullet$  has vertices $f\circ\la\circ i_j$, the vertices of $f$ permuted by $\la$, showing that any permutation of the vertices of $\s$ can be obtained. 
The result follows by local standardness and Proposition~\ref{standardprop}. 
\end{proof}

In the remainder of this section we shall prove the following result. 

\begin{thm}\label{caseA}
Let $(\C, \op, 0)$ be locally homogeneous and locally standard at $(A,X)$. Suppose that $W_n(A,X)_\bullet$ satisfies condition {\rm (A)} for all $n \geq 0$. Let $a,k\ge 1$. Then the simplicial complex $S_n(A,X)$ is $(\frac{n-a}{k})$-connected for all $n\ge 0$, if and only if the semi-simplicial set $W_n(A,X)_\bullet$ is $(\frac{n-a}{k})$-connected  for all $n\ge 0$.
\end{thm}

Note that in the theorem, we need to know about the connectivity of the spaces $S_n$ (resp.~$W_n$) for all $n\ge 0$ with a slope $k\ge 1$, whereas axiom LH3 only requires connectivity for all $n\ge 1$ with a slope $k\ge 2$. This means that if each $S_n(A,X)$ is $(n-2)$--connected for all $n\ge 0$, we can conclude that so is each $W_n(A,X)$, and hence that $W_n(A,X)$ is $(\frac{n-2}{2})$--connected for all $n\ge 1$ as $\lfloor\frac{n-2}{2}\rfloor\le n-2$ when $n\ge 1$. This situation occurs in several examples, where $S_0(A,X)$ and $W_0(A,X)$ are empty, in which case it is true that they are $(0-2)$-connected, though not $(\frac{0-2}{2})$-connected. 

We start by showing that, when condition (A) is satisfied, the simplicial complexes $S_n(A,X)$ satisfy the following niceness condition.

\begin{Def}\label{CMdef}
Following the terminology of \cite{GalRW14,HatWah10}, we call a simplicial complex $X$ 
{\em weakly Cohen--Macaulay of dimension $n$} if 
\begin{enumerate}[(i)]
\item $X$ is $(n-1)$-connected, and
\item for every $p\ge 0$, the link of each $p$-simplex in $X$ is $(n-p-2)$-connected. 
\end{enumerate}
The complex $X$ is called  {\em locally weakly Cohen--Macaulay of dimension $n$} if it just satisfies property (ii). 

As per our convention on connectivity, we allow $n$ to be a rational number. 
Note that a weakly Cohen--Macaulay complex of dimension $n$ is necessarily of actual dimension at least $\lfloor n\rfloor$, as property (ii) requires that the link of an $\lfloor n-1\rfloor$-simplex be $(-1)$-connected, i.e.\ non-empty, which inductively shows that there must exist simplices of dimension $\lfloor n\rfloor$. 
\end{Def}

\begin{prop}
Suppose $(\C, \op, 0)$ is locally homogeneous and standard at $(A,X)$, 
and that each $W_n(A,X)_\bullet$ satisfies condition {\rm (A)}. If $\sigma$ is a $p$-simplex of $S_n(A,X)$ with $p \leq n-1$ then $\link(\sigma) \cong S_{n-p-1}(A,X)$. 
\end{prop}
\begin{proof}
Let $\sigma = \{ v_0, v_1, \ldots, v_p \}$ be a $p$-simplex of $S_n(A,X)$. Using LH1, we may assume that $v_i=\iota_{A\op X^{n-p-1+i}}\op X\op \iota_{X^{\op p-i}}:X\to A\op X^{\op n}$ is the inclusion of the $(n-p+i)$th $X$-summand. We have an inclusion $$\al\colon S_{n-p-1}(A,X)\inc \link(\sigma)$$ 
of simplicial complexes defined by taking $\{ g_0,\dots,g_k\}$  to $\{ \widehat g_0,\dots,\widehat g_k\}$, where 
$$\widehat g_i:X\ \sta{g_i}\rar\  A\op X^{\op n-p-1}\ \sta{A\op X^{\op n-p-1} \op
  \iota_{X^{\op p+1}}}\rar\  A\op X^{\op n-p-1}\op X^{\op p+1}.$$
Indeed,  $n-p-1\ge 0$ as we have assumed that $p\le n-1$. Also, $\{ \widehat g_0,\dots,\widehat g_k,v_0,  \ldots, v_p \}$ is a $(p+k+1)$-simplex of $S_n(A,X)$ as, by LS2, the $\widehat g_i$'s are distinct and by LS1 combined with LH1, they are distinct from the vertices of $\s$.  Finally, if $\{ g_0,\dots,g_k\}$ is represented by $g:X^{\op k+1}\to A\op X^{\op n-p-1}$ in $W_{n-p-1}(A,X)_k$, then this larger simplex is represented by $g\op X^{\op p+1}$ in $W_n(A,X)_{p+k+1}$. So $\al$ is well-defined. It is clearly simplicial and it is injective by LS2. 

We are left to show the map $\al$ is an isomorphism, so let $\{ \widehat g_0,\dots,\widehat g_k\}$ be a $k$-simplex of the link of $\sigma$. As $\{ \widehat g_0,\dots,\widehat g_k,v_{0},\dots,v_p\}$ is a simplex of $S_n(A,X)$ there exists a $(p+k+1)$-simplex 
$$\widehat g\colon X^{\op k+1} \op X^{\op p+1}  \rar  A\op X^{\op n}$$
of $W_n(A,X)_\bullet$ such that $\widehat g\circ i_j= \widehat g_{j}$ for $j=0,\dots,k$ and $\widehat g\circ i_j=v_{j-k-1}$ for $j=k+1,\dots,k+p+1$.   (Note that we may choose this particular ordering because $W_n(A,X)_\bullet$ is assumed to satisfy condition (A).) Now let $\phi\in\Aut(A\op X^{\op n})$ be such that $\phi\circ \widehat g=\iota\op {X^{\op k+p+2}}$. (Such an automorphism exists by LH1.) As $\widehat{g}$ restricts to the inclusion of the last $X^{\op p+1}$, it follows that $\phi \in \Fix(\iota_{A \op X^{\op n-p-1}}\op X^{\op p+1})$ and thus by LH2 that $\phi=\phi'\op {X^{\op p+1}}$ for some $\phi'\in\Aut(A\op X^{\op n-p-1})$. Now let $$g=(\phi')^{-1}\circ (\iota_{A\op X^{\op n-p-k-2}}\op {X^{\op k+1}})\colon X^{\op k+1}\ \rar \ A\op X^{\op n-p-1}.$$
Then $g$ is a lift in $W_{n-p-1}(A,X)_\bullet$ of a simplex $\{ g_0,\dots,g_k\}$ of $S_{n-p-1}(A,X)$ which is mapped to $\{ \widehat g_0,\dots,\widehat g_k\}$ by $\al$.
\end{proof}

\begin{cor}\label{CMcor}
Under the same conditions, suppose there are numbers  $a,k\ge 1$ such that for all $n\ge 0$,  the complex $S_n(A,X)$ is $(\frac{n-a}{k})$-connected, then each $S_n(A,X)$ is weakly Cohen--Macaulay of dimension $\frac{n-a+k}{k}$.
\end{cor}
\begin{proof}
By the assumption and the previous proposition, the link of a $p$-simplex of
$S_n(A,X)$ is $(\frac{n-p-1-a}{k})$-connected for any $p\le n-1$. 
As $\frac{n-p-1-a}{k}\ge \frac{n-a}{k}-p-1$ when $k\ge 1$, we have that property (ii) of Definition~\ref{CMdef} is satisfied for any $p\le n-1$. On the other hand, property (ii) for $p\ge n$ is void as $\frac{n-a}{k}-p-1\le n-\frac{a}{k}-n-1< -1$ when $a>0$ and  $k\ge 1$, while property (i) is satisfied by assumption. 
\end{proof}

Given a simplicial complex $Y$, one can form an associated semi-simplicial set $Y^{ord}_\bullet$ which has a $p$-simplex for every $p$-simplex of $Y$ and every choice of ordering of its vertices. Under condition (A) and local standardness, the semi-simplicial set $W_n(A,X)_\bullet$ will be isomorphic to $S_n(A,X)^{ord}_\bullet$. 
Although for a simplicial complex $Y$ the geometric realisation of $Y$ and $Y^{ord}_\bullet$ are in general very different, the following shows that for weakly Cohen--Macaulay complexes, the connectivity property of $Y$ survives in $Y^{ord}_\bullet$. 

\begin{prop}\label{RWprop}
If $Y$ is a simplicial complex which is weakly Cohen--Macaulay of dimension $n$, then the associated semi-simplicial set $Y^{ord}_\bullet$ is $(n-1)$-connected. 
\end{prop}
\begin{proof}
Let $f:S^k\to |Y^{ord}_\bullet|$ be a map, and consider $\bar f:=q\circ f:S^k\to |Y|$ the composition of $f$ with the projection $q:|Y^{ord}_\bullet|\to |Y|$ which forgets the ordering. 
By \cite[Thm.~5.1]{RouSan71}, we can assume that $f$ is simplicial with respect to some PL triangulation of $S^k$.  If $k\le n-1$, there is a map $g:D^{k+1}\to |Y|$ such that $g$ restricts to $\bar f$ on $S^k$. Again, after changing $g$ by a homotopy we can assume that it is simplicial with respect to a PL triangulation of $D^{k+1}$ extending the triangulation of $S^k$, and in fact, by \cite[Thm.~2.4]{GalRW14}, the map $g$ can be chosen such that for each vertex $v$ in $D^{k+1}$ which does not lie in $S^{k}$
$$\st(v)=(\st(v)\cap S^k)*(\st(v)\cap \operatorname{int}(D^{k+1}))  \ \ \textrm{and}\ \  g(\link(v)) \subset \link(g(v)).$$     
The first property is equivalent to saying that every simplex $\sigma$ of $D^k$ decomposes as $\sigma_\partial * \sigma_{int}$, a boundary simplex joined with an interior simplex. 
The second property says that, 
 if $v$ does not lie in the boundary, and $(v,w)$ is a 1-simplex, then $(g(v), g(w))$ is also a 1-simplex, i.e.~$g(v) \neq g(w)$, which in particular implies that for any simplex $\s$ of $D^{k+1}$, $$g(\sigma) = g(\sigma_\partial) * g(\sigma_{int}).$$

We wish to find a lift $G$ of $g$ to $|Y^{ord}_\bullet|$ which extends $f$. To do so, choose a total order $v_1, \ldots, v_r$ of the internal vertices of $D^{k+1}$. 
Given a simplex $\s$ of $D^{k+1}$, its image $g(\sigma) = g(\sigma_\partial) * g(\sigma_{int})$
has a preferred lift, given by the ordering $f(\sigma_\partial)$ of $g(\sigma_\partial)$, the ordering of $g(\sigma_{int})$ coming from the total
order of the interior vertices of $D^{k+1}$ (which is well-defined because $g$ is simplexwise injective on interior simplices), and the standard order coming from the join, putting $g(\sigma_\partial)$ first, then $g(\sigma_{int})$. This defines a lift $G(\sigma)$ of $g(\s)$.  As the choices are compatible with passing to faces, we actually get a simplicial map $G:D^{k+1}\to |Y^{ord}_\bullet|$ extending $f$, as required. 
\end{proof}

Finally, we are ready to prove the main result of the section: 

\begin{proof}[Proof of Theorem~\ref{caseA}]  By Proposition~\ref{standardprop} and the condition (A) assumption, 
we have that  $W_n(A,X)_\bullet$ is isomorphic to $S_n^{ord}(A,X)_\bullet$. 
Choosing a total order of the vertices of $S_n(A,X)$ gives a section $\vert S_n(A,X) \vert \to \vert S_n^{ord}(A,X)_\bullet \vert=\vert W_n(A,X)_\bullet \vert $ of the projection map, so that
$$\vert W_n(A,X)_\bullet \vert \text{ is $l$-connected } \Longrightarrow \vert S_n(A,X) \vert \text{ is $l$-connected}$$
which proves one implication. 

Conversely, as the $W_n(A,X)_\bullet$ satisfy condition (A), by
Corollary~\ref{CMcor}, $S_n(A,X)$ is weakly Cohen--Macaulay of dimension $\frac{n-a+k}{k}$. Hence by Proposition~\ref{RWprop}, 
$S_n(A,X)^{ord}_\bullet=W_n(A,X)_\bullet$ is $(\frac{n-a}{k})$-connected, which proves the other implication.
\end{proof}

\section{Homological stability with constant and abelian coefficients}\label{cstsec}

In this section, we use Quillen's classical argument \cite{QuiNotes} in the case of general linear groups to show that, in a (locally) homogeneous category $(\C,\op,0)$ in which the semi-simplicial sets $W_n(A,X)_\bullet$ are highly connected, the groups $\Aut(A\op X^{\op n})$ satisfy homological stability with constant coefficients $\Z$. The precise statement of the theorem is as follows:

\begin{thm}\label{stabthm}
Let $(\C,\op,0)$ be a pre-braided category which is locally homogeneous at a pair of objects $(A,X)$. Suppose that $\C$ satisfies LH3 at $(A,X)$ with slope $k \geq 2$. 
Then the map $$H_i(\Aut(A\op X^{\op n});\bZ)\ \rar \  H_i(\Aut(A\op X^{\op n+1});\bZ)$$ is an epimorphism if $i\le \frac{n}{k}$ and an isomorphism if $i\le\frac{n-1}{k}$.
\end{thm}

\begin{rem} The hypothesis of pre-braidedness in this theorem can be replaced by the assumption that $\C$ also satisfies LH1 at $(0,X)$.
If $\C$ satisfies LH3 at $(A,X)$ with slope $k=1$, then the isomorphism range can be improved to be equal to the epimorphism range $i\le \frac{n}{2}$ by a variation of the argument below 
(see \cite{R-WConf, Gan16} for the case of the symmetric groups). 
\end{rem}

We shall also prove homological stability for these groups with respect to certain non-trivial coefficient modules, quite different from the ``homological stability with twisted coefficients" often considered (which we will treat in Section \ref{sec:twist}). 
To define these coefficient systems, we first form the stable group $\Aut(A \oplus X^{\oplus \infty})$ as the colimit of
$$\cdots \overset{-\oplus X}\lra \Aut(A \oplus X^{\oplus n}) \overset{-\oplus X}\lra \Aut(A \oplus X^{\oplus n+1}) \overset{-\oplus X}\lra \Aut(A \oplus X^{\oplus n+2}) \overset{-\oplus X}\lra \cdots.$$
Then any $\Aut(A \oplus X^{\oplus \infty})$-module $M$ may be considered as an $\Aut(A \oplus X^{\oplus n})$-module for any $n$, by restriction, which we continue to call $M$. We shall say that the module $M$ is \emph{abelian} if the action of $\Aut(A \oplus X^{\oplus \infty})$ on $M$ factors through the abelianisation of $\Aut(A \oplus X^{\oplus \infty})$, or in other words if the derived subgroup $\Aut(A \oplus X^{\oplus \infty})'$ acts trivially on $M$.

\begin{ex}[The sign representation]\label{signrepex}
Let $\Si=\coprod_{n\ge 0}\Si_n$ be the groupoid with objects the natural numbers and automorphism groups the symmetric groups. This is a symmetric
monoidal groupoid whose associated category $U\Si$ is homogeneous by Theorem~\ref{universal}. (The category $U\Si$ is a skeleton for 
 the category $FI$ of finite sets and injections introduced in Section~\ref{homcatsec}, see also Section~\ref{setex}.)
Taking $A=0$ and $X=1$, we have $\Aut(A\op X^{\op n})=\Si_n$ and  $\Aut(A\op X^{\op \infty})=\Si_\infty$, the group of permutations of the natural numbers with finite support. We have that $\Si_\infty^{ab}=\Z/2$ and the map $\Si_n\to \Si_{\infty}^{ab}$ is given by the sign of the permutation. Hence the sign representation of the symmetric groups, namely the sequence of $\Si_n$-modules $\Z$ with $\s\in\Si_n$ acting via its sign, is an example of an abelian coefficient system for $\C=U\Si$. 
\end{ex}

\begin{thm}\label{abstabthm}  
Let $(\C,\op,0)$ be a pre-braided category which is locally homogeneous at a pair of objects $(A,X)$. Suppose that $\C$ satisfies LH3 at $(A,X)$ with slope $k \geq 3$ (!). Then for any abelian $\Aut(A \oplus X^{\oplus \infty})$-module $M$ the map
\begin{equation*}
H_i(\Aut(A \oplus X^{\oplus n});M) \lra H_i(\Aut(A \oplus X^{\oplus n+1});M)
\end{equation*}
is an epimorphism for $i \leq \tfrac{n-k+2}{k}$ and an isomorphism for $i \leq \tfrac{n-k}{k}$.
\end{thm}

\begin{rem}
In these theorems we restrict to linear stability slopes with $k\ge 2$ (or $k \geq 3$ for Theorem \ref{abstabthm}) as the proof given here does not give a better range for lower $k$. 

Note that if each $W_n(A,X)_\bullet$ is $(\frac{n-a}{k})$-connected for some $a>2$, then
$W_n(A\op X^{\op a-2},X)_\bullet$ is $(\frac{n-2}{k})$-connected for each $n$, so that the theorem can instead be applied to the pair $(A\op X^{\op a-2},X)$, yielding a shifted stability range. 
\end{rem}

The proofs of these two theorems are largely the same, and we will give them simultaneously. The proof of Theorem \ref{stabthm} is very close to that of Theorem 5.1 in \cite{HatWah10}, which essentially yields the case $k=2$. The modification of this argument necessary to prove Theorem \ref{abstabthm} is based on one appearing in the first author's work with Galatius \cite[\S 8]{GalRW14}, which we believe is the first time homological stability with abelian coefficient systems has been explicitly considered. 

\begin{proof}[Proof of Theorem~\ref{stabthm} and Theorem~\ref{abstabthm}]
Let us write $G_n:=\Aut(A\op X^{\op n})$ and $W_n=W_n(A,X)_\bullet$, and fix an abelian $\Aut(A \oplus X^{\oplus \infty})$-module $M$ (which may also be $\mathbb{Z}$ with trivial action). Consider the spectral sequences associated to the
double complex 
\begin{equation}\label{eq:DoubleCx}
E_\bullet G_{n+1}\otimes_{G_{n+1}}(\widetilde{C}_*(W_{n+1}) \otimes_{\bZ} M),
\end{equation}
where 
$E_\bullet G_{n+1}$ is a free resolution of $\Z$ over $\Z G_{n+1}$, $\widetilde{C}_*(W_{n+1})$ is the augmented cellular chain complex of $W_{n+1}$, and $\widetilde{C}_*(W_{n+1}) \otimes_{\bZ} M$ has the diagonal $G_{n+1}$-action. 

We now make a number of observations about this set-up.

\medskip

\noindent (a) {\em $W_{n+1}$ is $(\frac{n-1}{k})$-connected.} (This holds by LH3.) 

\medskip

\noindent (b) {\em The action of $G_{n+1}$ on $W_{n+1}$ is transitive on $p$-simplices for each $p \leq n$ and the stabiliser of each simplex fixes the simplex pointwise.} The set of $p$-simplices of $W_{n+1}$ is $\Hom(X^{\op p+1},A\op X^{\op n+1})$, and by axiom LH1 the group $G_{n+1}$ acts transitively on this set for $0 \leq p \leq n$. The stabiliser of a $p$-simplex $f\in\Hom(X^{\op p+1},A\op X^{\op n+1})$ fixes its vertices $f\circ i_j$ and hence fixes the
simplex pointwise. 

\medskip

Let us fix for each $p$ the \emph{standard $p$-simplex} $\sigma_p$ of $W_{n+1}$ to be the map
$$\iota_{A \op X^{\op n-p}}\op {X^{\op p+1}}:X^{\op p+1}\cong 0\op X^{\op p+1}\lra A\op X^{\op n-p}\op X^{\op p+1}.$$

\medskip

\noindent (c) {\em The map $- \op X^{\op p+1} : G_{n-p} \to G_{n+1}$ is an isomorphism onto the stabiliser of $\sigma_p$.} By axiom LH2, this map is injective with image $\Fix(\iota_{A \op X^{\op n-p}} \op X^{\op p+1}) \leq G_{n+1}$. But this subgroup is precisely the stabiliser of the $p$-simplex $\sigma_p$ of $W_{n+1}$.

Note that (b) and (c) together show that the set of $p$-simplices of $W_{n+1}$ is isomorphic as a $G_{n+1}$-set to $G_{n+1}/G_{n-p}$, so that 
$$\widetilde C_p(W_{n+1})\cong \Z [G_{n+1}/G_{n-p}]$$ as a $\Z G_{n+1}$-module. Our choice of the standard $p$-simplex $\sigma_p$ determines such an isomorphism, but we emphasise that this is a choice: 
in fact, distinct such isomorphisms play a role in the argument, via distinct subgroups of $G_{n+1}$ isomorphic to $G_{n-p}$, and it is important to keep track of exactly which subgroup one is talking about. 

\medskip

The two standard filtrations of the double complex \eqref{eq:DoubleCx} give two spectral sequences. By observation (a) above, the chain complex $\widetilde{C}_*(W_{n+1})$ has trivial homology in degrees $* \leq \tfrac{n-1}{k}$, and so $\widetilde{C}_*(W_{n+1}) \otimes_\bZ M$ does too by the Universal Coefficient Theorem. The spectral sequence obtained by taking the $*$-homology first then shows that the totalisation of \eqref{eq:DoubleCx} has trivial homology in the same range of degrees. The spectral sequence obtained by taking the $\bullet$-homology first has, by observations (b) and (c) above and Shapiro's lemma (see \cite[VII.7]{B} for more details),
$$E^1_{p,q} = H_q(\St(\sigma_p) ; M) \cong H_q(G_{n-p};M) \ \text{ for } -1 \leq p \leq n,$$ 
where $M$ is considered as a $\St(\sigma_p)$-module via the inclusions $\St(\sigma_p) \leq G_{n+1}$ and $G_{n+1} \leq G_\infty$. We shall only be interested in degrees $p+q \leq \tfrac{n-1}{k}$, so shall ignore the fact that the above only holds for $p \leq n$. 
This spectral sequence must also converge to 0 for $p+q \leq \tfrac{n-1}{k}$. The $d^1$ differential is a map $d^1:E_{p,q}\to E_{p-1,q}$ induced by the differential in $\widetilde{C}_*(W_{n+1})$. It is given by the alternating sum of the maps
$$H_q(\St(\s_p);M)\stackrel{(inc, id_M)_*}{\rar}H_q(\St(d_i\s_p);M)\stackrel{c_{h_i}}{\rar}H_q(\St(\s_{p-1});M)$$ 
where $c_{h_i}$ denotes the map on homology induced by the map of pairs
\begin{equation*}
(g, m) \mapsto (h_i g h_i^{-1}, h_i \cdot m) : (\St(d_i\s_p),M) \lra (\St(\s_{p-1}),M)
\end{equation*}
for any element $h_i\in G_{n+1}$ which takes $d_i\s_p$ to the standard $(p-1)$-simplex $\s_{p-1}$. In particular, for $p=0$, the map
$$d^1\colon H_i(\St(\s_0);M)\rar H_i(\St(\s_{-1});M)=H_i(G_{n+1};M)$$ 
identifies with the map in the statement of the theorem by (c). In order to better understand the $d^1$ differential we make the following additional observation.

\medskip

\noindent (d) {\em For every $p\ge 1$ and $1\le i\le p$, there exists an element $h_{i}\in G_{n+1}$ such that  $h_{i}\circ d_i\sigma_p = \sigma_{p-1}$ with the property that $h_{i}$ centralises  $\St(\sigma_p)$.} Indeed, we have 
$$d_i\sigma_p=\sigma_p\circ (X^{\op i}\op \iota_X\op X^{\op p-i})=\iota_{A\op X^{\op n-p}}\op {X^{\op i}}\op \iota_X\op X^{\op p-i}$$ while 
$$\sigma_{p-1}=\iota_{A\op X^{\op n-p+1}}\op  {X^{\op p}}= \iota_{A\op X^{\op n-p}}\op \iota_X\op  {X^{\op i}}\op X^{\op p-i}.$$ As $\C$ is pre-braided, we have that $b_{X^{\op i},X}\circ  (X^{\op i}\op \iota_X)= \iota_X\op X^{\op i}$. Hence if we choose $h_{i}:={A\op X^{n-p}}\op b_{X^{\op i},X}\op X^{\op p-i}\in G_{n+1}$ then $h_{i}\circ d_i\sigma_p=\sigma_{p-1}$. Moreover, the map $- \op X^{\op p+1} : \Aut(A\op X^{\op n-p}) \to \St(\sigma_p)$ is an isomorphism by observation (c) above, and this subgroup of $G_{n+1}$ commutes with $h_{i}$ as $h_{i}$ acts as the identity on $A\op X^{\op n-p}$, proving the claim. 

\medskip

To identify the $i$th term $(-1)^i d_i$ in the differential $d^1$, we choose $h_i$ as in (d), whereupon $c_{h_i} \circ (inc, id_M)$ is
$$(g, m) \mapsto (h_i g h_i^{-1}, h_i \cdot m) = (g, h_i \cdot m) : (\St(\s_p),M) \lra (\St(\s_{p-1}),M)$$
because $h_i$ centralises $\St(\s_p)$. It follows that
$$d^1 = \sum_{i=0}^p (-1)^i (inc, h_i \cdot -)_* : H_*(\St(\sigma_p);M) \lra H_*(\St(\sigma_{p-1});M).$$

In particular, whenever all the $h_i$ act trivially on $M$---for example when $M$ is a trivial module---then the differential $d^1 : E^1_{p,q} \to E^1_{p-1,q}$ is zero if $p$ is odd and can be identified with the stabilisation map $H_q(G_{n-p};M) \to H_q(G_{n-p+1};M)$ if $p$ is even.

\medskip

\noindent\textbf{Constant coefficients case.} We want to show that the differential 
$$d^1\colon E^1_{0,i}=H_i(G_{n};\bZ)\rar E^1_{-1,i}=H_i(G_{n+1};\bZ)$$ 
is surjective when $n\ge ki$ and injective when $n\ge ki+1$. 
We prove this by induction on $i$, so consider the statements

\vspace{1ex}

\noindent \Ei{I} This map is epi for all $i \leq I$ and all $n$ such that $n \geq ki$.

\vspace{1ex}

\noindent \Ii{I} This map is iso for all $i \leq I$ and all $n$ such that $n \geq ki+1$.

\vspace{1ex}

\noindent The statements \Ei{0} and \Ii{0} hold trivially. 

\medskip

We start by showing the implication $({\bf E}_{I-1}) + ({\bf I}_{I-1}) \Rightarrow ({\bf E}_I)$, so let $i \leq I$ and $n$ satisfy $n \geq ki$, or equivalently $i \leq \tfrac{n}{k}$. Surjectivity of the map $d^1$ above follows from: 
\begin{itemize}
\item[\EI{1}] $E^\infty_{-1,i}=0$; 
\item[\EI{2}] $E^2_{p,q}=0$ for $p+q=i$ with $q<i$. 
\end{itemize}
Indeed, condition \EI{1} says that $E^1_{-1,i}$ has to be killed before $E^\infty$, and condition \EI{2} says that the sources of all the possible
differentials to that term after $d^1$ are 0, and therefore $d^1$ is the only differential that can kill it, so it must be surjective. 

Condition \EI{1} holds because $E^\infty_{p,q}=0$ when $p+q\le \frac{n-1}{k}$, and $i-1\le \frac{n-k}{k} \le \frac{n-1}{k}$ when $n\ge ki$ and $k\ge 1$.

For condition \EI{2}, we first show that for $p+q=i$ with $q<i$ the map induced by inclusion of stabilisers 
$$E^1_{p',q}=H_q(\St(\s_{p'});\Z)\stackrel{\cong}{\rar} H_q(G_{n+1};\Z)$$ 
is an isomorphism if $p'\le p$, and an epimorphism if $p'=p+1$. By observation (c) above, this map may be written as a composition
$$H_q(G_{n-p'};\Z) \lra H_q(G_{n-p'+1};\Z) \lra H_q(G_{n-p'+2};\Z) \lra \cdots \lra H_q(G_{n+1};\Z)$$
of stabilisation maps. Each of these maps is an epimorphism if $n-p' \geq kq$ and an isomorphism if $n-p' \geq kq+1$, by \Ei{I-1} and \Ii{I-1} respectively. If $p'+q\le p+q \leq i \leq \tfrac{n}{k}$ then $n \geq kp' + kq \geq p' + kq + 1$ for $k \geq 2$ and $p'\geq 1$, so each of these maps is an isomorphism, as claimed. When $p'=0$, we use $q<i\le \frac{n}{k}$ which gives $n\ge kq+k\ge kq+1$. 
Finally if $p'+q = i+1 \leq \tfrac{n+k}{k}$ then $n+k \geq kp'+kq$ and so $n+k(1-p') \geq kq$. But we have $k \geq 2$ and $p' \geq 2$ in this case, so $-p' \geq k(1-p')$ and it follows that $n-p' \geq kq$, so each of these maps is an epimorphism, as claimed.

Now consider the diagram 
\begin{equation}\label{eq:AcyclicComparison}
\begin{gathered}
\xymatrix{
H_q(\St(\sigma_{p-1});\bZ) \ar[d] & H_q(\St(\sigma_p);\bZ) \ar[l] \ar[d]& H_q(\St(\sigma_{p+1});\bZ) \ar[l] \ar[d]\\
H_{q}(G_{n+1};\bZ) & H_{q}(G_{n+1};\bZ) \ar[l] & H_{q}(G_{n+1};\bZ). \ar[l]
}
\end{gathered}
\end{equation}
By observation (d), the top horizontal maps are alternatingly the 0-map and the stabilization map, which on the bottom line correspond to the 0-map and the identity map. 
The bottom sequence is exact in the middle. By the previous paragraph, for $p+q \leq i$ the two leftmost vertical maps are isomorphisms, and the rightmost vertical map is an epimorphism; hence the top row is also exact in the middle, which gives \EI{2}.

\medskip

We will now show the implication $({\bf E}_{I-1}) + ({\bf I}_{I-1}) \Rightarrow ({\bf I}_I)$, so let $i \leq I$ and $n$ satisfy $n \geq ki+1$, so $i \leq \tfrac{n-1}{k}$. We will show that 
\begin{itemize}
\item[\II{1}] $E^\infty_{0,i}=0$; 
\item[\II{2}] $E^2_{p,q}=0$ for $p+q=i+1$ with $q<i$; 
\item[\II{3}] $d^1\colon E^1_{1,i}\to E^1_{0,i}$ is the $0$-map. 
\end{itemize}
This will imply $({\bf I}_I)$, as condition \II{1} says that $E^1_{0,i}$ lies in the vanishing range of the spectral sequence, 
condition \II{2} says that potential sources of differentials $d^p$ to $E^1_{0,i}$ with $p\ge 2$, i.e. terms $E^2_{p,q}$ with $p+q=i+1$ and $q< i$, vanish,  
and condition \II{3} says that the differential $d^1:E^1_{1,i}\to E^1_{0,i}$ is the zero map. 
The leaves the injectivity  of $d^1:E^1_{0,i}\to E^1_{-1,i}$ as the only way to kill $E^1_{0,i}$ by the $E^\infty$-term. 

Conditions \II{1} and \II{2} follow from the same argument as above. For \II{1}, we need $i\le \frac{n-1}{k}$, which we have assumed. 
For \II{2}, the leftmost two vertical maps in \eqref{eq:AcyclicComparison} are isomorphisms if $n-p \geq kq+1$, which holds when $p+q=i+1$ with $q<i$ and $n \geq ki+1$ as long as $k\geq 2$. The rightmost vertical map in \eqref{eq:AcyclicComparison} is an epimorphism if $n-p-1 \geq kq$, which holds under the same conditions.

For condition \II{3}, the differential is $d^1$ is the difference of two maps which by the explanation after observation (d) above are equal: thus this differential vanishes.

\medskip

\noindent\textbf{Abelian coefficients case.} The strategy in this case is identical, and all steps of the proof go through without change except \II{3}. However, as the ranges involved are different we repeat the major points of the argument. We want to show that the differential 
\begin{equation}\label{eq:AbCoeffDiff}
d^1\colon E^1_{0,i}=H_i(G_{n};M)\rar E^1_{-1,i}=H_i(G_{n+1};M)
\end{equation}
is surjective when $n\ge k-2+ki$ and injective when $n\ge k+ki$, so consider

\vspace{1ex}

\noindent \Ei{I} The map \eqref{eq:AbCoeffDiff} is epi for all $i \leq I$ and all $n$ such that $n \geq k-2+ki$.

\vspace{1ex}

\noindent \Ii{I} The map \eqref{eq:AbCoeffDiff} is iso for all $i \leq I$ and all $n$ such that $n \geq k+ki$.

\vspace{1ex}

By stability with trivial coefficients, which we have already proved, the stabilisation map $H_1(G_n;\bZ) \to H_1(G_\infty;\bZ)$ is surjective as long as $n \geq k$. As $M$ is an abelian $G_\infty$-module, so a $H_1(G_\infty;\bZ)$-module, the map on coinvariants
$$H_0(G_n;M) = M_{G_n} = M_{H_1(G_n;\bZ)} \lra M_{H_1(G_\infty;\bZ)} = M_{G_\infty} = H_0(G_\infty;M)$$
is always an epimorphism, and by the above is an isomorphism as long as $n \geq k$. In particular it is an epimorphism as long as $n \geq k-2$ and an isomorphism as long as $n \geq k$. Thus the statements \Ei{0} and \Ii{0} hold. 

\medskip

To show the implication $({\bf E}_{I-1}) + ({\bf I}_{I-1}) \Rightarrow ({\bf E}_I)$ let $i \leq I$ and $n$ satisfy $n \geq k -2 + ki$, or alternatively $i \leq \tfrac{n-k+2}{k}$. Surjectivity of the map \eqref{eq:AbCoeffDiff} follows from the statements \EI{1} and \EI{2} just as before. Condition \EI{1} holds as $i-1 \leq \tfrac{n-2k+2}{k} \leq \tfrac{n-1}{k}$ for $k \geq 2$. We prove condition \EI{2} as before: 
We need to show that if $p+q=i$ with $q < i$ then $E^2_{p,q}=0$. 
Using that conjugation by $h_i$ is the identity on $H_i(G_i;M)$, the analogue of the diagram \eqref{eq:AcyclicComparison} with coefficient modules commutes,  where we take the differential $d^1$ on the top line and alternatingly the identity and the 0--map on the bottom line. 
By  \Ii{I-1} and \Ei{I-1}, in this diagram the two leftmost vertical maps are isomorphisms as long as $n-p \geq k+kq$, and the rightmost vertical map is surjective as long as $n-p-1 \geq k-2+kq$. Note that the first condition implies the second. 
Our assumptions show that $n \geq k-2+kp+kq$ and so $n-p \geq k-2+(k-1)p+kq$, but as $p\geq 1$ and $k \geq 3$ we have $-2+(k-1)p \geq 0$ and so the required condition holds. It follows that $E^2_{p,q}=0$, as required. 

\medskip

To show the implication $({\bf E}_{I}) + ({\bf I}_{I-1}) \Rightarrow ({\bf I}_I)$ let $i \leq I$ and $n$ satisfy $n \geq k+ki$, or alternatively $i \leq \tfrac{n-k}{k}$. Condition \II{1} is verified as $i \leq \tfrac{n-k}{k} \leq \tfrac{n-1}{k}$ for $k \geq 1$. For Condition \II{2} we proceed as for \EI{2}. Supposing that $p+q = i+1$ with $q < i$ (so $p\geq 2$), we wish to show that $E^2_{p,q}=0$. We have $k(p+q-1) = ki \leq n-k$, so in particular $p + kq  \leq n-k$ as $p\geq 2$ and $k \geq 2$. 
Firstly this shows that $n-p \geq k+kq$ and so by $({\bf I}_{I-1})$ the leftmost two vertical maps in \eqref{eq:AcyclicComparison} are isomorphisms; secondly it shows that $n-p-1 \geq k-1+kq \geq k-2+kq$ so by \Ei{I-1} the rightmost vertical map in \eqref{eq:AcyclicComparison} is surjective: thus the top sequence is exact in the middle.

Condition \II{3} states that the differential
\begin{equation}\label{eq:diff}
d^1 : E^1_{1,i} = H_i(\St(\sigma_1);M) \lra E^1_{0,i} = H_i(\St(\sigma_0);M)
\end{equation}
is zero. This differential is the difference of two maps: $d^1=d_0-d_1$, which differ by multiplication by the element $h_1$ from observation (d) which centralises $\St(\s_1)$ and takes $d_1\s_1$ to $\s_0$. 
Consider the element $g\in G_{n+1}$ defined by 
$$g := {A \oplus X^{\oplus n-2}} \oplus b_{X,X}^{-1} \oplus X\colon  A \oplus X^{\oplus n+1} \lra A \oplus X^{\oplus n+1}.$$
We have that $g\in \Fix(\iota_{A\op X^{\op n}}\op X)=\Fix(\s_0)$, hence we also have $g\circ h_1\circ d_1\s_1=\s_0$. The element $g$  does not centralise $\St(\s_1)$, but it does centralise $\St(\s_2)$, as does $g\circ h_1$. The element $g\circ h_1$ has the additional property that it becomes trivial in the abelianisation $H_1(G_{n+1};\bZ)$, so acts trivially on $M$. 
Indeed,  the braid relation
\begin{equation}\label{braidrelation}
(b_{X,X} \oplus X) \circ (X \oplus b_{X,X}) \circ (b_{X,X} \oplus X) = (X \oplus b_{X,X}) \circ (b_{X,X} \oplus X) \circ (X \oplus b_{X,X})
\end{equation}
abelianises to show that
$$[b_{X,X} \oplus X] = [X \oplus b_{X,X}] \in H_1(\Aut(X^{\oplus 3});\bZ),$$
from which it follows that $[g \circ h_1] = 0 \in H_1(\Aut(A \oplus X^{\oplus n+1});\bZ)=H_1(G_{n+1};\bZ)$.
Hence the two compositions
\begin{equation*}
\xymatrix{
H_i(\St(\sigma_2);M) \ar[rr]^-{(inc, id_M)_*} & & H_i(\St(\sigma_1);M) \ar@/^/[rr]^-{(inc, id_M)_*} \ar@/_/[rr]_-{c_{(g\circ h_1)} \circ (inc, id_M)_*} & & H_i(\St(\sigma_0);M)
}
\end{equation*}
are equal, by the same argument as for $h_1$ in the case of trivial coefficients, using now that $g\circ h_1$ acts trivially on $M$. By $({\bf E}_{I})$ the left-hand map is surjective as long as $n-2 \geq k-2+ki$, which holds as we have assumed that $n \geq k+ki$, so it follows that the two right-hand maps are equal, and hence that their difference is zero. This finishes the proof of the theorem. 
\end{proof}

\begin{rem}\label{deg0rem}
The proof of homological stability given above for trivial coefficients $M=\Z$ actually also applies to coefficients given by a fixed abelian group $M$ equipped with compatible actions of the groups $G_n$ such that $\Aut(X\op X)\inc G_n=\Aut(A\op X^{\op n-2}\op X\op X)$ acts trivially (and hence the element $h_1$ used in the proof acts trivially). The action of $G_n$ on $M$ need not be trivial if $A\neq 0$. Examples of such coefficient systems are {\em degree 0} coefficient systems as defined in the next section (see Definition~\ref{fdegcoef}). 
\end{rem}

The case $A=0$ of Theorem \ref{abstabthm} enjoys an additional feature, and to explain it, we introduce a construction. If $G$ is a group and $M$ a $G$-module given by a homomorphism $\phi : G \to \mathrm{Aut}_{\bZ\operatorname{-Mod}}(M)$, then there is an action of the subgroup 
$$Z(G,M) := \phi^{-1}(Z(\phi(G))) \triangleleft G$$
on $H_*(G;M)$ given in terms of the standard bar complex by
\begin{align*}
Z(G,M) \times C_p(G;M) & \lra C_p(G;M)\\
(z, [g_1 \vert \cdots \vert g_p \vert m]) & \longmapsto [g_1 \vert \cdots \vert g_p \vert \phi(z)(m)].
\end{align*}
This defines a chain map precisely because $\phi(z)$ centralises $\phi(G)$. For each $z \in Z(G,M)$ write $\tau_z : H_*(G;M) \to H_*(G;M)$ for the induced isomorphism. Note that a $G$-module is abelian precisely when $Z(G,M)=G$, in which case $\tau$ defines an action of the entire group $G$ on $H_*(G;M)$, and on $H_*(K;M)$ for any $K \leq G$.

\begin{cor}\label{cor:TrivAct} 
In the set-up of Theorem \ref{abstabthm} with $A=X^{\op a}$, the action of $\Aut(X^{\oplus \infty})$ on $H_i(\Aut(X^{\oplus a+n});M)$ via $\tau$ is trivial in degrees $i \leq \tfrac{n-k}{k}$.
\end{cor}
\begin{proof}
Let $g \in \Aut(X^{\oplus \infty})=G_\infty$. By Theorem 3.1, the map $H_1(G_k;\bZ) \to H_1(G_\infty;\bZ)$ is surjective, so we may choose an $h \in G_k:=\Aut(X^{\op a+k})$ such that after stabilising $[h]=[g] \in H_1(G_\infty;\bZ)$. Consider the sum map
$$\op: G_n \times G_k \lra G_{n+a+k}.$$
We have that $X^{\op a+n} \op G_k$ centralises $G_n \op X^{\op a+k}$. There is a commutative diagram
\begin{equation*}
\xymatrix{
H_i(G_n;M) \ar[r]^-{\tau_g} \ar[d]^-{- \op X^{\op a+k}}& H_i(G_n;M) \ar[d]^-{- \op X^{\op a+k}}\\
H_i(G_{n+a+k};M) \ar[r]^-{\tau_{g}} & H_i(G_{n+a+k};M).
}
\end{equation*}
Now the map $\tau_{g} : H_i(G_{n+a+k};M) \to H_i(G_{n+a+k};M)$ is equal to $\tau_{X^{\op a+n}\op h}$, as $g$ and $h$ induce the same automorphism of $M$ so $g$ and $X^{\op a+n}\op h=b_{X^{\op a+ k},X^{\op a+n}}\circ (h\op X^{\op a+n})\circ b_{X^{\op a+k},X^{\op a+n}}^{-1}$ do too. But on the bar complex we have
\begin{align*}
\tau_{X^{\op a+n}\op h} \circ (-\op X^{a+k}) ([g_1 \vert \cdots \vert g_i \vert m]) &= \tau_{X^{\op a+n}\op h}([g_1 \op X^{a+k}  \vert \cdots \vert g_i \op X^{a+k} \vert m])\\
&= [g_1 \op X^{a+k} \vert \cdots \vert g_i \op X^{a+k} \vert (X^{\op a+n}\op h) \cdot m]
\end{align*}
which we may write as $[(X^{\op a+n}\op h) (g_1\op X^{a+k}) (X^{\op a+n}\op h)^{-1} \vert \cdots \vert (X^{\op a+n}\op h) (g_i \op X^{a+k}) (X^{\op a+n}\op h)^{-1} \vert (X^{\op a+n}\op h) \cdot m]$ because $X^{\op a+n}\op G_k$ centralises $G_n \op X^{\op a+k}$. But this simply agrees with $-\op X^{\op a+k}$ followed by the inner action of $G_{n+a+k}$ on $H_*(G_{n+a+k};M)$, which is the identity. Thus the square above still commutes when the lower map is replaced by the identity, but the vertical maps are isomorphisms as long as $i \leq \tfrac{n-k}{k}$ and so in this range of degrees the top map is also the identity.
\end{proof}

\begin{rem}
One can give a different proof of Theorem \ref{abstabthm} in the case $A=0$, in which Corollary \ref{cor:TrivAct} is proved in parallel. The point is that the $d^1$-differentials in the spectral sequence used in the proof of Theorems \ref{stabthm} and \ref{abstabthm} are an alternating sum of the contributions of the various face maps, and these can be shown to differ from each other by postcomposition with certain $\tau_g$. Thus in the range where all the $\tau_g$ act as the identity the $d^1$-differentials out of odd columns become zero, and those out of even columns may be identified with stabilisation maps.
\end{rem}

\subsection{Subgroups containing the commutator subgroup}\label{sec:Commutator}

An extreme example of an abelian $G$-module is the group ring of the abelianisation of $G$, $\bZ[H_1(G;\bZ)]$. By Shapiro's lemma we have an isomorphism
$$H_*(G;\bZ[H_1(G;\bZ)]) \cong H_*(G';\bZ),$$
identifying the homology of $G$ with coefficients in $\bZ[H_1(G;\bZ)]$ with the homology of the commutator subgroup $G' \leq G$. Thus when the hypotheses of Theorem \ref{abstabthm} apply we can deduce that the commutator subgroups $\Aut(A \oplus X^{\oplus n})'$ exhibit homological stability with trivial coefficients.

\begin{cor}\label{cor:CommutatorSubGp}
With the hypotheses of Theorem \ref{abstabthm}, the map
\begin{equation*}
H_i(\Aut(A \oplus X^{\oplus n})';\bZ) \lra H_i(\Aut(A \oplus X^{\oplus n+1})';\bZ)
\end{equation*}
is an epimorphism for $i \leq \tfrac{n-k+2}{k}$ and an isomorphism for $i \leq \tfrac{n-k}{k}$.
\end{cor}

\begin{rem}
Suppose that a sequence of groups $V_0 \to V_1 \to V_2 \to \cdots$ is such that the groups $H_1(V_n;\bZ)$ stabilise, and the commutator subgroups $V'_n$ exhibit homological stability with trivial coefficients. Then the $V_n$ exhibit homological stability with all abelian coefficients.

To see this, let $M$ be an abelian $V_\infty$-module and suppose that $n$ is large enough for $H_1(V_n;\bZ) \to H_1(V_\infty;\bZ)$ to be an isomorphism. We have a group extension
$$1 \lra V'_n \lra V_n \lra H_1(V_\infty;\bZ) \lra 1,$$
and similarly for $V_{n+1}$, with a map between them, and thus we obtain a relative Lyndon--Hochschild--Serre spectral sequence
$$H_p(H_1(V_\infty;\bZ) ; H_q(V'_{n+1}, V'_n ; M)) \Longrightarrow H_{p+q}(V_{n+1}, V_n;M).$$
But $M$ is a trivial $V'_\infty$-module, and $V_n'$ has stability with trivial coefficients, so the groups $H_q(V'_{n+1}, V'_n ; M)$ vanish in a range. Thus $H_{p+q}(V_{n+1}, V_n;M)$ vanishes in the same range.
\end{rem}

More generally, let ${U} \leq H_1(\Aut(A \oplus X^{\oplus \infty});\bZ)$ be a subgroup. Then we may define groups $\Aut^U(A \oplus X^{\oplus n})$ by the pullback
\begin{equation*}
\xymatrix{
\Aut^U(A \oplus X^{\oplus n}) \ar@{^(->}[r] \ar[d]& \Aut(A \oplus X^{\oplus n}) \ar[d]\\
U \ar@{^(->}[r]& H_1(\Aut(A \oplus X^{\oplus \infty});\bZ),
}
\end{equation*}
and $-\oplus X$ defines stabilisation maps $\Aut^U(A \oplus X^{\oplus n}) \to \Aut^U(A \oplus X^{\oplus n+1})$. The group ring
$$M_U := \bZ[H_1(\Aut(A \oplus X^{\oplus \infty});\bZ)/U]$$
is an abelian $\Aut(A \oplus X^{\oplus \infty})$-module, and as an $\Aut(A \oplus X^{\oplus n})$-module it can be identified with $\bZ[\Aut(A \oplus X^{\oplus n})/\Aut^U(A \oplus X^{\oplus n})]$ by the above pullback square. Thus Shapiro's lemma shows that $H_*(\Aut(A \oplus X^{\oplus n});M_U) \cong H_*(\Aut^U(A \oplus X^{\oplus n});\bZ)$.

\begin{cor}\label{cor:ParameterSubGp}
With the hypotheses of Theorem \ref{abstabthm}, and a fixed subgroup ${U} \leq H_1(\Aut(A \oplus X^{\oplus \infty});\bZ)$, the map
\begin{equation*}
H_i(\Aut^U(A \oplus X^{\oplus n});\bZ) \lra H_i(\Aut^U(A \oplus X^{\oplus n+1});\bZ)
\end{equation*}
is an epimorphism for $i \leq \tfrac{n-k+2}{k}$ and an isomorphism for $i \leq \tfrac{n-k}{k}$.
\end{cor}

\subsection{Stable homology}\label{sec:StabHomology}

Theorems \ref{stabthm} and \ref{abstabthm} interact very nicely with the technique of ``group-completion". More precisely, for a pre-braided homogeneous category $(\mathcal{C}, \oplus, 0)$ and an object $X$ of $\mathcal{C}$, we may consider the braided monoidal groupoid $\mathcal{X} \subset \mathcal{C}$ having objects $\{X^{\oplus n} \, \vert \, n \in \bN\}$ and all isomorphisms between them. 

Supposing that no two powers of $X$ are isomorphic (which is the case for example if $X \neq 0$ and $\mathcal{C}$ satisfies the cancellation property), the geometric realisation of the nerve of the category $\mathcal{X}$ is
$$\vert \mathcal{X} \vert = \coprod_{n \geq 0} B\Aut(X^{\oplus n}).$$
The monoidal structure on $\mathcal{X}$ gives $\vert \mathcal{X} \vert$ the structure of a topological monoid, and the braided structure makes this topological monoid be homotopy commutative. (In fact, the braided structure makes $\vert \mathcal{X} \vert$ into an $E_2$-space, but we shall not use this here.)

The topological monoid $\vert \mathcal{X} \vert$ is thus suitable input for the group-completion theorem. In the form due to McDuff--Segal \cite{McDuff-Segal}, this identifies the homology of the mapping telescope
$$\mathrm{hocolim}(\vert \mathcal{X} \vert \overset{-\oplus X}\to \vert \mathcal{X} \vert \overset{-\oplus X}\to \vert \mathcal{X} \vert \overset{-\oplus X}\to \cdots)$$
with that of the space $\Omega B_\oplus \vert \mathcal{X} \vert$, the homotopical group-completion of the monoid $\vert \mathcal{X} \vert$. Restricting to individual path components, it thus identifies the homology of $B\Aut(X^{\oplus \infty})$ with that of the basepoint component $\Omega_0 B_\oplus \vert \mathcal{X} \vert$.

Recent refinements of the technique of McDuff--Segal, due to Miller--Palmer \cite{MillerPalmer} and the first author \cite{RWGC}, show that the comparison map
$$\mathrm{hocolim}(\vert \mathcal{X} \vert \overset{-\oplus X}\to \vert \mathcal{X} \vert \overset{-\oplus X}\to \vert \mathcal{X} \vert \overset{-\oplus X}\to \cdots) \lra \Omega B_\oplus \vert \mathcal{X} \vert$$
is more than just a homology equivalence: it is in fact an acyclic map. Restricted to basepoint components it follows that the comparison map
$$B\Aut(X^{\oplus \infty}) \lra \Omega_0 B_\oplus \vert \mathcal{X} \vert$$
is an acyclic map. Any abelian $\Aut(X^{\oplus \infty})$-module $M$ gives a local coefficient system on $B\Aut(X^{\oplus \infty})$ which is pulled back from one on $\Omega_0 B_\oplus \vert \mathcal{X} \vert$, which we give the same name, and it follows from acyclicity that the induced map
$$H_*(B\Aut(X^{\oplus \infty});M) \lra H_*(\Omega_0 B_\oplus \vert \mathcal{X} \vert;M)$$
is an isomorphism.

In particular, for a fixed subgroup
$${U} \leq H_1(\Aut(A \oplus X^{\oplus \infty});\bZ) \cong H_1(\Omega_0B_\oplus \vert \mathcal{X} \vert;\bZ) \cong \pi_1(\Omega_0B_\oplus \vert \mathcal{X} \vert)$$
this discussion identifies the homology of the subgroup $\Aut^U(X^{\oplus \infty})$ defined in the last section with that of the covering space of $\Omega_0 B_\oplus \vert \mathcal{X} \vert$ corresponding to $U$. Hence the homology of the commutator subgroup $\Aut(X^{\oplus \infty})'$ is that of the universal cover $\widetilde{\Omega_0 B_\oplus \vert \mathcal{X} \vert}$, which implies that $\Aut(X^{\oplus \infty})'$ is perfect (cf.\ \cite[\S 3]{RWGC} for an elementary proof that such groups are perfect).

\begin{rem}
A similar discussion can be made to treat the groups $\Aut(A \oplus X^{\oplus n})$. In this case the right action of $\vert \mathcal{X} \vert$ on $\vert \mathcal{A} \vert := \coprod_{n \geq 0} B\Aut(A \oplus X^{\oplus n})$ should be used, and the stable homology of $\Aut(A \oplus X^{\oplus n})$ will be computed in terms of the homotopy fibre $\mathrm{hofib}(B(\vert \mathcal{A} \vert; \vert \mathcal{X}\vert; *) \to B(*; \vert \mathcal{X}\vert; *))$ of the induced map on two-sided bar constructions. We leave the details to the interested reader.
\end{rem}

\section{Stability with twisted coefficients}\label{sec:twist}

Let  $(\C,\op,0)$ be a monoidal category. The main result of this section is a stability theorem for homology of the automorphism groups of objects of $\C$ with twisted coefficients. Before stating the theorem, we will need to introduce the upper and lower suspension maps, the coefficient systems we work with, and the crucial notion of kernel and
cokernel of a coefficient system (see Definition~\ref{kercoker}). This section is a generalisation of the work of Dwyer \cite{Dwy80} and van der Kallen \cite{vdK80}, who worked with general linear groups, and of Ivanov \cite{Iva93}, who worked with mapping class groups. 

We will assume in this section that $\C$ is {\em pre-braided} in the sense of Definition~\ref{prebraidDef}. 

\subsection{Coefficient systems}\label{coefsec}

Let $\C_{A,X}$ denote the full subcategory of $\C$ whose objects are the objects $A\op X^{\op n}$ for all $n\ge 0$. We are interested in the stabilisation map $\Aut(A\op X^{\op n})\to \Aut(A\op X^{\op n+1})$, which is associated to the {\em upper suspension functor} 
$$\Si^X:=- \oplus X\colon \C_{A,X}\lra \C_{A,X},$$
the functor taking any object $B$ to $B\op X$ and any morphism $f:A\to B$ to $f\op X:A\op X\to B\op X$. 
There is an associated {\em upper suspension map}   
$$\s^X:=A\op X^{\op n}\op \iota_X: A\op X^{\op n}\rar A\op X^{\op n}\op X =A\op X^{\op n+1},$$
which defines a natural transformation 
 $\s^X:Id \Rightarrow \Si^X$. 

\begin{Def}\label{def:CoeffSys}
A {\em coefficient system for $\C$ at $(A,X)$} is a functor 
$$F\colon \C_{A,X} \lra  \A$$ 
from $\C_{A,X}$ to an abelian category $\A$. 
\end{Def}
We will be mostly interested in the case where $\A$ is the category of $\Z$-modules, or more generally the category of $R$-modules for some ring $R$.

A coefficient system $F$ associates to the sequence of groups $G_n=\Aut(A\op X^{\op n})$  a sequence of left $G_n$-modules $F_n=F(A\op X^{\op n})$, together with $G_n$-equivariant maps $F_n\to F_{n+1}$ induced by the upper suspension. Sequences of $G_n$-modules obtained this way have the property that the subgroup
$$\Aut(X^{\op m}) \inc \Aut(A \oplus X^{\op n} \oplus X^{\op m})$$
acts  trivially on the image of $F_n$ in $F_{n+m}$. In particular, in the case where $G_n=GL_n(\Z)$ with $\C$  the category with objects the finitely generated free abelian groups and morphisms the injective homomorphisms equipped with a choice of splitting (cf.\ Section~\ref{GLnsec}), it is {\em central} in the sense of \cite[Sec.~2]{Dwy80}. 
Conversely, we have the following: 

\begin{prop}\label{repfunct} 
Let $(\C, \op, 0)$ be a pre-braided homogeneous category and let $(A,X)$ be a pair of objects in $\C$. Suppose that $A\op X^{\op n}\not \cong A\op X^{\op m}$ in $\C$ when $n\neq m$. Let $G_n=\Aut(A\op X^{\op n})$. 
Suppose $M_n$, for $n\ge 0$, is a sequence of $G_n$-modules, together with $G_n$-equivariant maps $\s_n:M_n\to M_{n+1}$ satisfying that $\Aut(X^{\op
  m})$ acts trivially on the image of $M_n$ in $M_{n+m}$. Then there exists a functor 
$$F:\C_{A,X}\lra \Z\operatorname{-Mod}$$ 
satisfying that $F(A\op X^{\op n})=M_n$ as a $G_n$-module, and that $\s_n=F({A\op X^{\op n}}\op \iota_{X})$.  
\end{prop}

\begin{proof}
We define $F$ on objects by setting $F(A\op X^{\op n}):=M_{n}$. 
On morphisms, we set the action of $\Aut(A\op X^{\op n})=G_n$ on $F(A\op X^{\op n})=M_n$ to be the $G_n$-module structure of $M_n$, and we set 
$$F({A\op X^{\op n}}\op \iota_{X^{\op m}})=\s_{n+m-1}\circ\dots\circ \s_n\colon F(A\op X^{\op n})\rar F(A\op X^{\op n+m}).$$
Let $f:A\op X^{\op m}\to A\op X^{\op n}$ denote a general morphism of $\C_{A,X}$. We must have $m\le n$ for $f$ to exist, as the existence of such an $f$ with $m>n$ would imply, by the fact that endomorphisms are isomorphisms in $\C$, that $A\op X^{\op n}$ is isomorphic to $A\op X^{\op m}$, contradicting our hypothesis. We can always write such a morphism as a composition 
$$f= \phi \circ ({A\op X^{\op m}}\op \iota_{X^{\op n-m}})$$ 
for some  $\phi\in \Aut(A\op X^{\op n})=G_{n}$, by axiom {H1}. We propose to define $F(f)$ using such a decomposition of $f$; we must check is that this is consistent.

First note that by Proposition~\ref{H2sym} the automorphism $\phi$ is only determined up to an automorphism of $X^{\op n-m}$. As we have assumed that the automorphisms of $X^{\op n-m}$ act trivially on the image of $M_{m}$ inside $M_{n}$, we see that $F(f)$ is in fact independent of that choice. 

It remains to show that this proposed definition respects composition. This follows from the commutativity of the following diagram, 
\begin{equation*}
\xymatrix{
A\op X^{\op m} \ar[rr]^-{A\op X^{\op m}\op \iota_{X^{\op k}}}\ar[drr]_-{A\op X^{\op m} \op \iota_{X^{\op k+l}}} && A\op X^{\op m+k} \ar[r]^-\phi\ar[d]^-{A\op X^{\op m+k}\op \iota_{X^{\op l}}} & A\op X^{\op m+k} \ar[d]^-{A\op X^{\op m+k}\op \iota_{X^{\op l}}} \\
&& A\op X^{\op m+k+l} \ar[dr]_-{\psi\circ \Si^l\phi}\ar[r]^-{\Si^l \phi}  & A\op X^{\op m+k+l} \ar[d]^{\psi}\\
&&& A\op X^{\op m+k+l},
}
\end{equation*}
and the $G_n$-equivariance of $\s_n$.
\end{proof}

Many sequences of $G_n$-modules will come directly as functors from the category $\C_{A,X}$, but some sequences of modules do not obviously come from a functor. In this case the above proposition gives an easy way to check whether or not they do. The following example---which is one such less obvious example---was suggested to us by Christine Vespa. 

\begin{ex}[The Burau representation]\label{Bureauex} 
Let $\beta=\coprod_{n\ge 0} \beta_n$ be the groupoid with objects the natural numbers and automorphism groups the braid groups. This is a braided
monoidal groupoid satisfying the hypothesis of Theorem~\ref{universal}. Let $U\beta$ denote the associated homogeneous category. The category $U\beta$
is described in detail in Section~\ref{braidex}, but we will not need that description for our current purpose. 

We take here $X=1$ and $A=0$, so $\C_{A,X}=U\beta$.

The (unreduced) Burau representation is a sequence of representations for the braid groups: for each $n$, let $M_n=(\Z[t,t^{-1}])^n$ denote the rank $n$ free
$\Z[t,t^{-1}]$-module. Then the standard $i$th Artin generator $\s_i$ of $\beta_n$ acts on $M_n$ via the following matrix: 
$$\left(
\begin{array}{c|cc|c}
I_{i-1}& &&  \\
\hline
& 1-t & t& \\
& 1& 0& \\
\hline
&&& I_{n-i-1}
\end{array}
\right)$$
where $I_j$ denotes the $(j\x j)$-identity matrix. 
(See e.g. \cite[Sec.~1.3]{Tur02}.)

The natural map $(\Z[t,t^{-1}])^n\to (\Z[t,t^{-1}])^{n+1}$ defines a $\beta_n$-equivariant module map $M_n\to M_{n+1}$, and it follows from the above
description of the action that $\beta_m$ acts trivially on the image  $(\Z[t,t^{-1}])^n$ inside $(\Z[t,t^{-1}])^{n+m}$. Hence by Proposition~\ref{repfunct}, the Burau representation is part of a functor
 $$F:U\beta \ \rar \ \Z[t,t^{-1}]\operatorname{-Mod},$$
that is, it defines a coefficient system in our set-up. 
\end{ex}

\subsection{Lower suspension, and suspension of coefficient systems}\label{sec:LowerSusp}

We will also use in this section the stabilisation map associated with the {\em lower suspension by $X$},
$$\s_X:= (b_{X,A} \op X^{\op n}) \circ (\iota_X \op A  \op X^{ \op n}): A\op X^{\op n}\rar A\op X\op X^{\op n}.$$
As $\C$ is assumed to be pre-braided, these two suspension maps are related  via 
\begin{equation}\label{lowups}
\s_X=(b_{X,A} \op X^{\op n}) \circ (b_{A \op X^{\op n},X})\circ \s^X.
\end{equation}

Just as for the upper suspension, the lower suspension defines a natural transformation from the identity to an endofunctor of $\C$. Let 
$$\Si_X \colon \C_{A,X} \lra \C_{A,X}$$
be the functor taking $A \oplus X^{\op n}$ to $A \oplus X^{\op n+1}$, and taking a morphism $f : A \oplus X^{\op n} \to A \oplus X^{\op k}$ to
$$\Si_X(f) : A \oplus X^{\op n+1} \overset{b_{X,A}^{-1} \op X^{\op n}}\rar X \op A \oplus X^{\op n} \overset{X \op f}\lra X \op A \oplus X^{\op k} \overset{b_{X,A} \oplus X^{\op k}}\rar A \oplus X^{\op k+1}.$$
This defines a functor, and $\s_X$ defines a natural transformation $\s_X : Id\Rightarrow \Si_X$. Note that the functors $\Sigma^X$ and $\Sigma_X$ commute. Using the axioms of a pre-braided monoidal category one can moreover show that,
for any $f\in \Aut(A\op X^{\op n})$,  the upper and lower suspension satisfy 
\begin{equation}\label{eq:UpperLowerConj2}
\Sigma_X(f) = (b_{X,A} \op X^{\op n}) \circ (b_{A\op X^n, X}) \circ \Sigma^X(f) \circ (b_{A\op X^n, X}^{-1}) \circ (b_{X,A}^{-1} \op X^{\op n}).
\end{equation}
One can also check that for any morphism $f:A\op X^{\op n}\to A\op X^{\op k}$ in $\C$, we also have that $\Sigma_X(f) = (A\op b_{X,X^{\op k}}^{-1}) \circ \Sigma^X(f) \circ (A\op b_{X,X^{\op n}})$, though we will not use this. 

\medskip

Given a coefficient system $F:\C_{A,X}\to \A$, we can get a new one by precomposing with an endofunctor of $\C_{A,X}$. We define the {\em suspension} of $F$ with respect to $X$ as the composite functor $$\Si F:=F\circ \Si_X.$$ 
Given $F$ and its suspension $\Si F$, we can define two additional coefficient systems.

\begin{Def}\label{kercoker}
For a coefficient system $F:\C_{A,X}\to \A$, we define {\em $\ker F$} and {\em $\coker F$} to be the kernel and cokernel of the natural transformation
$F(\s_X):F\to \Si F=F\circ \Si_X$. 
\end{Def}
The kernel and cokernel of a functor again define coefficient systems.

\begin{lem}\label{suspker} 
Let $F:\C_{A,X}\to \A$ be a coefficient system. Then the  kernel and cokernel of $\Si^jF$ are isomorphic to the $j$-fold suspension of the kernel and cokernel of $F$.  
\end{lem}

\begin{proof} We want to compare $\ker(\Si^jF)=\ker(F\circ \Si_X^j)$ with $\Si^j\ker F=(\ker F) \circ \Si_X^j$ and likewise for the cokernels. Consider the diagram of functors
$$\xymatrix{0\to (\ker F)\circ \Si_X^j \ar[r] & F \circ \Si_X^j \ar[rr]^-{F(\s_X)\circ \Si_X^j} \ar@{=}[d] && (F\circ \Si_X) \circ \Si_X^j \ar[d]^\eta\ar[r] & (\coker F)\circ \Si_X^j \ar[r] & 0\\
0\to \ker (F\circ \Si_X^j) \ar[r] & F \circ \Si_X^j \ar[rr]^-{F\circ \Si_X^j(\s_X)} &&  (F\circ \Si^j_X) \circ \Si_X \ar[r] & \coker (F\circ \Si_X^j) \ar[r] & 0}$$ 
where the rows are exact and where $\eta$ is the natural isomorphism defined by 
$$\eta_{A\op X^{\op n}}=F(A\op b_{X^{\op j},X}^{-1}\op X^{\op n}): F(A\op X\op X^{\op j}\op X^{\op n}) \rar F(A\op X^{\op j}\op X\op X^{\op n}).$$
The result follows from the commutativity of the middle square, which we check now on the object $A\op X^{\op n}$: By definition, $ F\circ \Si_X^j(\s_X)$ is $F$ applied to the composition
\begin{align*}
& (b_{X^{\op j}, A} \op X^{\op n+1}) \circ (X^{\op j} \op b_{X,A} \op X^{\op n}) \circ (X^{\op j} \op\iota_X \op A  \op X^{ \op n}) \circ (b_{X^{\op j}, A}^{-1} \op X^{\op n})\\
= &  \big((b_{X^{\op j}, A} \op X) \circ (X^{\op j} \op b_{X,A}) \circ (X^{\op j} \op\iota_X \op A) \circ (b_{X^{\op j}, A}^{-1})\big) \op X^{\op n}.
\end{align*}
Replacing $ (X^{\op j} \op\iota_X \op A)$ by $ (X^{\op j} \op b_{A,X}) \circ (X^{\op j} \op A  \op\iota_X)$ using the pre-braiding relation and commuting the right-most two maps, we can rewrite this composition as
\begin{align*}
\big((b_{X^{\op j}, A} \op X) \circ (X^{\op j} \op b_{X,A}) \circ  (X^{\op j} \op b_{A,X}) \circ (b_{X^{\op j}, A}^{-1} \op X) \circ (A\op X^{\op j} \op\iota_X)\big)\op X^{ \op n}.
\end{align*}
Using the braid relations, one checks that this is equal to 
\begin{align*}
\big((A\op b_{X^{\op j},X}^{-1}) \circ  ( b_{A,X}\op X^{\op j}) \circ (b_{A\op X^{\op j},X}) \circ (A\op X^{\op j} \op\iota_X)\big)\op X^{ \op n},
\end{align*}
which reduces to 
\begin{align*}
\big((A\op b_{X^{\op j},X}^{-1}) \circ  ( b_{A,X}\op X^{\op j}) \circ (\iota_X\op A\op X^{\op j})\big)\op X^{ \op n}.
\end{align*}
using the pre-braiding relation, which identifies, after applying $F$, with $\eta\circ F(\s_X)\circ \Si_X^j$. 
\end{proof}

\subsection{Coefficient systems of $G_\infty^{ab}$-modules}

In the set-up of the previous sections, we have a group $G_\infty$ obtained as the colimit of the groups $G_n$ along the upper suspension homomorphisms $\Si^X\colon G_n \to G_{n+1}$; we write $\Si^X_\infty : G_n \to G_\infty$ for the canonical homomorphism, and $G_\infty^{ab} = H_1(G_\infty;\bZ)$ for the abelianisation. Let $\mathbb{Z}[G_\infty^{ab}]\operatorname{-Mod}$ denote the (abelian) category of left $G_\infty^{ab}$-modules. If
$$F : \C_{A,X} \lra \mathbb{Z}[G_\infty^{ab}]\operatorname{-Mod}$$
is a coefficient system, we obtain for each $n$ a left $G_\infty^{ab}$-module $F_n = F(A \oplus X^{\op n})$, with a left action of $G_n$ via $G_\infty^{ab}$-module maps. Thus it has commuting left $G_n$ and $G_\infty^{ab}$ actions, which for now we write as $\cdot$ and $\ast$. Via the canonical map $G_n \overset{\Si^X_\infty}\to G_\infty \to G_\infty^{ab}$ these can be combined to give a modified $G_n$-module structure, via
\begin{align*}
G_n \times F_n & \lra F_n\\
(g, x) & \mapsto g \cdot (\Si^X_\infty (g) \ast x).
\end{align*}
This defines a new left $G_n$-module structure on $F_n$, and we write $F_n^\circ$ for this $G_n$-module. An alternative point of view, occasionally convenient, is to take
$$F_n^\circ = \mathbb{Z}[G_\infty^{ab}] \otimes_{\mathbb{Z}[G_\infty^{ab}]} F_n$$
with the diagonal $G_n$-action. We shall call $F_n^\circ$ the module obtained by \emph{internalising} $F_n$. These will almost never come from a coefficient system.

Finally, as $\mathbb{Z}[G_\infty^{ab}]$ is a commutative ring the left $G_n$-action on $F_n^\circ$ commutes with the left $\mathbb{Z}[G_\infty^{ab}]$-module structure. Hence $H_i(G_n ; F_n^\circ)$ is a left $\mathbb{Z}[G_\infty^{ab}]$-module for each $i$.

\begin{ex}[Constant coefficient systems] 
Let $M$ be a $G_\infty^{ab}$-module. There is an associated functor $F_M:\C_{A,X}\to \mathbb{Z}[G_\infty^{ab}]\operatorname{-Mod}$ taking the value $M$ on all objects and the identity on $M$ on all morphisms. The corresponding internalised coefficient system $F^\circ_M$ is just the abelian coefficient system $M$ as considered in Section~\ref{cstsec}. 
\end{ex}

\begin{ex}[Induced $G_\infty^{ab}$-modules]
If $F : \C_{A,X} \to \mathbb{Z}\operatorname{-Mod}$ is a coefficient system of abelian groups, and $U \leq G_\infty^{ab}$ is a subgroup, then $F(-) \otimes_{\mathbb{Z}} \mathbb{Z}[G_\infty^{ab}/U]$ is a coefficient system of $G_\infty^{ab}$-modules.
\end{ex}

\begin{ex}[Twisting by determinants]
Let $R$ be a ring and $\G = fR\operatorname{-Mod}$ be the groupoid of finitely generated $R$-modules. This is a symmetric monoidal groupoid under direct sum, and $U\G$  is pre-braided  monoidal (in fact symmetric monoidal) by Proposition~\ref{braidandsym}. The category $U\mathcal{G}$ is isomorphic to the category with objects the finitely generated $R$-modules and morphisms the injective $R$-module homomorphisms equipped with a choice of splitting (cf.\ Section~\ref{GLnsec}). 

If $R$ is commutative, then there are determinant maps $\det : GL_n(R) \to R^\times$, which assemble to give a homomorphism $\det: GL_\infty(R)^{ab} \to R^\times$. This makes any $R$-module into a $\Z[GL_\infty(R)^{ab}]$-module in a canonical way. Hence if $F : U\mathcal{G}_{A,X} \to R\operatorname{-Mod}$ is a coefficient system, 
we get an associated coefficient system of $GL_\infty(R)^{ab}$-modules $\bar F : U\G_{A,X} \lra \mathbb{Z}[G_\infty^{ab}]\operatorname{-Mod}$, which, internalised, give the sequence of modules  
$\bar F_n^\circ = F_n^{(1)} := R \otimes_R F_n$ where $GL_n(R)$ acts via the determinant on the first factor and via the functoriality of $F$ on the second. More generally, using the $k$th power of the determinant, $k \in \Z$, we obtain $GL_n(R)$-modules $F_n^{(k)}$.
\end{ex}

\begin{lem}\label{lem:TwistedModule}
The maps 
$$F(\sigma^X) : F_n \lra F_{n+1} \quad\textrm{and}\quad F(\sigma_X) : F_n \lra F_{n+1}$$ 
define left $\mathbb{Z}[G_\infty^{ab}]$-module maps $F_n^\circ \to F_{n+1}^\circ$ which are equivariant with respect to $\Sigma^X : G_n \to G_{n+1}$ and $\Sigma_X : G_n \to G_{n+1}$ respectively.
\end{lem}
\begin{proof}
The first case is easy: the map $\Sigma^X : G_n \to G_{n+1}$ commutes with the canonical map from each of these groups to $G_\infty$, in other words, $\Si^X_\infty \circ \Sigma^X = \Si^X_\infty: G_n\to G_{n+1}\to G_\infty$. The claim is then immediate from the formula defining the $G_n$-action on $F_n^\circ$.

In the second case, the map $\Sigma_X : G_n \to G_{n+1}$ does \emph{not} necessarily commute with the canonical map to $G_\infty$. However, by \eqref{eq:UpperLowerConj2} the homomorphisms $\Sigma^X, \Sigma_X : G_n \to G_{n+1}$ are conjugate, so in particular they become equal under $G_{n+1} \overset{\Si^X_\infty}\to G_\infty \to G_\infty^{ab}$. Now the formula defining the $G_n$-action on $F_n^\circ$ shows that the claim holds.
\end{proof}

\subsection{Finite degree coefficient systems}

Fix objects $A$ and $X$ of $\C$ and  recall from Section~\ref{coefsec} that $\C_{A,X}$ denotes the full subcategory of $\C$ generated by the objects of the form
$A\op X^{\op n}$.  
Recall also that for a coefficient system $F:\C_{A,X}\to \A$, the associated suspended functor $\Si F:=F\circ \Si_X:\C_{A,X} \to \A$ takes the value $F(A \op X^{\op n+1})$ at
$A \op X^{\op n}$, and that $\ker F$ and $\coker F$ denote the kernel and cokernel of the suspension map $F\to \Si F$.  We will restrict attention in this section to coefficient systems which satisfy a certain {\em polynomiality} condition. 
The following definition is a generalisation of van der Kallen's definition of degree $k$ functors in the case of general linear groups \cite[5.5]{vdK80}. 
(See  Remark~\ref{polrem} below for other, different though closely related, uses of polynomiality terminology in the literature.)

\begin{Def}\label{fdegcoef}
A coefficient system $F:\C_{A,X}\to \A$ has {\em degree $r<0$ at $N \in \bZ$ with respect to $X$} if $F(A\op X^{\op n})=0$ for all $n\ge N$. 

For $r\ge 0$, we define inductively that $F$ {\em has degree $r$ at $N \in \bZ$} if 
\begin{enumerate}[(i)]
\item the kernel of the suspension map $F\to \Si F$ has degree $-1$ at $N$, 

\item the cokernel is a coefficient system of degree $(r-1)$ at $N-1$.  
\end{enumerate}

Furthermore, we say that $F$ is a {\em split} coefficient system of degree $r$ at $N$ if 
\begin{enumerate}[(i)]
\item the suspension map $F \to \Sigma F$ is split injective in the category of coefficient systems,

\item the cokernel is a split coefficient system of degree $(r-1)$ at $N-1$. 
\end{enumerate}
where any coefficient system of degree $r<0$ at $N$ is split. 
\end{Def}

\begin{rem}
The suspension map $(\ker F) \to (\Sigma \ker F)$ for the functor $\ker
F$ is the zero map, as the diagram 
$$\xymatrix{
\ker F\ar[r]\ar@{^(->}[d]\ar[dr]^0 & (\Sigma \ker F) \ar@{^(->}[d]\\
F \ar[r] & (\Sigma F)}$$
commutes.  It follows that $ \ker (\ker F)\cong \ker F$. Hence the condition that $\ker F$ have degree $-1$ is equivalent to asking that $\ker F$ have strictly smaller degree than $F$.
\end{rem}

\begin{rem}
In \cite{vdK80}, van der Kallen assumes that $F \to \Sigma F$ is split injective and $\coker F$ is of degree $(r-1)$, where $F$ has degree $r<0$ if it is the trivial functor: his definition corresponds to the case $N=0$ and $F$ is split in our definition. 

On the other hand, if $F:\C_{A,X}\to \A$ is of degree $r$ at $N$, then the restriction of $F$ to $\C_{A\op X^{\op N},X}$ is of degree $r$ at $0$, so by replacing if necessary the pair $(A,X)$ by $(A\op X^{\op N},X)$, one could just consider finite degree coefficient systems at 0  as in \cite{vdK80}. 

However, the point of view taken here will give better stability range already in the case of constant coefficients. Indeed, if $F=M$ is a constant coefficient system at $N$, our proof will in effect be as if we considered the actual constant coefficient system and only talked about the groups $H_i(\Aut(A\op X^{\op n});M)$ for $n\ge N$. The stability range will thus depend on how large $n$ is. If we on the other hand replaced $A$ by $A\op X^{\op N}$, we would be considering the same groups, namely  $H_i(\Aut((A\op X^{\op N})\op X^{\op n-N});M)$ with $n-N\ge 0$, but now the stability range would depend on how large $n-N$ is. 
\end{rem}

\begin{lem}\label{suspdeg} 
Let $F:\C_{A,X}\to \A$ be a coefficient system. If $F$ is of degree $r$ at $N$, then for any $j\ge 1$, $\Si^jF$ is of degree $r$ at $(N-j)$. 
\end{lem}

\begin{proof}
We prove the lemma by induction on $r$. 
This can be checked immediately for $r<0$ as the suspension of the trivial coefficient system is of the same form, and if $F$ is trivial on $A\op X^{\op n}$ with $n\ge N$, then $\Si^jF$ is trivial on $A\op X^{\op n}$ with $n\ge N-j$. 
So assume $r\ge 0$. By Lemma~\ref{suspker}, we have $\ker (\Si^j F)\cong \Si^j(\ker F)$ and  $\coker (\Si^j F)\cong \Si^j(\coker F)$. 
Hence the kernel  of $\Si^j F$ is of degree $-1$ at $(N-j)$ 
and its cokernel is of degree $(r-1)$ at $(N-1-j)$ by induction, that is $\Si^j F$ is of degree $r$ at $(N-j)$. 
\end{proof}

There are many examples of finite degree coefficient systems. We give a few below. 

\begin{ex}
Any constant coefficient system  $F:\C_{A,X}\to \Z\operatorname{-Mod}$ has degree 0 at $0$. More generally, unravelling Definition \ref{fdegcoef} shows that a coefficient system $M$ having degree 0 at 0 is one which has $\sigma_X : M_n \to M_{n+1} = (\Sigma M)_n$ an isomorphism for all $n$. This identifies all the $M_n$, but if $A \neq 0$ then $M$ need \emph{not} be a constant coefficient system: all that is implied is that the subgroup $\Aut(X^{\op n}) \leq \Aut(A\op X^{\op n}) = G_n$ acts trivially on $M_n$. Thus for example $\Aut(A)$ may act non-trivially on $M_0$.
\end{ex}

\begin{ex}\label{Bureau2}
Recall from Example~\ref{Bureauex} the Burau representation of the braid groups. Here the homogeneous category is $U\beta$, associated to the braid
groups, and $F:U\beta\to \Z\operatorname{-Mod}$ is a coefficient system taking the value $\Z[t,t^{-1}]^n$ at the object $n$. We take $X=1$ and $A=0$. The map 
$F(n)\to \Si F(n)=F(n+1)$ is the inclusion of $\Z[t,t^{-1}]^n$ into $\Z[t,t^{-1}]^{n+1}$ as the rightmost $n$ factors. It has trivial kernel and constant cokernel
$\Z[t,t^{-1}]$. Hence the Burau representation defines a degree 1 coefficient system at $0$. 
\end{ex}

\begin{ex}
Let $(f\G,\x,e)$ denote the symmetric monoidal groupoid of finitely generated groups and their automorphisms, with monoidal structure
induced by direct product. The associated category $Uf\G$ is equivalent to the category with objects the finitely generated groups and morphisms from $G$ to $H$ is a pair $(f,K)$ for $f:G\inc H$ an injective homomorphism and $K\le G$ a subgroup satisfying that $H=K\x f(G)$. See Section~\ref{directprodsec} for further discussion of this example.

Now let $F:Uf\G\to \Z\operatorname{-Mod}$ be the functor taking $G$ to $H_k(G;\bZ)$. Consider $A=\{e\}$ the trivial group and $X=\Z$. One can check, using the description $H_k(\Z^m;\bZ)\cong \wedge^k \bZ^m$, that the restriction of $F$ to $Uf\G_{(\Z,\{e\})}$ is a degree $k$ coefficient system at $0$. 
\end{ex}

\begin{ex}
Let $(FI,\sqcup,\emp)$ be the category of finite sets and injections. ($FI=U\Sigma$ for $\Sigma$ the groupoid of finite sets and isomorphisms, with monoidal product induced by disjoint union---see Section~\ref{setex}.) Let  $F:FI\to \Z\operatorname{-Mod}$ be defined on objects by 
$$F(T)=\left\{\begin{array}{ll}\Z & |T|=m\\
0 & |T|\neq m,\end{array}\right.$$
send all isomorphisms to the identity map, and all non-isomorphisms to the zero map. Take $X=\{*\}$ and $A=\emp$. Then $F$ is a degree $-1$ functor with respect to $X$ at $m+1$. 
\end{ex}

Functors $F : FI \to R\operatorname{-Mod}$ are also known as {\em $FI$-modules}, and under this name have a well-developed theory. In the following example we freely use this theory: see e.g.\ \cite{ChurchEllenberg} for the necessary background.

\begin{exprop}\label{FIex}
A coefficient system $F : FI \to R\operatorname{-Mod}$ is of finite degree if it is presented in finite degree. 
More precisely, if $F$ is generated in degree $\leq k$ and related in degree $\leq d$, then it is of degree $k$ at $d+\min(k,d)$.
\end{exprop}
\begin{proof} 
Djament has shown \cite[Proposition 4.4]{Dja13} (see also \cite[Prop.~4.6]{ChurchEllenberg}) that an $FI$-module is generated in degree $\leq k$ if and only if it is {\em strong polynomial} of degree $\leq k$, that is, its $k$th iterated cokernel is trivial. It remains to show that $\ker(F)$ is bounded (that is, vanishes in large enough degree) if $F$ is presented in finite degree, and that the same holds for the kernel of its iterated cokernels.
This follows from Theorem 4.8 in \cite{ChurchEllenberg} and its proof: Equation (18) in the proof holds for $p=1$ and $a=j+1\ge 1$, and by the theorem, the middle term vanishes for $n > d+\min(k,d) -1-(j+1) + 1=d+\min(k,d)-j-1$, which in our terms means is of degree $-1$ at $d+\min(k,d)-j$.  Now equation (18) implies that the same holds for the right most term, which, in our notation and for $W=F$, is $\ker(\coker^j(F))$ (as $\ker F$ and $\coker F$ are denoted $H_1^D(F)$ and $D(F)=H_0^{D}(F)$ in that paper).  
\end{proof}

\begin{rem}[Relationship to polynomiality for functors]\label{polrem}
Finite degree coefficient systems are closely related to the notion of strong polynomial functors introduced by Djament--Vespa in \cite[Def.~1.5]{DjaVes13} and mentioned in the above example. The difference between the two notions is that there are no
requirements on the kernel of the map $F\to \Si F$ for strong polynomial functors, but on the other hand strong polynomial functors are required to have a trivial iterated cokernel with respect to stabilisation with respect to {\em all} objects $X$ in the category. Polynomial functors were originally defined using cross effects for functors with domain category a monoidal category with the unit 0 a null object (see \cite{EilMac54}). Proposition 2.3 in \cite{DjaVes13} shows that strong polynomiality can also be defined in terms of cross effects, and hence that it is a direct generalisation of the classical notion of polynomiality. Note that if 0 is a null object in $\C$, functors $\C\to \A$ are automatically split, as in that case we have a canonical factorisation $id : A\to X\op A\to A$, natural in $A$. The unit 0 will never be null in the categories we consider, but certain functors may factor through a larger category in which 0 is a null object. 
\end{rem}

\subsection{The stability theorem}

We will phrase our main stability theorem in terms of the vanishing of relative homology groups. We first recall the definition and basic properties of these groups and refer to \cite[3.8-12]{vdK80} for more details. 

\medskip

Following \cite[3.9]{vdK80} we let $\mathcal{R}ep$ be the category with objects pairs $(G, M)$ of a group $G$ and a left $G$-module $M$, in which the morphisms $(\phi, f) : (G, M) \to (G',M')$ consist of a group homomorphism $\phi: G \to G'$ and a $\phi$-linear map $M \to M'$. We then let $\mathcal{R}elRep$ be the arrow category of $\mathcal{R}ep$, with objects the morphisms of $\mathcal{R}ep$ and morphisms the commutative squares in $\mathcal{R}ep$.

Let $(\phi, f) : (G, M) \to (G', M')$ be an object in $\mathcal{R}elRep$, or what is the same thing a morphism in $\mathcal{R}ep$. For any projective resolutions $P_\bullet$ and $P'_\bullet$ of $M$ and $M'$ as $G$- and $G'$-modules, there is a ($\phi$-linear) map of resolutions $P_\bullet\to P'_\bullet$ covering $f$. The mapping cone of the chain map 
$$\Z\ot_{\Z G} P_\bullet\rar \Z\ot_{\Z G'}P'_\bullet$$  
computes by definition the relative homology group $H_*(G',G;M',M)$. These relative homology groups fit into a long exact sequence 
$$\cdots \rar H_*(G,M)\rar H_*(G',M')\rar H_*(G',G;M',M) \rar H_{*-1}(G,M)\rar\cdots$$
and they define functors $H_i(-) : \mathcal{R}elRep \to \Z\operatorname{-Mod}$. More precisely, 
given a morphism
\begin{equation*}
\xymatrix{
(G_0, M_0) \ar[r] \ar[d]^-{(\phi_0, f_0)}& (G_1, M_1) \ar[d]^-{(\phi_1, f_1)}\\
(G'_0, M'_0) \ar[r] & (G'_1, M'_1)
}
\end{equation*}
in $\mathcal{R}elRep$ we get an induced map 
$$H_*(G_0',G_0;M_0',M_0)\rar H_*(G_1',G_1;M_1',M_1).$$

\medskip

Given a pre-braided monoidal category $\C$, a pair of objects $(A,X)$ in $\C$ and a coefficient system $F:\C_{A,X} \to \Z\operatorname{-Mod}$, we are interested in the groups 
$H_*(G_n;F_n)$ for  $G_n:=\Aut(A\op X^{\op n})$ and $F_n:=F(A\op X^{\op n})$. 
We denote by
$$Rel^F_*(A,n)=H_*(G_{n+1},G_n;F_{n+1},F_n)$$ 
the relative groups associated to the upper suspension maps $\Si^X:G_n\to G_{n+1}$ and $\s^X:F_n\to F_{n+1}$. By Lemma \ref{lem:TwistedModule}, if $F:\C_{A,X}\to \mathbb{Z}[G_\infty^{ab}]\operatorname{-Mod}$ then we can also form
$$Rel^{F^\circ}_*(A,n)=H_*(G_{n+1},G_n;F_{n+1}^\circ,F_n^\circ).$$

Our main result is the following.

\begin{thm}\label{twistrange}
Let $\C$ be a category and  $(A,X)$ a pair of objects in $\C$. Suppose that $\C$ is locally homogeneous at $(A,X)$ and satisfies  LH3 at $(A,X)$ with slope $k \geq 2$.

If $F:\C_{A,X}\to \Z\operatorname{-Mod}$ is a coefficient system of degree $r$ at $N$, then

\begin{enumerate}[(i)]
\item $Rel_i^F(A,n)$ vanishes for $n\ge \max(N+1,k(i+r))$, and

\item if $F$ is split then  $Rel_i^F(A,n)$ vanishes for $n\ge \max(N+1,ki + r)$. 
\end{enumerate}

If $F:\C_{A,X}\to \mathbb{Z}[G_\infty^{ab}]\operatorname{-Mod}$ is a coefficient system of degree $r$ at $N$ and $k \geq 3$, then

\begin{enumerate}[(i)]
\setcounter{enumi}{2}
\item $Rel_i^{F^\circ}(A,n)$ vanishes for $n\ge \max(2N+1,k(i+r)+k-2)$, and

\item if $F$ is split then $Rel_i^{F^\circ}(A,n)$ vanishes for $n\ge \max(2N+1,ki+2r+k-2)$.
\end{enumerate}
\end{thm}

Theorem \ref{main} then follows from the long exact sequence
\begin{align*}
\cdots\rar H_k(G_n;F_n)\rar H_k(G_{n+1};F_{n+1})\rar Rel^F_k(A,n) \rar  H_{k-1}(G_n;F_n)\rar \cdots,
\end{align*}
as the vanishing of the relative homology groups in a range of degrees gives a range in which the stabilisation maps
$$H_k(G_n;F_n)\rar H_k(G_{n+1};F_{n+1})$$
are epimorphisms and isomorphisms. Similarly for $Rel^{F^\circ}_k(A,n)$. Corollary \ref{cor:B} follows by applying Theorem \ref{main} to the induced coefficient system
$$\overline{F}(-) := F(-) \otimes_{\bZ} \bZ[G_\infty^{ab}] : \mathcal{C} \lra \mathbb{Z}[G_\infty^{ab}]\operatorname{-Mod}$$
and noting that $H_*(G_n ; \overline{F}_n^\circ) \cong H_*(G_n';F_n)$ by Shapiro's lemma, as long as $G_n^{ab} \to G_\infty^{ab}$ is an isomorphism, which happens for $n \geq k+1$ by Theorem \ref{stabthm} applied to first homology (and Corollary \ref{cor:B} makes no claim for $n \leq k$).

\subsection{General properties of the relative homology groups}

Before proving Theorem \ref{twistrange}, we shall establish some basic properties of the relative homology groups $Rel^F_*(A,n)$, and their analogues $Rel^{F^\circ}_*(A,n)$. 

The endofunctors $\Sigma^X$ and $\Sigma_X$ commute, and the diagram 
$$\xymatrix{
id \ar[rrr]^{\s_X}\ar[d]_{\s^X} &&& \Sigma_X \ar[d]^{\s^X(\Si_X)}\\
\Sigma^X \ar[rrr]^-{\s_X(\Si^X)} &&& \Sigma_X\Sigma^X = \Sigma^X\Sigma_X}$$
is a commutative diagram of natural transformations. Hence for any coefficient system $F$ and any $n$ we get a pair of compatible commuting squares 
\begin{equation}\label{LUsquares}
\begin{gathered}
\xymatrix{G_n \ar[r]^{\Si_X} \ar[d]_{\Si^X} & G_{n+1}\ar[d]^{\Si^X} & & F_n\ar[r]^{F(\s_X)} \ar[d]_{F(\s^X)} & F_{n+1} \ar[d]^{F(\s^X)}\\
G_{n+1} \ar[r]^{\Si_X} & G_{n+2} & & F_{n+1}\ar[r]^{F(\s_X)} & F_{n+2} }
\end{gathered}
\end{equation}
i.e.~a morphism in  $\mathcal{R}elRep$, thus inducing for each $i$ a map 
\begin{equation*} 
s_n = (\Sigma_X, F(\sigma_X)):Rel^F_i(A,n)\rar  Rel_i^F(A,n+1).
\end{equation*}
By Lemma \ref{lem:TwistedModule} the same formula defines a map on the relative homology groups $Rel^{F^\circ}_*(A,n)$.

\begin{prop}\label{lsusp}
The map $s_n$ factors as a composition
\begin{equation}\label{factorization}
Rel^F_i(A,n)\sta{(id,F(\s_X))}{\rar} Rel_i^{\Si F}(A,n)\sta{(\Si_X,id)}{\rar} Rel_i^F(A,n+1),
\end{equation}
and the same holds for the relative homology groups $Rel^{F^\circ}_*(A,n)$.
\end{prop}
\begin{proof}
Considering the bottom copies of $F_{n+1}=(\Si F)_n$ and $F_{n+2}=(\Si F)_{n+1}$ in the right square in (\ref{LUsquares}) as $G_n$- and $G_{n+1}$-modules under the lower suspension shows that the pair $(id,F(\s_X))$ defines a map $Rel_i^{F}(A,n)\to Rel_i^{\Si F}(A,n)$. The further composition with $(\Si_X,id)$ then reconsiders $F_{n+1}$ and $F_{n+2}$ as $G_{n+1}$- and $G_{n+2}$-modules.  Likewise replacing $F$ with $F^{\circ}$. 
\end{proof}

The following proposition implies that $\underset{n \to \infty}\colim \, Rel^F_i(A,n) = 0$, and is heavily influenced by Dwyer's ``qualitative stability theorem" \cite{Dwy80}.

\begin{prop}\label{prop:TwoCompIsZero}
For any coefficient system $F$ the composition
$$Rel^F_i(A,n) \overset{s_n}\rar Rel_i^F(A,n+1) \overset{s_{n+1}}\rar Rel_i^F(A,n+2)$$
is zero. Similarly if $F : \C_{A,X} \to \mathbb{Z}[G_\infty^{ab}]\operatorname{-Mod}$, with internalised modules $F_n^\circ$ having relative homology groups $Rel^{F^\circ}_i(A,n)$, then the composition
$$Rel^{F^\circ}_i(A,n) \overset{s_n}\rar Rel_i^{F^\circ}(A,n+1) \overset{s_{n+1}}\rar Rel_i^{F^\circ}(A,n+2)$$
is zero.
\end{prop}

\begin{proof}
Consider the diagram
$${
\xymatrix{\cdots\ar[r]& H_m(G_{n+1};F_{n+1})\ar[r]^-{f_1}\ar[d]& Rel^F_m(A,n)\ar[r]^-{f_2}\ar[d]^{s_n}& H_{m-1}(G_n;F_n)\ar[r]^-{l_2'}\ar[d]^{l_2} &\cdots\\
\cdots\ar[r]& H_m(G_{n+2};F_{n+2})\ar[r]^-{g_1}\ar[d]^{l_1}& Rel^F_m(A,n+1)\ar[r]^-{g_2}\ar[d]^{s_{n+1}}& H_{m-1}(G_{n+1};F_{n+1})\ar[r]\ar[d] &\cdots\\
\cdots\ar[r]^-{l_1'}& H_m(G_{n+3};F_{n+3})\ar[r]^-{h_1}& Rel^F_m(A,n+2)\ar[r]^-{h_2}& H_{m-1}(G_{n+2};F_{n+2})\ar[r] &\cdots
}}$$
where the horizontal lines are part of the long exact sequences associated to the upper suspension, and the vertical maps are the maps induced by the
lower suspension.

In the top right of the diagram, we have that the maps $l_2$ and $l_2'$ are induced by lower and upper suspension respectively. 
From (\ref{lowups}) and (\ref{eq:UpperLowerConj2}), we have a commutative diagram 
$$\xymatrix{A\op X^{\op n} \ar[drr]_{\s_X}\ar[rr]^{\s^X} && A\op X^{\op n+1}
      \ar[d]^{\phi}\ar[r]^{\Si^X(g)} & A\op X^{\op n+1}\ar[d]^\phi\\
&& A \op X\op X^{\op n}\ar[r]_{\Si_X(g)} & A\op X^{\op n+1}
}$$
where $\phi = (b_{X,A}\op X^{\op n})\circ(b_{A \op X^{\op n}, X}) \in G_{n+1}$ and any $g\in G_n$. As $l_2'$ is induced by the pair $(\Si^X,F(\s^X))$, and $l_2$ by the pair $(\Si_X,F(\s_X))$, the above diagram induces the commutative diagram
$$\xymatrix{ Rel^F_m(A,n)\ar[r]^{f_2} & H_{m-1}(G_n;F_n) \ar[r]^-{l_2'}\ar[d]_{l_2} & H_{m-1}(G_{n+1};F_{n+1}) \ar[dl]^{\ \ \ \ (c_{\phi},F(\phi))}_\cong \\
&H_{m-1}(G_{n+1};F_{n+1}) & 
}$$
As $l_2'\circ f_2=0$, it follows that we also have  
$g_2\circ s_n=l_2\circ f_2=0$. Thus every homology class in the image of $s_n$ lifts along $g_1$, and so
$$\mathrm{Im}(s_{n+1} \circ s_n) \subset \mathrm{Im}(s_{n+1} \circ g_1).$$

In the bottom left corner of the diagram, by the same argument we have maps 
$$\xymatrix{H_m(G_{n+2};F_{n+2})\ar[d]_{l_1} \ar[r]^{l'_1}& H_m(G_{n+3},F_{n+3}) \ar[dl]^{(c_{\phi'};F(\phi'))}_\cong\\
 H_m(G_{n+3};F_{n+3})& }$$
where, just as for $l_2$ and $l_2'$, the maps $l_1$ and $l_1'$ are induced by lower and upper suspension respectively, and conjugating by  
$\phi'=(b_{X,A}\op X^{\op n+2})\circ(b_{A \op X^{\op n+2}, X}) \in G_{n+3}$ induces an isomorphism making the above diagram commute. 

\emph{As conjugation induces the identity in homology}, $l_1$ and $l_1'$ induce the same map on homology. This implies that $s_{n+1}\circ g_1=h_1\circ l_1=h_1\circ l_1'=0$, as required. 

The first part still goes through for the coefficient modules $F_n^\circ$. Writing $F_{n+1}^\circ = \mathbb{Z}[G_\infty^{ab}] \otimes_{\mathbb{Z}[G_\infty^{ab}]} F_{n+1}$ with the diagonal action, the map $(c_{\phi}, \mathrm{Id}_{\mathbb{Z}[G_\infty^{ab}]} \otimes F(\phi))$ intertwines $l_2$ and $l'_2$, and this map is invertible so induces an isomorphism on homology, so $l_2 \circ f_2=0$. 

However in the second part the maps $l_1$ and $l_1'$ are intertwined by $(c_{\phi'}, \mathrm{Id}_{\mathbb{Z}[G_\infty^{ab}]} \otimes F(\phi'))$, which is \emph{not} left multiplication by $\phi'$ on $F_{n+3}^\circ$ (it is left multiplication by $\phi'$ on $F_{n+3}$, but this is a different module structure). Thus there is no reason why $(c_{\phi'}, \mathrm{Id}_{\mathbb{Z}[G_\infty^{ab}]} \otimes F(\phi'))$ should induce the identity map on homology. Instead, using the fact that $(c_{\phi'}, [\Si^X_\infty(\phi')] \otimes F(\phi'))$ induces the identity map on homology, and that $l'_1$ is a left $\mathbb{Z}[G_\infty^{ab}]$-module map, write
\begin{align*}
l_1(-) &= (c_{\phi'}, \mathrm{Id}_{\mathbb{Z}[G_\infty^{ab}]} \otimes F(\phi'))(l'_1(-))\\
&= [\phi']^{-1} \cdot [\phi'] \cdot ((c_{\phi'}, \mathrm{Id}_{\mathbb{Z}[G_\infty^{ab}]} \otimes F(\phi'))(l'_1(-)))\\
&= [\phi']^{-1} \cdot ((c_{\phi'}, [\Si^X_\infty(\phi')] \otimes F(\phi'))(l'_1(-)))\\
&= [\phi']^{-1} \cdot l'_1(-)\\
&= l'_1([\phi']^{-1} \cdot -)
\end{align*}
where ``$\cdot$'' denotes the $\mathbb{Z}[G_\infty^{ab}]$-module structure of the homology groups. 
Thus $l_1$ and $l'_1$ differ by \emph{precomposition} by an isomorphism, so as $h_1 \circ l'_1=0$ it follows that $s_{n+1} \circ g_1 = h_1 \circ l_1 = 0$ too, as required.
\end{proof}

\subsection{Proof of the stability theorem (Theorem \ref{twistrange})}

We will prove Theorem \ref{twistrange} by double induction on the degree $r$ of the coefficient system and the homological degree $i$. The case $r=-1$ and arbitrary $i$ is trivial, as is the case $i=-1$ and arbitrary $r$. These form the start of the induction. 

\begin{rem}
The case $r=0=N$ may be deduced from a slight generalisation of Theorems~\ref{stabthm} and \ref{abstabthm}, the case of stability with constant and abelian coefficients, using Remark~\ref{deg0rem}, which explains how to adapt the proof to degree 0 coefficient systems which are not necessarily constant functors. The case $r=0$ and arbitrary $N$ can also be obtained in this way by replacing a coefficient system $F$ which has degree 0 at $N$ by the coefficient system $\widetilde{F}$ defined by 
 $$\widetilde F(X^{\op n})=\left\{\begin{array}{ll}F(X^{\op N})& n<N\\
F(X^{\op n})& n\ge N\end{array}\right.$$
with the action of $G_n$ on $F(X^{\op N})$ induced by the upper suspension $G_n\to G_N$. 

However we shall not use this, and the argument that follows in particular gives another proof of the results of Section \ref{cstsec}, albeit with a slightly worse stability range for injectivity. 
\end{rem}

We fix $r \geq 0$ and $i \geq 0$ and make the following 

\begin{ind}\label{hyp2}
Each of the four statements of Theorem~\ref{twistrange} holds for all coefficient systems of degree $<r$ at any $N\ge 0$ in all homological degrees, and for all coefficient systems of degree $r$ at any $N\ge 0$ in homological degrees $< i$. 
\end{ind}

By Proposition \ref{prop:TwoCompIsZero} the composition $s_{n+1} \circ s_n$ is zero. Our strategy will be to show that each $s_n$ is injective in a range, from which it will follow that $Rel_i^F(A,n)=0$ in this range.  To do this, we factorise $s_n$ as in equation \eqref{factorization} of Proposition \ref{lsusp}, and study the two resulting maps separately.

We start by considering the first map in the composition \eqref{factorization}.

\begin{prop}\label{prop:FirstMapRange}
Let $F:\C_{A,X}\to \Z\operatorname{-Mod}$ be a coefficient system of degree $r$ at $N$.  Suppose that Inductive Hypothesis~\ref{hyp2} is satisfied. Then the map
$$Rel_i^F(A,n) \lra Rel^{\Si F}_i(A,n)$$
is
\begin{tabular}[t]{rl}
 { (i)}& surjective if $n \geq \max(N,k(i+r-1))$,\\
 { (ii)}&injective if $n \geq \max(N,k(i+r))$,\\
 { (iii)}&surjective if $n \geq \max(N, ki + (r-1))$ and $F$ is split,\\
 { (iv)}& split injective if the coefficient system is split.
\end{tabular}\\
For an internalised coefficient system of $G_\infty^{ab}$-modules the analogous map is \\ 
\begin{tabular}[t]{lrl}
& { (i$\,^\prime$)}& surjective if $n \geq \max(2N-1, k(i+r)-2)$,\\
& { (ii$\,^\prime$)}& injective if $n \geq \max(2N-1, k(i+r)+k-2)$,\\
& { (iii$\,^\prime$)}&surjective if $n \geq \max(2N-1, ki + 2r+k-4)$ and $F$ is split,\\ 
& { (iv$\,^\prime$)}& split injective if the coefficient system is split.
\end{tabular}
\end{prop}
\begin{proof}
If $F$ is split, then there is a retraction $\Sigma F \to F$ of coefficient systems. This induces a map on $Rel_i^{(-)}(A,n)$ splitting the map in question. If $F$ is a coefficient system of $G_\infty^{ab}$-modules, then the same splitting gives retractions $(\Sigma F)_n^\circ \to F_n^\circ$. This proves parts (iv) and (iv$^\prime$).

Now suppose that $F$ is not necessarily split. There are short exact sequences of coefficient systems
\begin{align*}
0 \lra \ker F \lra F \lra \s_X(F) \lra 0\\
0 \lra \s_X(F) \lra \Sigma F \lra \coker F \lra 0
\end{align*}
which yield long exact sequences on $Rel_i^{(-)}(A,n)$ (see e.g.~\cite[Lem.~2.4]{Dwy80}). By definition, $\ker F$ has degree $-1$ at $N$, so $(\ker F)_n=0$ for $n \geq N$. Thus the long exact sequence for the first extension shows that $Rel_i^{F}(A,n) \to Rel_i^{\s_X(F)}(A,n)$ is an isomorphism as long as $n \geq N$. In this range, the long exact sequence for the second extension is
$$\cdots \lra Rel_{i+1}^{\coker F}(A,n) \lra Rel_i^{F}(A,n) \lra Rel_i^{\Sigma F}(A,n) \lra Rel_i^{\coker F}(A,n) \lra  \cdots.$$
Thus the map in question is surjective for those $n \geq N$ for which $Rel_i^{\coker F}(A,n)=0$, and injective for those $n \geq N$ for which $Rel_{i+1}^{\coker F}(A,n)=0$. As $\coker F$ has degree $(r-1)$ at $(N-1)$, the ranges in the proposition follow from Inductive Hypothesis \ref{hyp2}.
\end{proof}

We now study the second map in the composition \eqref{factorization}.

\begin{prop}\label{prop:TwistedSSArgument}
Suppose that Inductive Hypothesis \ref{hyp2} is satisfied.

If $F$ is a coefficient system of degree $r$ at $N$, then the map
$$Rel_i^{\Si F}(A,n)\sta{(\Si_X,id)}{\rar} Rel_i^F(A,n+1)$$
is 
\begin{tabular}[t]{rl}
 {(i)}& surjective for $n \geq \max(N,k(i+r) - (k-1))$ \\ 
 {(ii)}& injective for $n \geq \max(N+1,k(i+r))$ \\  
 {(iii)}& surjective for $n \geq \max(N,ki+r-(k-1))$ if $F$ is split\\ 
 {(iv)}&injective for $n \geq \max(N+1,ki+r)$ if $F$ is split. 
\end{tabular}\\
Similarly, if $F$ is a coefficient system of $G_\infty^{ab}$-modules of degree $r$ at $N$, then the map 
$$Rel_i^{\Si F^\circ}(A,n)\sta{(\Si_X,id)}{\rar} Rel_i^{F^\circ}(A,n+1)$$
is 
\begin{tabular}[t]{rl}
 { (i$\,^\prime$)}& surjective for $n \geq \max(2N-2,k(i+r)-1)$ \\
 { (ii$\,^\prime$)}& injective for $n \geq \max(2N+1,k(i+r)+1)$\\
 { (iii$\,^\prime$)}&  surjective for $n \geq \max(2N-2,ki+2r-1)$ if $F$ is split\\
 { (iv$\,^\prime$)}& injective for $n \geq \max(2N+1,ki+2r+k-2)$ if $F$ is split. 
\end{tabular}
\end{prop}
\begin{proof} 
Let $W_n=W_n(A,X)_\bullet$ be the semi-simplicial set of Definition~\ref{simpdef}, having $p$-simplices $\Hom(X^{p+1},A\op X^{\op n})$. Just as in the proof of Theorem~\ref{stabthm} we have that, by LH1, the group $G_n:=\Aut(A\op X^{\op n})$ acts transitively on the set of $p$-simplices of $W_n$ for $0 \leq p<n$, and that, by LH2, in this range the stabiliser of a $p$-simplex is isomorphic to $G_{n-p-1}=\Aut(A\op X^{\op n-p-1})$. As in the proof of Theorem~\ref{stabthm}, because we are only interested in homological degrees up to $\tfrac{n-1}{k}$ we shall ignore the fact that this fails for $p\geq n$. 
The set of $p$-simplices of $W_n$ is hence isomorphic to $G_n/G_{n-p-1}$ as a $G_n$-set. We will consider a spectral sequence associated to the action of $G_{n+1}$ on  $W_{n+1}$, and of $G_{n+2}$ on  $W_{n+2}$.  (This spectral sequence is a relative version of the one used in the proof of Theorem~\ref{stabthm}.) 
Note that our assumption implies that $W_{n+1}$ and $W_{n+2}$ are at least $(\frac{n-1}{k})$-connected, the latter being in fact at least  $(\frac{n}{k})$-connected. 

Let $F_n:=F(A\op X^{\op n})$. The upper suspension map induces a map $F(\s^X):F_{n+1}\to F_{n+2}$ which is $G_{n+1}$-equivariant. Choose  projective resolutions $P_\bullet$ and $Q_\bullet$ of $F_{n+1}$ and $F_{n+2}$ as $G_{n+1}$- and $G_{n+2}$-modules. Recall that there is a map of resolutions $P_\bullet\to Q_\bullet$ compatible with the suspension, and the cone of the chain map 
$$\Z \ot_{\Z G_{n+1}} P_\bullet\rar \Z \ot_{\Z G_{n+2}} Q_\bullet$$ 
computes the relative homology group $H_*(G_{n+2},G_{n+1};F_{n+2},F_{n+1})$.

Let $\wt C_*(W_n)$ denote the augmented cellular chain complex of $W_n$.  Post-composing by the upper suspension $\s^X$ induces a map $W_{n+1}\to W_{n+2}$ which is also $G_{n+1}$-equivariant with respect to upper suspension. Hence  we get a map of double complexes 
$$\bZ \ot_{\Z G_{n+1}} (\wt C_*(W_{n+1}) \ot_\bZ P_\bullet)  \rar \bZ \ot_{\Z G_{n+2}} (\wt C_*(W_{n+2}) \ot_\bZ Q_\bullet)$$
whose levelwise cone (in the $\bullet$ direction) 
$$C_{p,q}\ =\ \big( \Z\ot_{\Z G_{n+1}}(\wt C_p(W_{n+1}) \ot_\bZ P_{q-1}) \big)\ \bigoplus\ \big(\Z\ot_{\Z G_{n+2}} (\wt C_p(W_{n+2}) \ot_\Z Q_q)\big ) $$
is a double complex and hence has two  associated spectral sequences. Taking first the homology in the $p$-direction yields a spectral sequence with
$$E^1_{p,q} = H_p\big(\Z\ot_{\Z G_{n+2}} (\wt C_*(W_{n+2}) \ot_\Z Q_q)\big) \oplus H_{p}\big( \Z\ot_{\Z G_{n+1}}(\wt C_*(W_{n+1}) \ot_\bZ P_{q-1}) \big).$$
As $W_{n+1}$ is $(\frac{n-1}{k})$-connected and $W_{n+2}$ is $(\frac{n}{k})$-connected, applying the Universal Coefficient Theorem (twice) shows that $E^1_{p,q}=0$ for $p+q \le \frac{n}{k}$. It follows that both spectral sequences converge to zero in that range of total degrees. 
By Shapiro's lemma, transitivity of the action on $p$-simplices for all $p$ implies that the other spectral sequence  has $E^1$-term 
\begin{align*}
E^1_{p,q}\cong H_q(\St(\s^X\!\circ\s_p),\St(\s_p);F_{n+2},F_{n+1})
\end{align*}
for $\s_p$ any $p$-simplex of $W_{n+1}$. We now analyse the terms in this second spectral sequence.  

For $p=-1$, the stabiliser is the whole group and  $$E^1_{-1,i}=Rel^F_i(A,n+1)=H_i(G_{n+2},G_{n+1};F_{n+2},F_{n+1})$$ 
is the target of the map we are interested in. 

For $p \geq 0$, we here choose as a standard $p$-simplex
$$\sigma_p =  \iota_{A} \op {X^{\op p+1}} \op \iota_{X^{\op n-p}}: X^{\op p+1} \to A \op X^{\op n+1} \in (W_{n+1})_p.$$
Then the map 
$$(\Sigma_X)^{p+1} : G_{n-p} \ \lra\  \St(\sigma_p) = \Fix(\iota_{A} \op {X^{\op p+1}} \op \iota_{X^{\op n-p}}) \leq G_{n+1}$$
is an isomorphism, as for $f\in G_{n-p}$ we have that $\Si_{X}^{p+1}(f)=(b_{X^{\op p+1},A}\op X^{\op n-p})\circ (X^{\op p+1}\op f)\circ (b^{-1}_{X^{\op p+1},A}\op X^{\op n-p})$, and the pre-braid relation gives that $\Si_X^{p+1}(f)\circ\s_p=\s_p$. 

Under this map, the module $F_{n+1}$ pulls back to $(\Sigma^{p+1} F)_{n-p}$. Likewise, for the action of $G_{n+2}$ on $W_{n+2}$, the simplex $\s^X\!\circ\sigma_p \in (W_{n+2})_p$ is $ \iota_{A} \op {X^{\op p+1}} \op \iota_{X^{\op n+1-p}}$, and
$$(\Sigma_X)^{p+1} : G_{n+1-p} \ \lra\ \St(\s^X\!\circ\s_p) = \Fix(\iota_{A} \op {X^{\op p+1}} \op \iota_{X^{\op n+1-p}}) \leq G_{n+2}$$
pulls back $F_{n+2}$ to $(\Sigma^{p+1} F)_{n+1-p}$. As these identifications commute with the upper suspension, we
get an isomorphism 
$$E^1_{p,q} \cong H_q(\St(\s^X\!\circ\s_p),\St(\s_p);F_{n+2},F_{n+1}) \cong Rel_q^{\Sigma^{p+1}F}(A, n-p).$$
This discussion goes through with only notational changes if we use internalisations $F_{n+1}^\circ$ and $F_{n+2}^\circ$ coming from a coefficient system of $G_\infty^{ab}$-modules.

In particular, for $p=0$, we have that $$E^1_{0,i}\cong H_i(\St(\s^X\!\circ\s_0),\St(\s_0);F_{n+2},F_{n+1})\cong Rel_i^{\Si F}(A,n)$$ identifies with the source of the map in the statement of the proposition.  
Moreover, 
the differential $d^1:E^1_{0,i}\to E^1_{-1,i}$, which is induced by the inclusion of the stabilisers into the groups,
is, when precomposed with this isomorphism, the lower suspension, i.e.~the map in the statement of the proposition. 

\vspace{2ex}

\noindent\textbf{Split case}. By Lemma \ref{suspdeg} the coefficient system $\Sigma^{p+1} F$ has degree $r$ at $N-p-1$, and so by Inductive Hypothesis \ref{hyp2}, $Rel_q^{\Sigma^{p+1}F}(A, n-p)=0$ for $n-p \geq \max(N-p,kq + r)$ and $q < i$. If $Rel^{\Sigma^{p+1} F}_q(A,n-p)=0$ for all $p+q=i$ and $q<i$ and $E_{-1,i}^\infty=0$, then the differential $d^1 : E^1_{0,i} \to E^1_{-1,i}$ must be onto; this happens when $n \geq \max(N,k(i-1)+1 + r)$.

Similarly, the sequence $E^1_{1,i} \overset{d^1}\to E^1_{0,i} \overset{d^1}\to E^1_{-1,i}$ must be exact in the middle position if $Rel^{\Sigma^{p+1} F}_q(A,n-p)=0$ for all $p+q=i+1$ and $q<i$ and $E_{0,i}^\infty=0$; this happens for $n \geq \max(N,ki + r)$. Finally, we will show that $d^1 : E^1_{1,i} \to E^1_{0,i}$ is zero for $n \geq \max(N+1,ki+r)$. It will then follow by exactness that $d^1 : E^1_{0,i} \to E^1_{-1,i}$ is injective for $n \geq \max(N+1,ki+r)$.

The differential $d^1 : E^1_{1,i} \to E^1_{0,i}$ is given as the difference of two maps, corresponding to the two 0-simplices $d_0\sigma_1$ and $d_1\sigma_1$ arising as the faces of the standard 1-simplex. We may observe that $d_1\sigma_1 = \sigma_0$, and $h_1 \cdot  d_0\sigma_1 = \sigma_0$ for some $h_1 \in G_{n+1}$. The first map defining $d^1$ is thus the map induced on homology by the composition of the map
$$(\St(\s^X\!\circ \s_1), \St(\s_1); F_{n+2}, F_{n+1}) \lra (\St(\s^X\!\circ d_0\s_1), \St(d_0\s_1); F_{n+2}, F_{n+1})$$
induced by inclusion of stabilisers, with the map
$$(\St(\s^X\!\circ d_0\s_1), \St(d_0\s_1); F_{n+2}, F_{n+1}) \lra (\St(\s^X\!\circ \s_0), \St(\s_0); F_{n+2}, F_{n+1})$$
induced by $(c_{\Sigma^X h_1}, c_{h_1})$. (Recall from the proof of Theorems~\ref{stabthm} and~\ref{abstabthm} that we write $c_{h_1} := (h_1(-)h_1^{-1}, h_1 \cdot (-)) : (\St(d_0\s_1), F_{n+1}) \to (\St(\s_0), F_{n+1})$ for the map induced by conjugation by $h_1$ on the groups and the action of $h_1$ on the modules, and similarly $c_{\Sigma^X h_1}$.) The second map defining $d^1$ is the map induced on homology by 
$$(\St(\s^X\!\circ \s_1), \St(\s_1); F_{n+2}, F_{n+1}) \lra (\St(\s^X\!\circ \s_0), \St(\s_0); F_{n+2}, F_{n+1}),$$
the inclusion of stabilisers. 

The choice $h_1 := A \oplus b_{X,X} \oplus X^{\op n-1} \in G_{n+1}$ has the additional property that it centralises $\St(\sigma_1)=\Fix(\iota_X\op X^{\op 2}\op \iota_{X^{\op n-1}})$; indeed, this follows from Proposition~\ref{H2sym} noting that $h_1 = (A\op b_{X^2,X^{n-1}}^{-1}) \circ (A \oplus X^{\op n-1}\oplus b_{X,X}) \circ (A\op b_{X^2,X^{n-1}})$.
Similarly $\Sigma^X(h_1) \in G_{n+2}$ centralises $\St(\s^X\!\circ \sigma_1)$. Thus the two homomorphisms
$$Rel_i^{\Sigma^2 F}(A, n-1) \lra Rel_i^{\Sigma F}(A, n)$$
only differ by the automorphism of the source induced by $(F(\Sigma^X h_1), F(h_1))$ acting on $(F_{n+2}, F_{n+1})$. This automorphism is trivial on $(F(\sigma_X^2)(F_{n}), F(\sigma_X^2)(F_{n-1})) \subset (F_{n+2}, F_{n+1})$. 
Thus, as the composition
$$Rel_i^{F}(A, n-1) \lra Rel_i^{\Sigma F}(A, n-1) \lra Rel_i^{\Sigma^2 F}(A, n-1)$$
is surjective for $n-1 \geq \max(N,ki + r-1)$ by Proposition \ref{prop:FirstMapRange} (iii), it follows that the two homomorphisms agree for $n \geq \max(N+1,ki+r)$, and so their difference, $d^1 : E^1_{1,i} \to E^1_{0,i}$, is trivial

\vspace{2ex}

\noindent\textbf{Non-split case}. The argument is the same as the previous case, but the ranges are different; we thus go through it quickly. By Lemma \ref{suspdeg} and Inductive Hypothesis \ref{hyp2} we may assume that $Rel_q^{\Sigma^{p+1}F}(A, n-p)=0$ for $n-p \geq \max(N-p,k(q+r))$ and $q < i$. One checks that if $n \geq \max(N,k(i+r))$ then $d^1 : E^1_{0,i} \to E^1_{-1,i}$ is onto, and that if $n \geq \max(N,k(i+r))$ then $E^1_{1,i} \overset{d^1}\to E^1_{0,i} \overset{d^1}\to E^1_{-1,i}$ is exact in the middle. By the same argument as above, the differential $d^1 : E^1_{1,i} \to E^1_{0,i}$ is zero in the range that
$$Rel_i^{F}(A, n-1) \lra Rel_i^{\Sigma F}(A, n-1) \lra Rel_i^{\Sigma^2 F}(A, n-1)$$
is surjective, which by Proposition \ref{prop:FirstMapRange} (i) is when $n-1 \geq \max(N,k(i+r-1))$, or in other words when $n \geq \max(N+1,k(i+r) -(k-1))$. Thus if $n \geq \max(N+1,k(i+r))$ then $d^1: E^1_{0,i} \to E^1_{-1,i}$ is injective.

\vspace{2ex}

\noindent\textbf{Internalised split coefficient system of $G_\infty^{ab}$-modules case}. The argument for surjectivity at $E^1_{-1,i}$ and exactness at $E^1_{0,i}$ is the same as in the previous cases. We deduce that $d^1 : E^1_{0,i} \to E^1_{-1,i}$ is onto for $n \geq \max(2N-2, ki+2r-1)$ and  that $E^1_{1,i} \overset{d^1}\to E^1_{0,i} \overset{d^1}\to E^1_{-1,i}$ is exact in the middle for $n \geq \max(2N-3,ki+2r)$.

To show that $d^1 : E^1_{1,i} \to E^1_{0,i}$ is zero, our description of this differential in the previous sections still holds, but the first map making up this differential is the composition of the map induced by inclusion of stabilisers followed by
$$(c_{\Sigma^X h_1}, c_{h_1}) =(\Sigma^X h_1(-)\Sigma^X h_1^{-1}, h_1(-)h_1^{-1}; [\Si^X_\infty h_1] \otimes F(\Sigma^X h_1), [\Si^X_\infty h_1] \otimes F(h_1))$$
where $[\Si^X_\infty h_1] \otimes F(h_1) : \mathbb{Z}[G_\infty^{ab}] \otimes_{\mathbb{Z}[G_\infty^{ab}]} F_{n+1} \to \mathbb{Z}[G_\infty^{ab}] \otimes_{\mathbb{Z}[G_\infty^{ab}]} F_{n+1}$ is the map induced on $F_{n+1}^\circ$ by multiplication by $[\Si^X_\infty h_1] \in G_\infty^{ab}$ on the first factor and by $F(h_1)$ on the second; similarly for the map on $F_{n+2}^\circ$.

Consider precomposing $d^1 : E^1_{1,i} \to E^1_{0,i}$ with the map induced on homology by
$$\phi: (\St(\s^X\!\circ \s_2), \St(\s_2); F_{n-1}^\circ, F_{n-2}^\circ) \lra (\St(\s^X\!\circ \s_1), \St(\s_1); F_{n+2}^\circ, F_{n+1}^\circ),$$
given by the inclusion between the groups, and by
$$id \otimes F(\sigma_X^3) : F_{n-2}^\circ = \mathbb{Z}[G_\infty^{ab}] \otimes_{\mathbb{Z}[G_\infty^{ab}]} F_{n-2} \lra \mathbb{Z}[G_\infty^{ab}] \otimes_{\mathbb{Z}[G_\infty^{ab}]} F_{n+1} = F_{n+1}^\circ$$
on coefficients (with the analogous map $F_{n-1}^\circ \to F_{n+2}^\circ$).

The choice $h_1 := (A \op X \op b^{-1}_{X,X} \op X^{\op n-2}) \circ (A \op b_{X,X} \op X^{\op n-1}) \in G_{n+1}$ takes $d_0\sigma_1$ to $\sigma_0$, centralises $\St(\sigma_2) =\Fix(\iota_X\op X^{\op 3}\op \iota_{X^{\op n-2}})$, and $\Sigma^X h_1 \in G_{n+2}$ centralises $\St(\s^X\circ \sigma_2)$, using just as before Proposition~\ref{H2sym}. In addition, it satisfies $[\Si^X_\infty  h_1]=0 \in G_\infty^{ab}$, which may be seen by abelianising (\ref{braidrelation}). 

The map $(F(\Sigma^X h_1), F(h_1))$ restricts to the identity on the submodules 
$$(F(\sigma_X^3)(F_{n-1}), F(\sigma_X^3)(F_{n-2})) \subset (F_{n+2}, F_{n+1}).$$ 
Thus after precomposing with $H_i(\phi)$, the two maps whose difference forms $d^1 : E^1_{1,i} \to E^1_{0,i}$ become equal, and so this differential is zero when $H_i(\phi)$ is surjective. The map $H_i(\phi)$ is the composition
$$Rel_i^{F^\circ}(A, n-2) \overset{(id, id \otimes F(\sigma_X^3))}\lra Rel_i^{\Sigma^3 F^\circ}(A, n-2) \overset{(\Sigma_X, id)}\lra Rel_i^{\Sigma^2 F^\circ}(A,n-1)$$
so by the surjectivity part of this proposition, which is already established, and Proposition \ref{prop:FirstMapRange} (iii$^\prime$), it is surjective as long as $n \geq \max(2N+1,ki+2r+k-2)$ (here we have used our assumption that $k \geq 3$, so $k-2 \geq 1$). Thus $d^1 : E^1_{0,i} \to E^1_{-1,i}$ in injective in this range of degrees.

\vspace{2ex}

\noindent\textbf{Internalised non-split coefficient system of $G_\infty^{ab}$-modules case}. The argument is the same as the previous case, but the ranges are different. By Inductive Hypothesis \ref{hyp2} we may assume that $Rel_q^{\Sigma^{p+1}F^\circ}(A, n-p)=0$ for $n-p \geq \max(2N-2p-1, k(q+r)+k-2)$ and $q < i$. It thus follows that if $n \geq \max(2N-2,k(i+r)-1)$ then $d^1 : E^1_{0,i} \to E^1_{-1,i}$ is onto, and that if $n \geq \max(2N-3,k(i+r))$ then $E^1_{1,i} \overset{d^1}\to E^1_{0,i} \overset{d^1}\to E^1_{-1,i}$ is exact in the middle. By the same argument as above, the differential $d^1 : E^1_{1,i} \to E^1_{0,i}$ is zero in the range that the composition
$$Rel_i^{F^\circ}(A, n-2) \overset{(id, id \otimes F(\sigma_X^3))}\lra Rel_i^{\Sigma^3 F^\circ}(A, n-2) \overset{(\Sigma_X, id)}\lra Rel_i^{\Sigma^2 F^\circ}(A,n-1)$$
is surjective, which by Proposition \ref{prop:FirstMapRange} (i$^\prime$) and the surjectivity part of this proposition is when $n \geq \max(2N+1,k(i+r)+1)$. Thus in this range $d^1: E^1_{0,i} \to E^1_{-1,i}$ is injective.
\end{proof}

Let us now conclude the proof of Theorem \ref{twistrange}. We have shown in Proposition \ref{prop:TwoCompIsZero} that the composition
\begin{equation*}
    \xymatrix{
      & & Rel^F_i(A,n) \ar[r]^-{(id,F(\s_X))} & Rel_i^{\Si F}(A,n) 
  \ar@{->} `r[d] `[ll] `^dl[lll] `^r[dll]  [dll]^-{(\Sigma_X, id)} \\
& Rel_i^F(A,n+1) \ar[r]^-{(id,F(\s_X))} & Rel_i^{\Si F}(A,n+1) \ar[r]^-{(\Si_X,id)} & Rel_i^F(A,n+2)
    }
\end{equation*}
is zero, and similarly with internalised $G_\infty^{ab}$-module coefficients. Hence $Rel_i(A,n)$ must be zero in the range in which those four maps are all injective. We check this in each case. Note that the ranges for which we know that the first and second maps are injective are always included in the ranges for which we know that the third and fourth maps are injective, so it is enough to consider the first two maps. 

\vspace{1ex}

\noindent\textbf{In the split case}: The first map is injective, as the coefficient system is split, and the second map is injective for $n \geq \max(N+1, ki+r)$ by Proposition \ref{prop:TwistedSSArgument} (iv). It follows that $Rel_i^F(A,n)=0$ for $n \geq \max(N+1, ki+r)$.

\noindent\textbf{In the non-split case}: The first map is injective for $n \geq \max(N, k(i+r))$ by Proposition~\ref{prop:FirstMapRange} (ii)  and the second map 
 for $n \geq \max(N+1, k(i+r))$ by Proposition \ref{prop:TwistedSSArgument} (ii). Thus $Rel_i^F(A,n)=0$ for $n \geq \max(N+1, k(i+r))$.

\vspace{1ex}

In the case of an internalised coefficient system of $G_\infty^{ab}$-modules, we have shown in Proposition \ref{prop:TwoCompIsZero} that the analogous composition is zero.

\vspace{1ex}

\noindent\textbf{In the split case}:  The first map is injective by the splitness of the coefficient system, and the second map is injective for $n \geq \max(2N+1, ki+2r+k-2)$ by Proposition \ref{prop:TwistedSSArgument} (iv$^\prime$), hence $Rel_i^{F^\circ}(A,n)=0$ in that range.

\noindent\textbf{In the non-split case}: The first map is injective for $n \geq \max(2N-1, k(i+r)+k-2)$ by Proposition~\ref{prop:FirstMapRange} (ii$^\prime$)  and the second map for $n \geq \max(2N+1, k(i+r)+1)$ by Proposition \ref{prop:TwistedSSArgument} (ii$^\prime$). Thus $Rel_i^F(A,n)=0$ for $n \geq \max(2N+1, k(i+r)+k-2)$.

\section{Examples}\label{examples}

In this section, we apply our theory to prove or reprove stability theorems with twisted coefficients for classical families of groups. We will construct homogeneous categories from braided monoidal groupoids which satisfy the hypotheses of Theorem~\ref{universal}. In most cases considered here, the construction of the category and hence its associated semi-simplicial sets and simplicial complexes will be essentially automatic. The only non-trivial property we will have to check is the high-connectivity of the resulting semi-simplicial sets. This will be achieved by using the results of Section \ref{sec:SxCxSSets} to relate these to simplicial complexes already studied in the literature.

\subsection{Symmetric groups}\label{setex}

The simplest example is obtained from the groupoid $\Sigma$ of finite sets and bijections. This groupoid is symmetric monoidal, with the sum induced by the disjoint union of finite sets, and the symmetry given by the canonical bijection $A\sqcup B\to B\sqcup A$. The unit is the empty set. Finite sets have cancellation, have no zero divisors, and $\Aut(A)\to \Aut(A\sqcup B)$ is injective for all $A$ and $B$. Hence Theorem~\ref{universal} provides an associated homogeneous category $U\Sigma=\langle\Sigma,\Sigma\rangle$.

This homogeneous category has objects the finite sets. A morphism from $A$ to $B$ is an equivalence class of pairs $(X,f)$ with $X$ a finite
set and $f:X\sqcup A\to B$ a bijection. Two pairs $(X,f)$ and $(X',f')$ are equivalent if there is an isomorphism $g:X\to X'$ such that 
$$f=f'\circ (g'\sqcup A):X\sqcup A\to X'\sqcup A\to B.$$ 
In particular, we must have $f|_A=f'|_A$. In fact,  the morphism $[X,f]$ is exactly determined by this injection $f|_A$ and $U\Sigma$ 
identifies with the category $FI$ of finite sets and injections, our first example of a homogeneous category in Section~\ref{homcatsec}. 
Note that the automorphism groups of objects are (still) their groups of symmetries, and indeed the groupoid $\Sigma$ satisfies the hypothesis of Proposition~\ref{underlying}. 

\smallskip

As the category $U\Sigma$, or $FI$, is symmetric monoidal and locally standard, by Proposition \ref{prop:SymMonBuilding} the semi-simplicial sets $W_n(A,X)_\bullet$ satisfy condition (A). Hence, by Theorem \ref{caseA} these are $(\tfrac{n-a}{k})$-connected as long as the associated simplicial complexes $S_n(A,X)$ are $(\tfrac{n-a}{k})$-connected for some $a, k\ge 1$ for all $n\ge 0$. The interesting case is when $A=\emp$ and $X=\{*\}$, in which case the simplicial complex $S_n := S_n(\emp,\{*\})$ has vertices the inclusions of $X=\{*\}=[1]$ inside $[n]=\{*\}\sqcup\dots\sqcup\{*\}\cong\{1,\dots,n\}$, so the vertices of $S_n$ can be identified with the numbers $1,\dots,n$. Now for any collection of distinct vertices $1\le i_0,\dots,i_p\le n$, there is an associated injection $f:[p+1]\to [n]$ that takes $j$ to $i_j$. Hence any collection of vertices defines a simplex in $S_n$ and so $S_n$ may be identified with $\De^{n-1}$. In particular it is contractible for all $n\ge 1$ and hence at least $(n-2)$-connected for all $n\ge 0$; Theorem \ref{caseA} then implies that $W_n(\emp,\{*\})_\bullet$ is also $(n-2)$-connected, and hence at least $(\frac{n-2}{2})$-connected for all $n\ge 1$. 
 
\medskip

Applying Theorems \ref{stabthm}, \ref{abstabthm}, and~\ref{twistrange} with the above choices of $X$ and $A$ gives homological stability for the symmetric groups, with trivial coefficients, abelian coefficients, and coefficients in finite degree coefficient systems, including finitely generated FI-modules as explained in Example~\ref{FIex}. Let us unwrap all these results in this case.

\pagebreak

\begin{thm}
Let $F : FI \to R\operatorname{-Mod}$ be a coefficient system, which may be constant. Then the map
$$H_i(\Si_n;F(n))\lra H_i(\Si_{n+1};F(n+1))$$
is: 
\begin{enumerate}[(i)]
\item an epimorphism for $i \leq \tfrac{n}{2}$ and an isomorphism for $i \leq \tfrac{n-1}{2}$, if $F$ is constant;

\item an epimorphism for $i \leq \tfrac{n-r}{2}$ and an isomorphism for $i \leq \tfrac{n-r-2}{2}$, if $F$ is split of degree $r$ at 0; 

\item an epimorphism for $i \leq \tfrac{n}{2}-r$ and an isomorphism for $i \leq \tfrac{n-2}{2}-r$, if $F$ is of degree $r$ at 0;
\end{enumerate}

If $A_n \subset \Sigma_n$ denotes the alternating groups, then the map
$$H_i(A_n;F(n))\lra H_i(A_{n+1};F(n+1))$$
is:
\begin{enumerate}[(i)]
\item an epimorphism for $i \leq \tfrac{n-1}{3}$ and an isomorphism for $i \leq \tfrac{n-3}{3}$, if $F$ is constant;

\item an epimorphism for $i \leq \tfrac{n-2r-1}{3}$ and an isomorphism for $i \leq \tfrac{n-2r-4}{3}$, if $F$ is split of degree $r$ at 0;

\item an epimorphism for $i \leq \tfrac{n-1}{3}-r$ and an isomorphism for $i \leq \tfrac{n-4}{3}-r$, if $F$ is of degree $r$ at 0;
\end{enumerate}
In both cases, if $F$ is of degree $r$ at $N>0$, then (ii) holds for all $n\ge N+1$ and (iii) for all $n\ge 2N+1$. 
\end{thm}

For the symmetric group with constant coefficients, this result is originally due to
Nakaoka \cite{Nak60}. The particular proof we give here in that case is the one given in \cite[Thm.~2]{Ker05}. (Nakaoka though shows more, namely that the map is always injective.)

For the symmetric group with a split coefficient system of finite degree we recover Betley's Theorem 4.3 in \cite{Bet02} (up to a minor difference in the range). 
Betley's coefficients are functors $T:\Ga\to \A$ of degree $r$, for $\Ga$ the category of finite pointed sets, where degree $r$ here means that $T([0])=0$ and $(r+1)$st cross-effect is trivial, i.e.\ it is polynomial in the classical sense. As explained in Remark~\ref{polrem}, such functors have trivial iterated co\-kernels.  There is a functor $FI\to \Ga$ taking  a set $A$ to $A\sqcup *$ with $*$ its basepoint, and extending injections between finite sets to pointed injections between pointed sets, so a functor from $\Ga$
can be pulled-back to a functor from $FI$, and this preserves the property of the cokernels. Furthermore, as all injective morphisms in $\Gamma$ are split injective, the coefficient systems so obtained are always split. 

As we have explained in Example~\ref{FIex}, an $FI$-module $F : FI \to R\operatorname{-Mod}$ which is generated in degrees $\leq k$ and related in degrees $\leq d$ is a coefficient system of degree $k$ at $d+\min(k,d)$. This proves Corollary \ref{cor:E}. Many such $FI$-modules do not extend to functors from $\Ga$ because they are not split. 

For the alternating groups with constant coefficients, we recover a theorem of Hausmann \cite{Hausmann}, who proved stability of the homology of the alternating groups by direct calculation. This has been recently reproved by Palmer \cite{Palmer}. The remaining cases are new.

\medskip

By the calculations of Hausmann \cite[Proposition B]{Hausmann}, we see that the slope of $\tfrac{1}{3}$ is sharp for the homology of alternating groups with constant coefficients. Furthermore, the stable homology of the symmetric groups can be calculated by group-completion, where the theorem of Barratt--Priddy--Quillen--Segal identifies it with the homology of $Q_0(S^0)$, the basepoint component of the free infinite loop space on a point. Thus the discussion in Section \ref{sec:StabHomology} shows that the homology of the infinite alternating group is that of the universal cover of ${Q_0(S^0)}$. Thus we get the following computation:  

\begin{thm}
Let $\widetilde{Q_0S^0}$ denote the universal cover of $Q_0S^0$ and $A_n$ the $n$th alternating group. Then $H_i(A_n)\cong H_i(\widetilde{Q_0S^0})$ for all $i\le \tfrac{n-3}{3}$.  
\end{thm}

Given any group $G$, we can consider the groupoid $\Sigma\wr G:=\coprod_n \Si_n\wr G$ of wreath products of the symmetric groups with $G$ (or equivalently, the groupoid of finite sets and invertible {\em $G$-maps} in the language of \cite{SamSno14}). Applying our machine just as above yields simplicial complexes $S_n$ denoted $(\De^{n-1})^G$ in \cite{HatWah10},  
the simplicial complex $\De^{n-1}$ with the constant labelling system $L=G$ in the language of \cite[Ex.~3.3]{HatWah10} with vertices $\{0,\dots,n-1\}\times G$ and where vertices $(i_0,g_0),\dots,(i_p,g_p)$ form a simplex if and only if their first components form a simplex in $\De^{n-1}$. 
This complex is $(n-2)$-connected by Proposition 3.5 in that paper because $\De^{n-1}$ is wCM of dimension $n-1$ and $(\De^{n-1})^G$ is a complete join complex over it. 
Hence the same stability result holds for these groups, with the same stability range, which establishes Theorem \ref{thm:D} in the case $G_n = G \wr \Sigma_n$. The corresponding homogeneous category and its representations is studied in \cite{SamSno14}.

\subsection{Automorphisms of free groups and free products of groups}

Let $f\G$ be the groupoid of finitely generated groups and their isomorphisms, and consider the monoidal structure on $f\G$ induced by taking free products. Then $(f\G,*,e)$ is a symmetric monoidal groupoid without zero divisors. This satisfies the cancellation axiom C: this is part of Grushko's theorem giving the uniqueness of decomposition as a free product for finitely generated groups. Also, for any groups $H$ and $G$, the map $\Aut(H)\to \Aut(H*G)$ is injective. Hence by Theorem~\ref{universal} the category $(\U f \G,*,e)$ is homogeneous. Concretely, $\U f\G$ has objects the finitely generated groups and morphisms from $A$ to $B$ given by equivalence classes of pairs $(H,f)$ for $H$ a group and $f:H*A\to B$ an isomorphism.  Alternatively, such a morphism can be described as  a pair $(g,H)$ of an injective homomorphism $g:A\inc B$ and a subgroup $H\le B$ such that $B=H*f(A)$. Composition in these terms is defined by $(g,K)\circ (f,H)=(g\circ f,g(H)*K)$. The category $\U f\G$ is a symmetric monoidal category by Proposition~\ref{braidandsym} and we have that $\Aut_{\U f\G}(G)=\Aut(G)$ for any object $G$.

\medskip

We are left to consider the connectivity of the associated semi-simplicial sets. We will start by considering the full  subcategory $\U f\G_{free}$ generated by the finitely generated free groups, and then discuss what is known in the more general case\footnote{The category $\U f\G_{free}$ appears in the work of Djament--Vespa under the name $\G$ \cite[Def
3.1]{DjaVes15}, where they also verified that it satisfies H1 and H2---see the proof of Proposition 3.4 in \cite{DjaVes15}.}.

\subsubsection{Stability for $\Aut(F_n)$}
  
The category  $\U f\G_{free}$ is generated as a monoidal category by the free group on one generator $\Z$, which means that stabilisation by $\Z$ is the only interesting stabilisation in that category. 

The simplicial complex $S_n=S_n(e,\Z)$ is essentially the complex of split factorisations $SF_n$ of Hatcher--Vogtmann \cite[Sec.~6]{HatVog-cerf}: A $p$-simplex in $S_n$ can be written as  an equivalence class of choice of lifts of the simplex, i.e. a morphism $(f,H):\Z^{* p+1}\to \Z^{* n}$ in $Uf\G$ up to the equivalence relation defined by pre-composition by permutations of the $p+1$ factors in the domain. On the other hand, a $p$-simplex of $SF_n$ is an unordered factorisation $\Z^{* n}=\Z_0*\dots*\Z_{p}*H$ with each $\Z_i\cong \Z$. There is a forgetful map $S_n\to SF_n$, and the only data forgotten is an actual choice of isomorphism $\Z\to \Z_i$ for each factor. Proposition 6.4 in \cite{HatVog-cerf} says that $SF_n$ is $(\frac{n-3}{2})$-connected. 

\begin{prop} 
Let $(\C,\op,0)=(\U f\G_{free},*,e)$. Then $S_n(e,\Z)$ is $(\tfrac{n-3}{2})$-connected. 
\end{prop}

\begin{proof}
$SF_n$ is weakly Cohen--Macaulay of dimension $(\frac{n-1}{2})$ (in the sense of Definition~\ref{CMdef}). Indeed, it is $(\frac{n-3}{2})$-connected \cite[Prop.~6.3]{HatVog-cerf}, and the link of $p$-simplex $\Z_0*\dots*\Z_{p}*H $ is isomorphic to the complex of split factorisations of $H$. But $H\cong \Z^{* n-p-1}$ by the uniqueness of free factorisation for finitely generated groups (Grushko decomposition theorem). Hence the link is isomorphic to $SF_{n-p-1}$ and hence is $(\frac{n-p-4}{2})$-connected, and so in particular $(\frac{n-1}{2}-p-2)$-connected. Now $S_n$ is a complete join over $SF_n$ in the terminology of \cite[Def.~3.2]{HatWah10}, with $(S_n)_x(\s)=\{+,-\}$ for each vertex $x$ and simplex $\s$ of $SF_n$. By Proposition 3.5 of that paper, we get that $S_n$ is also weakly Cohen--Macaulay, of the same dimension as $SF_n$, and in particular has the same connectivity.    
\end{proof}

Alternatively, we could have shown that $S_n(e,\Z)$ is isomorphic to the complex $V^{\pm}_{n,1}$ of \cite{HatVog04} and use Proposition 2 of that paper. 

\medskip 

It follows from Theorem \ref{caseA} that $W_n(\Z,\Z)_\bullet \cong W_{n+1}(e,\Z)_\bullet$ is $(\frac{n-2}{2})$-connected. From now on we shall not try to state every possible consequence of our Theorems \ref{stabthm}, \ref{abstabthm}, and~\ref{twistrange}, but rather focus on the results which are important or new. Applying Theorem~\ref{stabthm}, we recover the main theorem of \cite{HatVog-cerf}. Applying Theorem \ref{twistrange}, we get a stability theorem for automorphisms of free groups with twisted coefficients which does not seem to have been proved before: 

\begin{thm}\label{twistaut}
Let $F:\U f\G_{free}\to \bZ\operatorname{-Mod}$ be a coefficient system of degree $r$ at $0$. Then the map 
$$H_i(\Aut(F_n);F(F_n))\rar H_i(\Aut(F_{n+1});F(F_{n+1}))$$
is an epimorphism for $i\le \frac{n-1}{2}-r$ and an isomorphism for  $i\le \frac{n-3}{2}-r$. If $F$ is split then it is an epimorphism for $i\le \frac{n-r-1}{2}$ and an isomorphism for  $i\le \frac{n-r-3}{2}$.
\end{thm}

This establishes Theorem \ref{thm:G} in the case $G_n = \Aut(F_n)$. This was conjectured, in the split case with a similar range, by the first author in \cite[Conjecture A]{Ran10}. The stable homology with coefficients in a reduced finite degree coefficient system (i.e.\ one which vanishes on the trivial group) which factors through the category $gr$ of finitely generated free groups and all maps (via the forgetful
functor $\U f\G_{free}\to gr$) has been shown in \cite{DjaVes15} to be trivial. An example of such a functor is the abelianisation functor
$$Ab:\U f\G_{free}\rar \bZ\operatorname{-Mod} \ \ \ \textrm{with}\ \ \ Ab(F_n)= H_1(F_n;\bZ).$$
In general, though, the stable homology is non-trivial, as is for example
the stable homology with constant coefficients \cite{Gal11}. Another example where we know that the stable homology is non-trivial is the dual of the abelianisation functor: consider the functor $D:\U f \G_{free}\to gr^{op}$  which is the identity on objects and takes a morphism $(f,H):F_m\to F_n$ to the projection $F_n=H*f(F_m)\surj f(F_m)\sta{f^{-1}}\rar  F_m$. Now define 
$$Ab^{*}:=H^1(-;\Z)\circ D:\U f\G_{free}\rar \bZ\operatorname{-Mod}.$$ 
This is a degree 1 coefficient system, and Satoh \cite{Sat06} has computed $H_1(\Aut(F_n);Ab^*) \cong\Z$ for $n\ge 4$. 

Let $S\Aut(F_n)$ be the index two subgroup of $\Aut(F_n)$ of those automorphisms which induce an automorphism of determinant 1 on $H_1(F_n;\bZ)$. This is the commutator subgroup of $\Aut(F_n)$ as long as $n \geq 3$ (as $H_1(\Aut(F_n);\bZ)\cong \Z/2$ given by the determinant for $n\ge 3$), and Theorem \ref{abstabthm} gives the following.

\begin{thm}
The map
$$H_i(S\Aut(F_n);\bZ)\rar H_i(S\Aut(F_{n+1});\bZ)$$
is an epimorphism for $i\le \frac{n-2}{3}$ and an isomorphism for $i\le \frac{n-4}{3}$.
\end{thm}

Furthermore, the discussion in Section \ref{sec:StabHomology} and Galatius' theorem \cite{Gal11} identifying the stable homology of $\Aut(F_n)$ with that of $Q_0(S^0)$ shows that the stable homology of $S\Aut(F_n)$ is identified with that of the universal cover of $Q_0(S^0)$. We also have that the inclusion $A_n\to S\Aut(F_n)$ of the alternating groups into $S\Aut(F_n)$ induces an isomorphism in homology in the range of the theorem. Finally, Theorem \ref{twistrange} also establishes homological stability for the groups $S\Aut(F_n)$ with coefficients in a coefficient system of finite degree.
 
\medskip

Another interesting subgroup of $\Aut(F_n)$ is the symmetric automorphisms group $\Si\Aut(F_n)$, which will be discussed in Section~\ref{3mfdex} (see Remark~\ref{symaut}).

\subsubsection{More general free product stabilisation}\label{genfreeprod}

It does not seem to be known yet whether for each pair $(H,G)$ in the homogeneous category $\U f\G$, the associated semi-simplicial sets $W_n(H,G)_\bullet$ are highly-connected, though stability with constant coefficients is known in almost all cases (see \cite[Th{\'e}or{\`e}me 1.2]{ColDjaGri}), and expected to hold in general. It is likewise to be expected that these semi-simplicial sets are always highly-connected, but so far we only know it in the special cases treatable via the methods of \cite{HatWah10}, namely for groups that are fundamental groups of certain 3-manifolds. We briefly explain here how this can be seen. 

\medskip
 
Suppose that  $G=\pi_1P$ is the fundamental group of  an orientable prime 3-manifold $P\neq D^3$, whose diffeomorphism group surjects onto the automorphism group of its fundamental group. Examples of such groups are $\Z,\Z/2,\Z/3,\Z/4,\Z/6$, $\pi_1(S_g)$ for $S_g$ a closed surface of genus $g$, or $\pi_1(M)$ for $M$ hyperbolic, finite volume, with no orientation reversing isometry (see the introduction of \cite{HatWah10}). Pick likewise groups $H_i=\pi_1P_i$, for $1\le i\le k$, of the same sort. Now let $M=P_1\#\cdots\#P_k$ be the connected sum of such prime 3-manifolds, with 
 $H=\pi_1M=H_1*\dots*H_k$ and write $P=P_0$. Assume that we have the following additional properties 
\begin{enumerate}
\item $\del P_i\neq \emp$ for some $i\in \{1,\dots,k\}$ with $\del P_i=S^2$ for at most one $i$,
\item each $P_i$ which is not closed has incompressible boundary, i.e.~there is no disc inside $P_i$ with boundary a non-trivial curve in $\del P_i$, 
\item any two $P_i,P_j$ satisfying that $\pi_1 P_i\cong \pi_1 P_j$ also satisfy that $P_i\cong P_j$ by an orientation
preserving diffeomorphism.
\end{enumerate}
Then \cite[Prop.~2.1,2.2]{HatWah10} imply that the map 
$$\pi_0\Dif(M\# P\#\cdots\#P \ \textrm{rel}\ \del M)\to \Aut(\pi_1(M\# P\#\cdots\# P))=\Aut(H*G*\dots*G)$$
is surjective, with kernel generated by Dehn twists along 2-spheres. Note that one can always just take  $M=D^3$, giving $H=e$, with $P$ some prime manifold as above.

\begin{lem}
Let $(\C,\op,0)=(\U f\G,*,e)$.  For $H$ and $G$ satisfying the hypothesis above, the complex $S_n(H,G)$ is 
$(\frac{n-3}{2})$-connected. 
\end{lem}

\begin{proof}
For $G\neq \Z$, so that $P\neq S^1\x S^2$, the statement follows from \cite[Prop.~4.2]{HatWah10} as, under the assumption, the complex $S_n(G,H)$
identifies with the complex $X^A=X^A(M\#P\#\cdots\# P,P,\del_0M)$ of \cite[Sec.~4]{HatWah10}. 
Indeed, a vertex in $X^A$ is an isotopy class of  pairs $(f,a)$ of an embedding 
$$f:P^0:=P\minus D^3\ \inc \ M\#P\#\cdots\# P$$
with $f(\del P)\subset \del (M\#P\#\cdots\# P)\minus \del M$, together with an embedded arc 
$$a:I\ \inc \ M\#P\#\cdots\# P$$
satisfying that $a(0)=* \in \del_0M$ is a given basepoint and $a(I)\cap f(P^0)=a(1)=f(0)\in f(\del_0P^0)$ is the basepoint of the boundary sphere in $P^0$ coming from
the removed ball. A $p$-simplex is then a collection of $p+1$ such pairs such that the embeddings can be chosen disjoint. 

Now a pair $(f,a)$ induces a map $$G=\pi_1P=\pi_1P^0\ \rar \pi_1(M\#P\#\cdots\# P)=H*G*\cdots*G$$
where we take the basepoint of $P^0$ on its boundary sphere, and that of $M\#P\#\cdots\# P$ to be the marked point $*\in \del_0M$, and use the arc $a$
to identify the basepoints. 
Also, if we let $K=\pi_1(M\#P\#\cdots\# P\minus f(P^0))$, we have that $ \pi_1(M\#P\#\cdots\# P)=K*\pi_1(f(P^0))$. 
In particular, a vertex of $X^A$ determines a vertex of $S_n(H,G)$. This association defines an isomorphism of simplicial complexes 
$$X^A(M\#P\#\cdots\# P,P,\del_0M)\ \rar \ S_n(H,G).$$
Indeed, the map is surjective on the set of $p$--simplicies for every $p$ by the transitivity of the action of the automorphism group (the same on both side) and the fact that the ``standard'' $p$-simplex $\iota_H* \iota_{G^{*(n-p-1}}*G^{*(p+1)}$ is the image of the corresponding standard $p$--simplex in 
$X^A(M\#P\#\cdots\# P,P,\del_0M)$. Injectivity follows likewise from the fact that the stabilisers of the action are isomorphic, and identified by the map.

For $G=\Z$, we can take $P=S^1\x S^2$ and the proposition follows similarly from  \cite[Prop.~4.5]{HatWah10}. 
\end{proof}

Applying Theorem \ref{twistrange} to $(H*G,G)$ for $H$ and $G$ as above, we thus get the following generalisation of Theorem~\ref{twistaut}:

\begin{thm}\label{freetwist}
Let $F:\U f\G_{H,G}\to \bZ\operatorname{-Mod}$ be a coefficient system of degree $r$, and $H$ and $G$ be groups satisfying the hypothesis above. 
Then the map $$H_i(\Aut(H*G^{*n});F(H*G^{*n}))\rar H_i(\Aut(H*G^{* n+1});F(H*G^{* n+1}))$$
is an epimorphism for $i\le \frac{n-1}{2}-r$ and an isomorphism for $i\le \frac{n-3}{2}-r$. If $F$ is split then the map is an epimorphism for $i\le \frac{n-r-1}{2}$ and an isomorphism for $i\le \frac{n-r-3}{2}$.
\end{thm}

Note that by iterating the theorem, the statement holds more generally for $G$ a free product of such groups. Theorems \ref{abstabthm} and~\ref{twistrange} may also be used to obtain stability for the commutator subgroups $\Aut(H*G^{*n})'$ with constant or twisted coefficients.

\subsection{General linear groups}\label{GLnsec}

Let $R$ be a ring, and $(fR\operatorname{-Mod},\op,0)$ denote the groupoid of finitely generated right $R$-modules and their isomorphisms, with symmetric monoidal structure given by direct sum. As a convenient piece of language, in this section we shall say an $R$-module map is a \emph{splittable injection} if it admits a right inverse, and a \emph{split injection} if it is furthermore equipped with a choice of right inverse. The associated category $\U fR\operatorname{-Mod}$ has objects the finitely generated right $R$-modules, and the morphisms from $M$ to $N$ consist of pairs $(X, f)$ where $X$ is a finitely generated $R$-module and $f : X \oplus M \overset{\sim}\to N$ is an isomorphism. Alternatively, morphisms can be described as pairs of an injective map $g := f\vert_{M} : M \to N$ along with a submodule $H := f(X \oplus 0) \leq N$ which is a complement to $\mathrm{Im}(f)$; i.e.\ the morphisms are the split injections.

The groupoid $fR\operatorname{-Mod}$ has no zero divisors, but it will not generally satisfy the cancellation axiom C (though there are interesting special cases in which it does: for $R$ a PID or more generally a Dedekind domain, the classification of finitely generated modules implies cancellation). Consequently, neither will the associated category $\U fR\operatorname{-Mod}$ generally satisfy H1 (though it does satisfy H2).

It is for this reason that we have introduced the axiom LC of local cancellation, and the associated axioms LH1 and LH2. Recall that a vector $v\in R^n$ is called \emph{unimodular} if there exists a homomorphism $f:R^n\to R$ such that $f(v)=1$. Then a ring $R$ satisfies Bass' condition $(SR_{n})$ \cite[Definition V.3.1]{Bass} if for every $m \geq n$ and every  unimodular vector $(x_1, x_2, \ldots, x_m) \in R^{m}$, there are elements $r_1, \ldots, r_{m-1}\in R$ such that the vector $(x_1 + r_1 x_m, x_2+r_2 x_m, \ldots, x_{m-1} + r_{m-1} x_m) \in R^{m-1}$ is also unimodular. Following Vaserstein \cite{Vaserstein}, 
we define the \emph{stable rank} of $R$ to be
$$sr(R) := \min\{k \in \bN \, \vert \, R \text{ satisfies }(SR_{k+1})\}.$$
For example, $sr(R) \leq 2$ if $R$ is a PID. More generally, the basic estimate \cite[Theorem V.3.5]{Bass} is that if a ring $R$ is finitely generated as a module over a commutative ring whose spectrum of maximal ideals is Noetherian of dimension $d$, then $sr(R) \leq d+1$. In the following we assume that $sr(R)$ is finite.

\begin{rem}
If $sr(R)$ is finite then $R$ has the Invariant Basis Number property, and hence the full subgroupoid of $fR\operatorname{-Mod}$ consisting of the free modules satisfies the cancellation axiom C (and hence its associated homogeneous category satisfies H1 by Theorem \ref{universal}). This result is due to Veldkamp, but we could not find a reference and so we give a proof following Exercise I.1.5 (e) of \cite{Kbook}.

If $\phi : R^s \to R^t$ were an isomorphism with $s < t$, then by taking (stabilised) compositions we can find an isomorphism $\psi : R^{N} \to R^{N+n}$ with $n \geq sr(R)$, given by a $(N+n) \times N$ matrix $B$ which has a left inverse. By \cite[Theorem 3$^\prime$]{Vaserstein} we may left-multiply $B$ by an invertible matrix to put it in the form $\left( \begin{smallmatrix} B'\\ u \end{smallmatrix} \right)$ where $u$ is a $1 \times N$ matrix and $B'$ is a $(N+n-1) \times N$ matrix having a left inverse. As $B'$ has a left inverse, we may perform row operations to reduce $\left( \begin{smallmatrix} B'\\ u \end{smallmatrix} \right)$ to the form $\left( \begin{smallmatrix} B'\\ 0 \end{smallmatrix} \right)$, which still represents $\psi$ in some bases. But then $\psi$ is not surjective, a contradiction.
\end{rem}

\begin{prop}\label{prop:ModCancellation}
If $N \op R \cong R^{n+1}$ and $n \geq sr(R)$, then $N \cong R^n$.
\end{prop}
\begin{proof}
Let $\varphi : N \oplus R \to R^{n+1}$ be the isomorphism: this gives a unimodular vector $\varphi(0,1) = (x_1, x_2, \ldots, x_{n+1}) \in R^{n+1}$, and $(0,0,\ldots,0,1) \in R^{n+1}$ is also unimodular. The quotient of $R^{n+1}$ by the span of these vectors gives $N$ and $R^n$ respectively: thus if these vectors are in the same $GL_{n+1}(R)$-orbit then $N \cong R^n$, as required.

By assumption $R$ satisfies Bass' condition $SR_{n+1}$: this means that there are elements $r_1, \ldots, r_n \in R$ such that $(x_1 + r_1 x_{n+1}, \ldots, x_n + r_n x_{n+1})$ is unimodular. Thus we may operate on $(x_1, x_2, \ldots, x_{n+1})$ by an element of $GL_{n+1}(R)$ so that the new unimodular vector $(x'_1, x'_2, \ldots, x'_{n+1})$ has $x'_{n+1}=x_{n+1}$ and the sequence $(x'_1, x'_2, \ldots, x'_{n})$ is itself unimodular, that is, there are $y_1, \ldots, y_n$ for which $\sum_{i=1}^n y_i x'_i = 1$. Thus, by a suitable sequence of row operations we may transform $(x'_1, x'_2, \ldots, x'_{n+1})$ to $(x'_1, x'_2, \ldots, x'_n, 1)$, which may then easily be transformed to $(0,0,\ldots,0,1)$, as required.
\end{proof}

This implies that the pair $(R^{sr(R)},R)$ in $fR\operatorname{-Mod}$ satisfies the axiom LC. Hence, by Theorem \ref{universal}, the category $\U fR\operatorname{-Mod}$ satisfies LH1; it is easy to see that it satisfies LH2 (in fact, it satisfies H2), so is locally homogeneous at $(R^{sr(R)},R)$. To verify LH3, we shall deduce the connectivity of $W_n(R^{sr(R)},R)_\bullet$ from the connectivity of a similar complex considered by van der Kallen \cite{vdK80}. (Note that Charney considers in \cite{Cha84} a version of $W_n(R^{sr(R)},R)_\bullet$ for congruence subgroups. An alternative argument for the connectivity of the complex (yielding the same connectivity bound) can be obtained by adapting the proof of Theorem 3.5 in that paper.) 

\begin{lem}
The semi-simplicial set $W_n(R^{sr(R)},R)_\bullet$ is $(\frac{n-2}{2})$-connected. 
\end{lem}

\begin{proof}
Let $s:=sr(R)$. Define the semi-simplicial set $X(R^{s+n})_\bullet$ having $p$-simplices the splittable injections $f: R^{p+1} \to R^{s+n}$, (recall that this means an injective homomorphism which \emph{admits} a splitting, though none is chosen), and with $i$th face map given by precomposing with the inclusion $R^i \op 0 \op R^{p-i} \to R^{p+1}$. There is a semi-simplicial map $W_n(R^s,R)_\bullet \to X(R^{s+n})_\bullet$ given by forgetting the choice of complement in a split injection.  

Let $U(R^{s+n})$ be the simplicial complex whose vertices are the splittable injections $v : R \to R^{s+n}$, and where a tuple $v_0, v_1, \ldots, v_p$ span a $p$-simplex if and only if the sum $v_0 \op \cdots \op v_p : R^{p+1} \to R^{s+n}$ is a splittable injection. (Note that $v$ is a vertex of $U(R^{s+n})$ if and only if $v(1)$ is a unimodular vector in $R^{s+n}$, hence the name $U$.) This simplicial complex bears the same relation to $X(R^{s+n})_\bullet$ as $S_n(R^s,R)$ does to $W_n(R^s,R)_\bullet$, namely $X(R^{s+n})_\bullet = U(R^{s+n})^{ord}_\bullet$. In particular there is a commutative diagram
\begin{equation*}
\xymatrix{
\vert W_n(R^s,R)_\bullet \vert \ar[r] \ar[d]& \vert X(R^{s+n})_\bullet \vert \ar[d]\\
\vert S_n(R^s,R)\vert \ar[r] & \vert U(R^{s+n})\vert
}
\end{equation*}
of maps between geometric realisations. The poset of simplices of  $X(R^{s+n})_\bullet$ is equal to the poset $\mathcal{O}(R^{s+n}) \cap \mathcal{U}$ of \cite[\S 2]{vdK80}, which has been shown in Theorem 2.6 (i) of that paper to be $(n-1)$-connected. The nerve of this poset of simplices is the barycentric subdivision of $X(R^{s+n})_\bullet$, so they have homeomorphic geometric realisations, and so $\vert X(R^{s+n})_\bullet \vert$ is also $(n-1)$-connected. Choosing a total order of the vertices of $U(R^{s+n})$ defines a right inverse of $\vert X(R^{s+n})_\bullet \vert \to \vert U(R^{s+n})\vert$, and hence $U(R^{s+n})$ is also $(n-1)$-connected.

Similarly, for a $p$-simplex $\sigma = \{ v_0, \ldots, v_p \} \in U(R^{s+n})$ the poset of simplices of $(\link_{U(R^{s+n})}(\sigma))^{ord}_\bullet$ agrees with the poset $\mathcal{O}(R^{s+n}) \cap \mathcal{U}_{(v_0, \ldots, v_p)}$  of \cite[\S 2]{vdK80}, which has been shown in Theorem 2.6 (ii) of that paper to be $(n-1-p-1)$-connected. By the right inverse argument above it follows that $ \link_{U(R^{s+n})}(\sigma)$ is $(n-1-p-1)$-connected, and hence $U(R^{s+n})$ is weakly Cohen--Macaulay of dimension $n$.

\medskip

We wish to apply \cite[Theorem 3.6]{HatWah10} to the map of simplicial complexes $\pi: S_n(R^s,R) \to U(R^{s+n})$, and we claim that the map $\pi$ represents $S_n(R^s,R)$ as a \emph{join complex} over $U(R^{s+n})$ in the sense of \cite[Definition 3.2]{HatWah10}. We show this using the language of \emph{labelling systems} \cite[Example 3.3]{HatWah10}: for a vertex $v_0$ of a simplex $\{ v_0, v_1, \ldots, v_p \} \in U(R^{s+n})$, let $L_{v_0}(\{ v_0, v_1, \ldots, v_p \})$ be the set of complementary submodules for $v_0(R)$ which contain $v_1(R), \ldots, v_p(R)$. This defines a labelling system, and the associated simplicial complex $U(R^{s+n})^L$ which we show now is isomorphic to $S_n(R^s,R)$. 

The complex $U(R^{s+n})^L$ has vertices pairs $(v,H)$ with $v$ a vertex of $U(R^{s+n})$ and $H$ a choice of complement for $v$, and this is the same as the vertices of $S_n(R^s,R)$. Note that to such a vertex $(v,H)$ there is a canonically associated $R$-module map $\pi_v : R^{s+n} \to R$ given by projection away from $H$, i.e.~the chosen left inverse to $v$.  By definition, a collection of vertices $\{(v_0, H_0), (v_1, H_1), \ldots, (v_p, H_p)\}$ forms a simplex of $U(R^{s+n})^L$ if $\{ v_0,\dots, v_p\}$ is a simplex of $U(R^{s+n})$, i.e.~$v_0 \op \cdots\op v_p : R^{p+1} \to R^{s+n}$ is a splittable injection, and for each $i$ the complement $H_i$ contains $v_j(R)$ whenever $j\neq i$. Let $H := \cap_{i=0}^p H_i$. We want to show that  $(v_0,\dots , v_p, H)$ is a $p$-simplex of $W_n(R^s,R)_\bullet$, i.e.~that $H$ is a complement to the splittable injection $v_0 \op\cdots\op v_p : R^{p+1} \to R^{s+n}$. This is a consequence of the following two statements (where we identify $v_i$ and $v_i(1)$): 
\begin{enumerate}[(i)]
\item If $x = \sum a_i v_i \in \lgl v_0,\dots, v_p\rgl \cap H$, then $\pi_{v_i}(x)= a_i$ as $v_j \in H_i = \ker(\pi_{v_i})$ for $j \neq i$. But on the other hand $x \in H \subset H_i$ so $\pi_{v_i}(x)=0$. Thus $a_i =0$ for each $i$, and hence $x=0$, so $\lgl v_0,\dots, v_p\rgl \cap H = \{0\}$.

\item For $x \in R^{s+n}$ write $x = \sum_{i=0}^n \pi_{v_i}(x) v_i + y$ for some $y$. Then $\pi_{v_i}(y)=0$ for all $i$, so $y \in H_i$ for all $i$, and hence $y \in H$, so $\lgl v_0,\dots, v_p\rgl + H = R^{s+n}$.
\end{enumerate}
It follows that $\{(v_0, H_0), (v_1, H_1), \ldots, (v_p, H_p)\}$ is also a simplex of $S_n(R^s,R)$. 

Conversely, if $\{(v_0, H_0), (v_1, H_1), \ldots, (v_p, H_p)\}$ is a $p$-simplex of $S_n(R^s,R)$, there must be a $p$-simplex  
$(v_0, \ldots, v_p, H)$ of $W_n(R^s,R)$ with $H_i = H \oplus\bigoplus_{j\neq i} \lgl v_j\rgl$. In this case $\{ (v_0, H_0), (v_1, H_1),\dots, (v_p, H_p)\}$ is a simplex of $U(R^{s+n})^L$.

We have already shown that the target of $\pi$ is weakly Cohen--Macaulay of dimension $n$, so to apply \cite[Theorem 3.6]{HatWah10} it remains to show that $\pi(\link_{S_n(R^s,R)}(\sigma))$ is weakly Cohen--Macaulay of dimension $(n-p-2)$ for each $p$-simplex $\sigma \leq S_n(R^s,R)$.

Let $\sigma = \{ (v_0, H_0), \ldots, (v_p, H_p) \} \in S_n(R^s,R)$ be a $p$-simplex, with $p < n$. Now the simplicial complex $\link_{S_n(R^s,R)}(\sigma)$ is analogous to $S_n(R^s,R)$, but consisting of vectors and complements inside $H := \cap_i H_i$. By the isomorphism $R^{p+1} \op H \cong R^{s+n}$ and the inequality $n+sr(R)-p-1 \geq sr(R)$ coming from our assumption $p < n$, Proposition \ref{prop:ModCancellation} implies that $H \cong R^{s+n-p-1}$ and so $\link_{S_n(R^s,R)}(\sigma) \cong S_{n-p-1}(R^s,R)$. Thus $\pi(\link_{S_n(R^s,R)}(\sigma)) \cong U(R^{s+n-p-1})$ which by the above is weakly Cohen--Macaulay of dimension $(n-p-1)$ (so also weakly Cohen--Macaulay of dimension $(n-p-2)$). Note that if $p \geq n$ then there is nothing to show, so $\pi(\link_{S_n(R^s,R)}(\sigma))$ is weakly Cohen--Macaulay of dimension $(n-p-1)$.

We may now apply \cite[Theorem 3.6]{HatWah10}, showing $S_n(R^s,R)$ is $(\frac{n-2}{2})$-connected. Theorem \ref{caseA} hence implies that $W_n(R^s,R)_\bullet$ is $(\frac{n-2}{2})$-connected.
\end{proof}

One can obtain a slight improvement of the theorem in the case of Euclidian rings using \cite[Cor.~III.4.5]{Maazen} for the connectivity of $X(R^{s+n})_\bullet$. 

\medskip

Applying Theorem~\ref{twistrange} to $(\U fR\operatorname{-Mod},\op,0)$, we recover the following theorem of van der Kallen \cite{vdK80}.

\begin{thm}\label{thm:GL}
Let $F:\U fR\operatorname{-Mod} \to \bZ\operatorname{-Mod}$ be a coefficient system of degree $r$ at 0. Then the map 
$$H_i(GL_n(R);F(R^n))\rar H_i(GL_{n+1}(R);F(R^{n+1}))$$
is an epimorphism for $i\le \frac{n-sr(R)}{2}-r$ and an isomorphism for $i\le \frac{n-sr(R)-2}{2}-r$. If $F$ is split then the map is an epimorphism for $i\le \frac{n-sr(R)-r}{2}$ and an isomorphism for $i\le \frac{n-sr(R)-r-2}{2}$.
\end{thm}

The coefficient systems considered here are a slight generalisation of those of \cite{vdK80} in that we do not require the kernel to be zero and we do not suppose that the coefficient system is split (though our stability range is better if it is split). The range obtained by van der Kallen is better by one degree, using special properties of general linear groups (specifically, ``surjective stability for $K_1$") to start with an improved base case for the induction.

Finally, we remark that van der Kallen's choice \cite[\S 3.2]{vdK80} of intermediate group $E(R) \leq XL(R) \leq GL(R)$ corresponds to a choice of subgroup $U \leq H_1(GL(R);\bZ)$, as the elementary subgroup $E(R)$ is $GL(R)'$ by Whitehead's lemma. Thus van der Kallen's stability theorem, with constant or twisted coefficients, for $G_n := GL_n(R) \cap XL(R)$ may be recovered from our Theorems \ref{abstabthm} and~\ref{twistrange} as described in Section \ref{sec:Commutator}. However, van der Kallen obtains stability for these groups with slope $\tfrac{1}{2}$, whereas our theorem can only give slope $\tfrac{1}{3}$. This is again a special feature of general linear groups: the matrix $\left (\begin{smallmatrix} 0&-1\\ 1&0 \end{smallmatrix} \right)$ has stabilisation in $GL(R)'$ for any ring $R$ (i.e.~is trivial in $GL(R)^{ab}$), and sends the first copy of $R \op R$ to the second. Using this for the element $h_1$ rather than the element constructed from the braiding, the proof of Theorem \ref{twistrange} in the twisted abelian coefficients case can be done identically to the twisted coefficients case, and this argument gives slope $\tfrac{1}{2}$. The same can be done in the proof of Theorem \ref{abstabthm}.

\subsection{Unitary groups}\label{sec:Unitary}

Let $R$ be a ring with anti-involution $r \mapsto \overline{r}$, $\epsilon \in R$ be a central element such that $\epsilon \overline{\epsilon}=1$, and $\Lambda \subset R$ be an additive subgroup such that
$$\Lambda_{min} := \{r-\epsilon \overline{r} \, \vert \, r \in R\} \subset \Lambda \subset \Lambda_{max} := \{r \in R \, \vert \, \epsilon \overline{r} = -r\}$$
and $\overline{r} \Lambda r \subset \Lambda$ for all $r \in R$. The following definition agrees with the usual definition of a quadratic module when the underlying $R$-module is projective (cf.\ \cite[5.1.15]{HO}), but is more general.

\begin{Def}
A \emph{quadratic module} $(M, \lambda, \mu)$ over $(R,\epsilon,\Lambda)$ consists of a right $R$-module $M$, a sesquilinear form $\lambda : M \times M \to R$ (i.e.\ $R$-linear in the second variable and $R$-antilinear in the first), and a function $\mu : M \to R/\Lambda$, such that
\begin{enumerate}[(i)]
\item $\lambda$ is $\epsilon$-Hermitian, i.e.\ $\overline{\lambda(x,y)} = \epsilon \lambda(y,x)$,

\item $\mu(x+y) = \mu(x)+\mu(y) + [\lambda(x,y)]$, where $[\lambda(x,y)]$ is the class of $\lambda(x,y)$ modulo $\Lambda$,

\item $\mu(x r) = \overline{r} \mu(x) r$ for $r \in R$,

\item $\lambda(x,x) = \mu(x) + \epsilon\overline{\mu(x)}$, noting that the right-hand side gives a well-defined element of $R$, as $\Lambda \subset \Lambda_{max}$.
\end{enumerate}

A \emph{morphism} $f: (M, \lambda, \mu) \to (M', \lambda', \mu')$ of quadratic modules over $(R,\epsilon,\Lambda)$ is a right $R$-module homomorphism $f : M \to M'$ such that $\lambda'(f(x),f(y)) = \lambda(x,y)$ and $\mu'(f(x))=\mu(x)$. We let $(R,\epsilon,\Lambda)\operatorname{-Quad}$ denote the groupoid of quadratic modules over $(R,\epsilon,\Lambda)$ and their isomorphisms.
\end{Def}

The main example of a quadratic module we will consider, $H$, has underlying module $R^{\op 2}$ with basis $e:=(1,0)$ and $f:=(0,1)$, $\lambda$ given by $(\begin{smallmatrix} 0&1\\ \epsilon&0 \end{smallmatrix})$ in this basis, and $\mu(eA+fB) = \overline{A}B \mod \Lambda$. 
We then let $H^{\oplus n}$ be the orthogonal direct sum of $n$ copies of $H$, and, following \cite[Sec.~6]{MvdK}, let $U_n^\epsilon(R, \Lambda) := \Aut_{(R,\epsilon,\Lambda)\operatorname{-Quad}}(H^{\op n})$.

When $R$ is commutative, we can take the involution to be the identity. In this case, if $\epsilon=1$, we have $\Lambda_{min}=0\subset \Lambda_{max}=\{r\in R\ |\ 2r=0\}$ and $U_n^1(R,\Lambda_{min})=O_{n,n}(R)$. If $\epsilon=-1$, we have $\Lambda_{min}=2R\subset R=\Lambda_{max}$ and $U^{-1}_n(R,\Lambda_{max})=Sp_{2n}(R)$. For a non-trivial involution, $U^{-1}_n(R,\Lambda_{max})$ is the classical unitary group.

Mirzaii and van der Kallen have defined a condition $(US_n)$ for a tuple $(R,\epsilon,\Lambda)$, which is satisfied if $R$ satisfies Bass' condition $(SR_{n+1})$ and in addition for each $r \in R$ the subgroup of $U_{n+1}^\epsilon(R,\Lambda)$ generated by elementary matrices acts transitively on the set of unimodular vectors $x \in H^{\op n+1}$ with $\mu(x) = r + \Lambda$. The \emph{unitary stable rank} \cite[Def.~6.3]{MvdK} is then
$$usr(R,\epsilon,\Lambda) =  \min\{k \in \bN \, \vert \, (R,\epsilon,\Lambda) \text{ satisfies }(US_k)\}.$$
From now on we shall suppose that this is finite, and we shall shorten it to $usr(R)$.

Let us write $f(R,\epsilon,\Lambda)\operatorname{-Quad}$ for the subcategory of those quadratic modules which are finitely-generated as $R$-modules. This is a braided monoidal groupoid, has no zero-divisors, and has an associated pre-braided category $\U f(R,\epsilon,\Lambda)\operatorname{-Quad}$.

\begin{prop}
If $N \op H \cong H^{\op n+1}$ and $n \geq usr(R)$, then $N \cong H^{\op n}$.
\end{prop}
\begin{proof}[Proof sketch] 
Let $\phi: N \op H \to H^{\op n+1}$ be an isomorphism, and write $e, f$ for the hyperbolic basis for the copy of $H$ in the source and $e_1, f_1, \ldots, e_{n+1}, f_{n+1}$ for the hyperbolic bases for the copies of $H$ in the target. We follow the proof of \cite[Cor.~8.3]{MvdKV}. As $\phi(f)$ and $f_n$ are unimodular and $\mu (\phi(f)) = \mu(f) = 0+\Lambda = \mu(f_{n+1})$, by the transitivity part of $(US_{n})$ we may change $\phi$ by postcomposing with an automorphism and hence assume that $\phi(f)=f_{n+1}$. Then
$$\phi(e) = \sum_{i=1}^{n+1} e_i A_i + f_i B_i$$
with $A_{n+1}=1$. We then follow the reference in the proof of \cite[Corollary 8.3]{MvdKV} to Step 6 in the proof of \cite[Theorem 8.1]{MvdKV}, which gives an explicit sequence of automorphisms of $H^{\op n+1}$ which fix $f_{n+1}$ and take such a $\phi(e)$ to $e_{n+1}$. This gives a new isomorphism $\phi$ which is the identity on the last copy of $H$, so restricts to an isomorphism $N \cong H^{\op n}$ as required.
\end{proof}

This implies that the pair $(H^{\op k},H)$ in $f(R,\epsilon,\Lambda)\operatorname{-Quad}$ satisfies the axiom LC as long as $k \geq usr(R)$, so in particular by Theorem \ref{universal} the category $\U f(R,\epsilon,\Lambda)\operatorname{-Quad}$ satisfies LH1 at $(H^{\op usr(R)+1},H)$; it is easy to see that it satisfies H2, and hence LH2, so is locally homogeneous at $(H^{\op usr(R)+1},H)$. The following lemma establishes Axiom LH3.

\begin{lem}\label{connH}
The semi-simplicial set $W_n(H^{\op usr(R)+1}, H)_\bullet$ is $(\frac{n-2}{2})$-connected.
\end{lem}
\begin{proof}
The poset of simplices of the semi-simplicial set $W_n(H^{\op usr(R)+1}, H)_\bullet$ is equal to the poset $\mathcal{HU}(H^{\op usr(R)+n+1})$ of \cite[\S 7]{MvdK} (see also \cite[Sec.~3]{Cha87}), and hence they have homeomorphic geometric realisations. By \cite[Theorem 7.4]{MvdK} the poset $\mathcal{HU}(H^{\op usr(R)+n+1})$ is $(\frac{n-2}{2})$-connected. 
\end{proof}

When $R=\Z$ with the trivial involution, Theorem 3.2 of \cite{GalRW14} shows that $W_n(A,H)_\bullet$ is $(\frac{n-4}{2})$-connected for any quadratic module $A$. Even in the case of the trivial module $A$, this is a slight improvement of the above.

\medskip 

Applying Theorems \ref{stabthm} and ~\ref{twistrange} to $(\U f(R,\epsilon,\Lambda)\operatorname{-Quad}, \op, 0)$ gives the following.

\begin{thm}\label{thm:Unitary}
Let $F:\U f(R,\epsilon,\Lambda)\operatorname{-Quad} \to \bZ\operatorname{-Mod}$ be a coefficient system of degree $r$ at 0. Then the map 
$$H_i(U_n^\epsilon(R,\Lambda);F(H^{\op n}))\rar H_i(U_{n+1}^\epsilon(R,\Lambda);F(H^{\op n+1}))$$
is an epimorphism for $i\le \frac{n-usr(R)-1}{2}-r$ and an isomorphism for $i\le \frac{n-usr(R)-3}{2}-r$. If $F$ is split then the map is an epimorphism for $i\le \frac{n-usr(R)-1-r}{2}$ and an isomorphism for $i\le \frac{n-usr(R)-r-3}{2}$. If $F$ is constant then the map is an epimorphism for $i\le \frac{n-usr(R)-1}{2}$ and an isomorphism for $i\le \frac{n-usr(R)-2}{2}$.
\end{thm}

For constant coefficients this slightly improves the stability theorem of Mirzaii--van der Kallen \cite{MvdK}, and, as they remark, their results on the connectivity of the poset they use suffice to prove stability with twisted coefficients: that is what we have done. Earlier, stability with split coefficient systems of finite degree in the case of $R$ a Dedekind domain, trivial involution, $\epsilon=\pm 1$, and $\Lambda =\Lambda_{max}$ (that is, the groups $O_{n,n}(A)$ for $\epsilon=1$ and $Sp_{2n}(A)$ for $\epsilon=-1$) was proved by Charney \cite{Cha87}. In terms of the \emph{absolute stable rank} of Magurn--van der Kallen--Vaserstein \cite{MvdKV} we have $asr(R) \leq 2$ for $R$ Dedekind \cite[Theorem 3.1]{MvdKV}, and with trivial involution we have $usr(R) \leq asr(R)$ \cite[Remark 6.4]{MvdK}. Thus for split coefficient systems we get: the map in Theorem~\ref{thm:Unitary} is an epimorphism for $i\le \frac{n-r-3}{2}$ and isomorphism for $i\le \frac{n-r-5}{2}$, which improves \cite[Theorem 4.3]{Cha87} slightly. In the case of constant coefficients we get: epimorphism for $i\le \frac{n-3}{2}$ and isomorphism for $i\le \frac{n-4}{2}$, which improves \cite[Corollary 4.5]{Cha87} slightly. Stabilisation by quadratic modules other than $H$ has been considered by Vogtmann \cite{Vog82}, Cathelineau \cite{Cat07}, and Collinet \cite{Col11}.

We can also apply Theorems \ref{abstabthm} and \ref{twistrange} to study the (twisted) homology of the commutator subgroups $U_n^\epsilon(R,\Lambda)'$. There is defined \cite[II \S1]{BassUnitary} a subgroup $EU_n^\epsilon(R,\Lambda) \leq U_n^\epsilon(R,\Lambda)$ of elementary matrices, and by \cite[II Proposition 5.1]{BassUnitary} the group $EU_n^\epsilon(R,\Lambda)$ is perfect for $n \geq 3$, and so $EU_n^\epsilon(R,\Lambda) \leq U_n^\epsilon(R,\Lambda)'$. These groups are stably equal by the unitary Whitehead lemma \cite[II Theorem 5.2]{BassUnitary}, and the composition
$$\frac{U_n^\epsilon(R,\Lambda)}{EU_n^\epsilon(R,\Lambda)} \lra \frac{U_n^\epsilon(R,\Lambda)}{U_n^\epsilon(R,\Lambda)'} \lra \frac{U^\epsilon(R,\Lambda)}{U^\epsilon(R,\Lambda)'} = \frac{U^\epsilon(R,\Lambda)}{EU^\epsilon(R,\Lambda)}$$
has been shown \cite{Sinchuk} to be an isomorphism for $n \geq sr(R)+1$. As the first map is surjective anyway, it follows that $EU_n^\epsilon(R,\Lambda) = U_n^\epsilon(R,\Lambda)'$ for $n \geq sr(R)+1$, so also for $n \geq usr(R)+1$.

\begin{thm}\label{thm:ElementaryUnitary}
Let $F:\U f(R,\epsilon,\Lambda)\operatorname{-Quad} \to \bZ\operatorname{-Mod}$ be a coefficient system of degree $r$ at 0. Then the map 
$$H_i(EU_n^\epsilon(R,\Lambda);F(H^{\op n}))\rar H_i(EU_{n+1}^\epsilon(R,\Lambda);F(H^{\op n+1}))$$
is an epimorphism for $i\le \frac{n-usr(R)-2}{3}-r$ and an isomorphism for $i\le \frac{n-usr(R)-5}{3}-r$. If $F$ is split then the map is an epimorphism for $i\le \frac{n-usr(R)-2-2r}{3}$ and an isomorphism for $i\le \frac{n-usr(R)-5-2r}{3}$. If $F$ is constant then the map is an epimorphism for $i\le \frac{n-usr(R)-2}{3}$ and an isomorphism for $i\le \frac{n-usr(R)-4}{3}$.
\end{thm}

As in the case of general linear groups, the automorphism $\left (\begin{smallmatrix} 0&-\mathrm{Id}_H\\ \mathrm{Id}_H&0 \end{smallmatrix} \right)$ of $H^{\op 2}$ sends the first copy of $H$ to the second and has stabilisation in $U^\epsilon(R,\Lambda)'$ for any $(R,\epsilon, \Lambda)$. It can be used for $h_1$ in the proofs of Theorems \ref{abstabthm} and~\ref{twistrange} to improve the stability ranges to have slope $\tfrac{1}{2}$. Theorems \ref{thm:Unitary} and \ref{thm:ElementaryUnitary} imply Theorem \ref{thm:H}.

\subsection{Automorphism groups of direct products of groups}\label{directprodsec}

Consider again the groupoid $f\G$ of finitely generated groups and their isomorphisms, but now with the monoidal structure induced by direct product of
groups. Then $(f\G,\x,e)$ is a symmetric monoidal groupoid, and for every $H$ and $G$ in $f\G$,  the map $\Aut(H)\to \Aut(H\x G)$ extending automorphisms by the identity is injective. It has no zero divisors but it is however no longer the case that $f\G$ satisfies the cancellation property. Indeed, there exist for example finitely generated groups $G$ such that $G\cong G\x G$, or, more generally, groups isomorphic to a proper direct factor of themselves (see
\cite{Tyr74}). There are however natural subcategories of groups which do satisfy the cancellation property with respect to direct product, such as for example finitely generated abelian groups, right angled Artin groups, or finite groups. But the connectivity of the associated semi-simplicial sets is not always known, even when cancellation
holds. 

Homological stability  for finitely generated groups with respect to direct product is known in the case of products of centre-free groups: this follows from
Johnson's description of the automorphism groups of such groups \cite{Joh83} and the known stability for wreath products of symmetric groups with any other group
\cite[Prop.~1.6]{HatWah10}. On the other extreme, Theorem \ref{thm:GL} gives the result for abelian groups, as in this case the automorphism group is simply a general linear group (see \cite[Prop.~5.2]{GanWah15}). Stability in the case of right angled Artin groups is proved in \cite{GanWah15}. We conjecture that stability holds in general: 

\begin{con}
Let $H$ and $G$ be finitely generated groups, and let $F:\U f\G \to \bZ\operatorname{-Mod}$ have finite degree. Then the map 
$$H_i(\Aut(H\x G^n);F(H\x G^n)) \ \rar \ H_i(\Aut(H\x G^{n+1});F(H\x G^{n+1}))$$
is an isomorphism for $n$ large. 
\end{con}

\subsection{Braid groups and mapping class groups of surfaces}\label{surfaceex}

Let $\M_2$ denote the groupoid of  {\em decorated surfaces} $(S,I)$, where $S$ is a compact connected surface with at least one boundary component and $I:[-1,1]\inc \del S$ is a parametrised interval in its boundary (as in Figure~\ref{surfacesum}(a)). The morphisms in $\M_2$ are the isotopy classes of diffeomorphisms restricting to the identity on a neighbourhood of $I$. 
Note that the mapping class group of $S$ identifies with the group of isotopy classes of homeomorphisms of $S$, see e.g.\ \cite[Sec.~1.4.2]{FarMar12}. 
We will here work with smooth surfaces, though allowing corners in their boundary, and sometimes represent mapping classes using diffeomorphisms and sometimes using homeomorphisms, depending on which is most convenient for a given construction. 

We will show below that boundary connected sum induces a monoidal structure on $\M_2$ which is braided, and hence, by Proposition~\ref{braidandsym}, the associated category $\U\M_2$ is pre-braided. This category satisfies H2, though not H1. Indeed, a pair $(A,X)$ satisfies local cancellation LC (and so, by  Theorem~\ref{universal} (a), $\U\M_2$ satisfies LH1) if $X$ is orientable, but not if $X$ is non-orientable. In the latter case, LC only holds if we restrict to the subgroupoid $\M_2^{-}$ of non-orientable or genus 0 surfaces. We describe in the first section the monoidal categories $\M_2$ and $\M_2^{-}$, and give an alternative description of the associated categories $\U\M_2$ and $\U\M_2^{-}$ in the following section. We shall then show that these categories satisfy LH3 at $(A,X)$ in all the interesting cases, namely when $X$ is a cylinder or a punctured torus, or when $X$ is a M\"obius band and $A$ is any surface. For surface braid groups, we will also consider a subgroupoid $\cB_2$ of surfaces and ``braided morphisms''.

\subsubsection{The groupoids of surfaces $\M_2$, $\M_2^-$, and $\cB_2$} 
We want to show that the groupoid $\M_2$ defined above, as well as two subgroupoids $\M_2^-$ and $\cB_2$, is braided monoidal and satisfies the assumptions of Theorem~\ref{universal}. We start by recalling a few facts about isotopy classes of diffeomorphisms.

Given that, up to isotopy, fixing an interval in a boundary component is the same as fixing the whole boundary component, the endomorphisms of an object $(S,I)$  in $\M_2$ identifies with the mapping class group of $S$ fixing the boundary component $\del_0S$ containing $I$ pointwise. The other boundary components are freely moved by the mapping classes and can also be thought of as punctures. When $S$ is orientable, the orientation of $I$ specifies an orientation on $S$ and such mapping classes automatically preserve that orientation. Note also that the group $\Aut_{\M_2}(D^2,I)$ is trivial (by the Alexander trick and the fact that the mapping class defined using homeomorphisms instead is the same, see e.g.\ \cite[Lem.~2.1]{FarMar12}).  

A diffeomorphism $(S_1,\del_0 S_1)\arsim (S_2,\del_0S_2)$ induces a diffeomorphism $(\widetilde S_1,\del \widetilde S_1)\arsim (\widetilde S_2,\del\widetilde S_2)$ where $\widetilde S_i$ is obtained from $S_i$ by gluing a disc on all the boundary components of $S_i$ but $\del_0 S_i$. 
We will denote by $\cB_2$ the subgroupoid of $\M_2$ with the same objects and with morphisms those that become trivial when gluing discs in this way. 
By the parameterised isotopy extension theorem (see \cite[II 2.2.2 Corollaire 2]{Cer61}
), there is a homotopy fibre sequence
$$\Dif(S \ \textrm{rel}\ \del_0S)\rar \Dif(\widetilde S \ \textrm{rel}\ \del\widetilde S) \rar \Emb(\{1,2,\ldots,k\},\widetilde S\minus \del \widetilde S)/\Sigma_k$$
for $k+1$ the number of boundary components of $S$, where we can think of the space $\Emb(\{1,2,\ldots,k\}, \widetilde S\minus \del \widetilde S)/\Sigma_k$ as the space $\operatorname{Conf}(\mathrm{int}\widetilde S,k)$ of unordered configurations of $k$ points in the interior of $\widetilde S$.
Now using the associated long exact sequence of homotopy groups and the contractibility of the component of the diffeomorphism groups \cite[Thm.~1]{Gra73}, 
we get that the automorphism group of an object $(S,I)$ in $\cB_2$ is the surface braid group $\beta^{\widetilde{S}}_k=\pi_1 \operatorname{Conf}(\mathrm{int}\widetilde S,k)$. 

Boundary connected sum induces a monoidal product on $\M_2$:  
Given decorated surfaces $(S_1,I_1) $ and $(S_2,I_2)$, we define $(S_1\natural S_2,I_1\natural I_2)$ 
to be the surface obtained by gluing $S_1$ and $S_2$ along the right half-interval $I^+_1\in \del S_1$  and the left half-interval $I_2^-\in \del S_2$, defining $I_1\natural I_2=I_1^-\cup I_2^+$.  (See Figure~\ref{surfacesum} (b).) This product is an associative version of the pairs of pants multiplication: a neighbourhood of $\del_0S_1\cup \del_0S_2$ in $S_1\natural S_2$ is a pair of pants which may be assumed to be fixed by $\Aut(S_1,I_1)\x \Aut(S_2,I_2)$. Just like the pairs of pants multiplication, it is braided. (The braiding can be constructed by doing half a Dehn twist in this pair of pants neighbourhood of $\del_0S_1$ and $\del_0S_2$, as shown in Figure~\ref{surfacesum} (c).) The unit is the disc $(D^2,I)$. For it to be a strict unit, we define the sum with the disc to be the identity, i.e.~we set $(S_1\natural D^2,I_1\natural I):=(S_1,I_1)$ and $(D^2\natural S_2,I\natural I_2):=(S_2,I_2)$. This extends to morphisms in $\M_2$ as $(D_2,I)$ has no non-trivial automorphism. 
\begin{figure}[h]
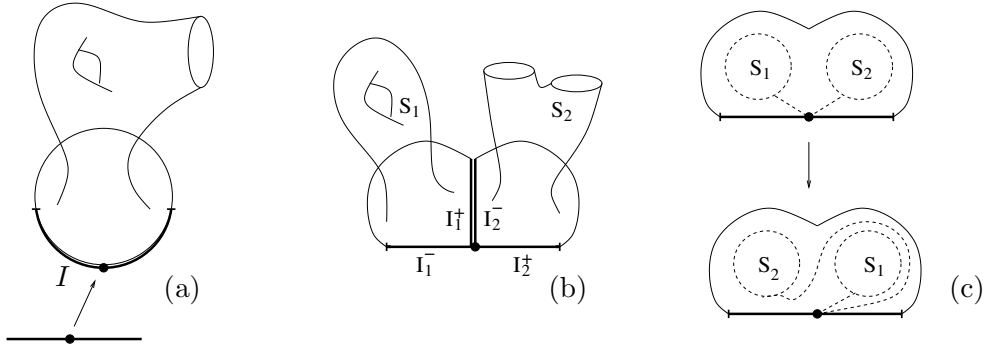

\begin{lpic}{surfacesum3(0.35,0.35)}
\lbl[b]{67,15;(a)}
\lbl[b]{22,22;$I$}
\lbl[b]{212,15;(b)}
\lbl[b]{362,15;(c)}
\end{lpic}
\caption{Braided monoidal structure for decorated surfaces}\label{surfacesum}
\end{figure}
This structure restricts to a braided monoidal structure on the  subgroupoid $\cB_2$ defined above. We will also consider the subgroupoid $\M_2^-$, again braided monoidal, which is the full subgroupoid of $\M_2$ with objects the non-orientable or genus 0 surfaces. Note that,  by the classification of surfaces, there are no zero divisors in $\cB_2,\M_2^-$ and $\M_2$. 

\begin{prop}\label{M2hom} Let $\cB_2\subset \M_2\supset \M_2^-$ be the groupoids described above. We have that: 
\begin{enumerate}[(i)]
\item The category $\U\M_2$ is pre-braided with underlying groupoid $\M_2$. It is locally homogeneous at $(A,X)$ for any orientable surface $X$.   
\item The category $\U\cB_2$ is pre-braided with underlying groupoid $\cB_2$.  It is locally homogeneous at $(A,X)$ for any orientable surface $X$. 
\item The category $\U\M_2^-$ is pre-braided homogeneous with underlying groupoid $\M_2^-$. 
\end{enumerate}
\end{prop}

\begin{proof}
As the above three groupoids are braided monoidal, by Proposition~\ref{braidandsym} we have that $\U\cB_2$, $\U\M_2$ and $\U\M_2^-$  are pre-braided monoidal categories. As $(D^2,I)$ has no non-trivial automorphisms and there are no zero divisors in $\M_2$, Proposition~\ref{underlying} implies that the underlying groupoids of $\U\cB_2$, $\U\M_2$, and $\U\M_2^{-}$ are $\cB_2,\M_2$ and $\M_2^{-}$. From the classification of surfaces, we have that $\M_2$, and hence $\cB_2$, satisfies local cancellation at any $(A,X)$ with $X$ orientable, and $\M_2^-$ satisfies cancellation: the only non-trivial case is if $X$ is orientable and $A$ non-orientable, in which case the existence of an isomorphism $Y \natural X \cong A \natural X^{\natural k}$ implies that $Y$ must be non-orientable and hence isomorphic to $ A \natural X^{\natural k-1}$ by the classification of surfaces, which verifies local cancellation.  
Finally, there is a fibration sequence
\begin{equation}\label{eq:FibSeq}
\Dif(S_1\natural S_2 \ \textrm{rel}\ S_2\cup \del_0)\lra \Dif(S_1\natural S_2 \ \textrm{rel}\ \del_0) \lra \Emb((S_2,I_2^+),(S_1\natural
S_2,I_2^+))
\end{equation}
(using again \cite[II 2.2.2 Corollaire 2]{Cer61})
and we may identify $\Dif(S_1\natural S_2 \ \textrm{rel}\ S_2\cup \del_0)$ with $\Dif(S_1 \ \textrm{rel}\ \del_0)$. The base of this fibration has contractible components. This is more generally true for $\Emb((S,I_2^+),(T,I_2^+))$ with $S \subset T$ and may be seen by induction over a handle decomposition of $S$ using a theorem of Gramain \cite[Thm.~5]{Gra73}. More precisely, if $S = S' \cup_{\partial D^1 \times D^1} D^1 \times D^1$ is obtained by attaching a handle to a surface $S'$, and $I_2^+ \subset S'$, then there is a fibration sequence
$$\Emb((S, S'),(T, S')) \lra \Emb((S,I_2^+),(T,I_2^+)) \lra \Emb((S',I_2^+),(T,I_2^+))$$
and $\Emb((S, S'),(T, S')) \cong \Emb((D^1 \times D^1, \partial D^1 \times D^1), (T \setminus \mathrm{int}S',\partial D^1 \times D^1))$. The restriction map to the space $\Emb((D^1 \times \{0\}, \partial D^1 \times \{0\}), (T \setminus \mathrm{int}S',\partial D^1 \times \{0\}))$ of embedded arcs is a weak homotopy equivalence onto the path components which it hits, because the space of thickenings of an embedded arc in a surface is homotopy equivalent to a space of paths on $GL_1(\bR)$, so either empty or contractible. This space of embedded arcs has contractible components by \cite[Thm.~5]{Gra73}. This inductively proves the contractibility of the components of $\Emb((S,I_2^+),(T,I_2^+))$. 

The fact that the map $\Aut_{\M_2}(S_1,I_1)\to \Aut_{\M_2}(S_1\natural S_2,I_1\natural I_2)$ is injective then follows using the long exact sequence of homotopy groups associated to the fibration sequence \eqref{eq:FibSeq} and the contractibility of the components of the base space. This gives injectivity for the automorphisms when we replace $\M_2$ by $\M_2^-$ or $\cB_2$ as both are subgroupoids. The result then follows from Theorem~\ref{universal}. 
\end{proof}

Unfortunately, $\U\M_2$ will not satisfy LH1 for any pair $(A,X)$ with $X$ non-orientable: by the classification of surfaces one can find an isomorphism $\phi: Y \natural X \overset{\sim}\to A \natural X^{\natural k}$ with $k \geq 2$ and $Y$ orientable, but then $Y$ cannot be isomorphic to $A \natural X^{\natural k-1}$ (and so $[Y, \phi]$ cannot be in the same $\Aut_{U\M_2}(A \natural X^{\natural k})$-orbit as the standard map from $X$ to $A \natural X^{\natural k}$). We thus need to restrict to the groupoid of non-orientable and genus 0 surfaces and corresponding homogeneous category $\U\M_2^-$ for stabilisation with $X$ non-orientable. As we shall see, this corresponds to what was done previously in proofs of homological stability for non-orientable surfaces. 

\medskip

By the classification of surfaces, any object in the categories of surfaces occurring here can be obtained from the disc by taking sums of copies of the following three basic building blocks: the cylinder, the torus with one boundary component, and the M\"{o}bius band. Indeed, boundary connected sum with the torus/M\"{o}bius band increases the genus/non-orientable genus, and boundary connected sum with the cylinder increases the number of  boundary components. As the surfaces are assumed to be connected, any surface in the category can be obtained that way. 
This means that these three objects are the interesting possible $X$'s one can stabilise with. We consider below these three possible stabilisations and verify LH3 in each case. To do this, we will identify the complexes $S_n(X,A)$ with complexes earlier studied in the literature. We start by giving a simpler description of the morphisms in the categories $\U\M_2$ and $\U\M_2^-$.

Let $(S_1,I_1)$ and $(S_2,I_2)$ be decorated surfaces and $[T,f]:(S_1,I_1)\to (S_2,I_2)$ a morphism in $\U\M_2$, with $f:T\natural S_1\arsim S_2$ the mapping class of a diffeomorphism (Fig.~\ref{surfacemorphism}). Consider $f|_{S_1}:S_1\to S_2$. This map is an embedding taking $I_1^+$ to $I_2^+$ and such that the complement of $S_1$ in $S_2$ is diffeomorphic to $T\minus I_T^+$ (except when $T=D_2$, in which case the complement is empty). It also has the property that it takes the boundary components $\del'S_1=\del S_1\minus \del_0S_1$ to components of $\del'S_2$ or in other words that the image of $S_1$ in $S_2$ is separated from its complement by a single arc, namely the image of $I_1^-$. 
We denote by $\Emb((S_1,\del'S_1;I_1^+),(S_2,\del'S_2;I_2^+))$ the space of such embeddings and $\Emb^-((S_1,\del'S_1;I_1^+),(S_2,\del'S_2;I_2^+))$ the subspace of such embeddings with non-orientable complement. (To be precise, an element of this space will take the arc $I_1^-$ to an arc in $S_2$ meeting the $\del S_2$ only at its endpoints, unless $S_1$ and $S_2$ are diffeomorphic, in which case $I_1^-$ is taken to $I_2^-$.)  
\begin{figure}[h]
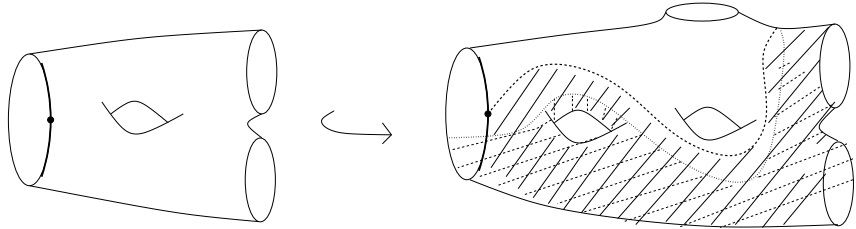

\begin{lpic}{surfacemorphism3(0.25,0.25)}
\end{lpic}
\caption{A morphism in $U\M_2$}\label{surfacemorphism}
\end{figure}

As the equivalence class $[T,f]$ is determined by $f|_{S_1}$, we have the following: 

\begin{lem}
Let $S_1,S_2$ be surfaces. Then
\begin{enumerate}[(i)]
\item $\Hom_{\U\M_2}((S_1,I_1),(S_2,I_2))=\pi_0\Emb((S_1,\del' S_1;I_1),(S_2,\del'S_2;I_2)).$
\item For $S_1,S_2\in \M_2^-$, 
$$\Hom_{\U\M^-_2}((S_1,I_1),(S_2,I_2))=\pi_0\Emb^-((S_1,\del' S_1;I_1),(S_2,\del'S_2;I_2)).$$
\end{enumerate}
\end{lem}

Note that the lemma is also a consequence of a fibration of the type occurring e.g.~in the proof of Proposition~\ref{M2hom}. 

\begin{rem}
The category $U\M_2$ is closely related to the category of decorated surfaces defined by Ivanov in \cite[2.5]{Iva93} whose objects are orientable decorated surfaces, and the morphisms are pairs of an orientation preserving embedding $S_1\inc S_2$ and an arc between $I_1$ and $I_2$. 
\end{rem}

\subsubsection{Braid groups and stabilisation by punctures}\label{braidex}

Consider the pair $(A,X)$ of objects in $\U\M_2$ given by the cylinder
$$X=(S^1\x [0,1],I)$$ 
(which can be thought of as a once punctured disc), and any surface
$$A=(S,J).$$
The simplicial complex $S_n(A,X)$ has vertices the set of morphisms from $X$ to $A\natural X^{\natural n}$, i.e.~the set of isotopy classes of embeddings 
$$S^1\x [0,1]\cong D^2\minus \mathrm{int}{D}^2 \ \ \inc\ \ S\ \natural\ (S^1\x [0,1])^{\natural n}\cong S\minus (\sqcup_n \mathrm{int}{D}^2)=:S_{(n)}$$ 
taking the interval $I^+$ in the first boundary component of the cylinder to $J^+$ in $S$, and taking the other boundary component of the cylinder to some other boundary in $S_{(n)}$.  (See Figure~\ref{braidmorphism} for an example.)  Such an isotopy class of embeddings $f$ is determined by the isotopy class of the arc $f(\{0\}\x [0,1])$. 
Indeed, two embeddings $f_1,f_2$ yielding isotopic arcs can be isotoped to have the same image, namely a tubular neighbourhood of one of the arcs, and to agree on $I\cup \{0\}\x [0,1]$. Then they can be isotoped to agree on $S^1\x\{0,1\}\cup I$ by the Alexander trick in dimension 1, and then on the whole of $X$ by the Alexander trick in dimension 2 as the complement of $\{0\}\x [0,1]$ in $X$ is a disc.  
Note also that any isotopy class of arcs from the mid-point of $J$ to $\del S_{(n)}\minus \del_0 S_{(n)}$ determines such an embedding by picking a neighbourhood of the arc in $S$. 
More generally, a morphism from $X^{\natural k}$, which is a $k$-legged pair of pants, to $A\natural X^{\natural n}$ is determined by a collection of $k$ embedded disjoint such arcs. 
\begin{figure}
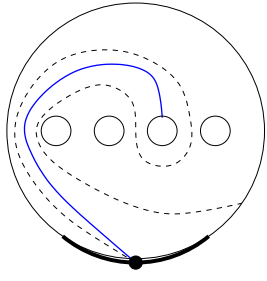

\begin{lpic}{braidmorphism2(0.55,0.55)}
\end{lpic}
\caption{A vertex in $S_4(D^2,X)$ for $X$ a cylinder, and arc representing it.}\label{braidmorphism}
\end{figure}
Note now that for this particular $X$, we have an isomorphism 
$$\Hom_{\U\cB_2}(X^{\natural p+1},S_{(n)})\cong \Hom_{\U\M_2}(X^{\natural p+1},S_{(n)})$$
as $\Hom_{\U\cB_2}(S,T)\cong\ker(\Hom_{\U\M_2}(S,T) \to \Hom_{\U\M_2}(\widetilde S,\widetilde T))$.  
So the complexes $S_n(A,X)$ and semi-simplicial sets $W_n(A,X)$ associated to $U\cB_2$ and $U\M_2$ are isomorphic in this case, a fact that was already used in a way in \cite{HatWah10} to prove stability for surface braid groups. 

The above gives an identification of this simplicial complex $S_n(A,X)$ in the present case with the arc complex denoted 
$A(\widetilde S;\{*\},\Lambda)$ in \cite{HatWah10}, where $\Lambda$ is  the set of centres of the discs glued onto $S$ to form $\widetilde S$.  
A connectivity bound for that complex is computed in that paper: 

\begin{lem}\cite[Prop.~7.2]{HatWah10}
For $X=(S^1\x [0,1],I)$ the cylinder and $A$ any object in $\U\M_2$ (or equivalently $\U\cB_2$), we have that $S_n^{\U\M_2}(A,X)\cong S_n^{\U\cB_2}(A,X)$ is $(n-2)$-connected. 
\end{lem}

(The complex $S_n(A,X)$ is actually contractible when $A=D^2$ by \cite[Thm.~2.48]{Dam13} or \cite[Prop.~3.2]{HatVoginfinity}, but this does not improve our stability results.) Now, the semi-simplicial set $W_n(A,X)_\bullet= W_n^{\U\M_2}(A,X)_\bullet=W_n^{\cB_2}(A,X)_\bullet$ satisfies condition (B) of Section \ref{sec:SxCxSSets}, as any simplex of $S_n(A,X)=S_n^{\U\M_2}(A,X)=S_n^{\cB_2}(A,X)$ has a canonical ordering of its vertices induced by the local orientation of the surfaces near the parameterised interval in their boundary. Thus $\vert W_n(A, X)_\bullet \vert$ is homeomorphic to $S_n(A,X)$, so is also $(n-2)$-connected.

\medskip

When $X=(S^1\x [0,1],I)$ and $A=(S,J)$ as above,  we have that 
$$\Aut_{\U\M_2}(A\op X^{\op n})\cong\pi_0\Dif( S\minus (\sqcup_n \mathrm{int}{D}^2)\ \textrm{rel}\ \del_0S)=\pi_0\Dif(S_{(n)}\ \textrm{rel}\ \del_0S)$$
is the mapping class group of $S$ punctured $n$ times fixing $\del_0S$, while, under the additional assumption that $S$ has a single boundary component, 
$$\Aut_{\U\cB_2}(A\natural X^{\natural n})\cong\pi_1\operatorname{Conf}(\mathrm{int}S,n)$$
is the surface braid group $\beta_n^S$. When $A$ is a disc, both groups identify with the classical braid group on $n$ strands. 

The following is obtained by applying Theorems \ref{stabthm} and~\ref{twistrange} to such $X$ and $A$.

\begin{thm}
Let $F:(\U\M_2)_{S,S^1\x I}\to \Z\operatorname{-Mod}$ be a coefficient system, which may be constant. Then the map
$$H_i(\pi_0\Dif(S_{(n)} \textrm{rel}\ \del_0S);F(S_{(n)})) \rar H_i(\pi_0\Dif(S_{(n+1)}\textrm{rel}\ \del_0S);F(S_{(n+1)}))$$
is: an epimorphism for $i \leq \tfrac{n}{2}$ and an isomorphism for $i \leq \tfrac{n-1}{2}$, if $F$ is constant; an epimorphism for $i \leq \tfrac{n-r}{2}$ and an isomorphism for $i \leq \tfrac{n-r-2}{2}$, if $F$ is split of degree $r$; an epimorphism for $i \leq \tfrac{n}{2}-r$ and an isomorphism for $i \leq \tfrac{n-2}{2}-r$, if $F$ is of degree $r$. Moreover, for a coefficient system $G:(U\cB_2)_{S,S^1\x I}\to \Z\operatorname{-Mod}$ the same holds for the map 
$$H_i(\pi_1\operatorname{Conf}(\mathrm{int}S,n);G(S_{(n)})) \rar H_i(\pi_1\operatorname{Conf}(\mathrm{int}S,n+1);G(S_{(n+1)}))$$
under the additional assumption that $S$ has a single boundary component. 
\end{thm}

For constant coefficients and $S=D^2$, this last theorem is due to Arnold \cite{Arn70}, and for general $S$ the first part of the theorem is the dimension 2 case of Proposition 1.5 in \cite{HatWah10}, while the second part is the dimension 2 case of \cite[Prop.~A.1]{Seg79}. 

An example of a degree 1 coefficient system to which our result can be applied 
is the Burau representation (see Examples \ref{Bureauex} and \ref{Bureau2}). The homology of the braid groups with coefficients in the reduced complexified Burau representation has recently been computed by Chen \cite{Che15}, and it follows from his computation that the stability slope for the rationalised Burau representation is actually 1. We do not know whether our integral slope is optimal, but such a difference between the rational and integral slopes is typical for constant coefficients.
  
A stability result has been proved by Church--Farb \cite[Cor.~4.4]{ChuFar13} for the twisted homology of braid groups for certain rational coefficient systems which factor through the category $FI$, with slope $\tfrac{1}{4}$.

\medskip

To state our results for abelian coefficients, let us restrict to the case $S=D^2$, where the group $\pi_0\Dif(S_{(n)} \textrm{rel}\ \del_0S)$ is identified with the braid group $\beta_n$. In this case it is well-known that the abelianisation of $\beta_n$ is $\bZ$ as long as $n \geq 2$, with $\beta_n \to \bZ$ given by the total winding number, and stability results for $\beta_n$ with abelian (twisted) coefficients give stability results for the commutator subgroup $\beta'_n$ with (twisted) coefficients. Theorems \ref{abstabthm} and~\ref{twistrange} give the following.

\begin{thm}\label{thm:CommutatorBraid}
Let $F:(\U\M_2)_{D^2,S^1 \times I}\to \Z\operatorname{-Mod}$ be a coefficient system, which may be constant. Then the map  
$$H_i(\beta'_n;F(D^2_{(n)})) \rar H_i(\beta'_{n+1};F(D^2_{(n+1)}))$$
is: an epimorphism for $i\le \frac{n-1}{3}$ and an isomorphism for $i\le \frac{n-3}{3}$, if $F$ is constant; an epimorphism for $i \leq \tfrac{n-2r-1}{3}$ and an isomorphism for $i \leq \tfrac{n-2r-4}{2}$, if $F$ is split of degree $r$; an epimorphism for $i \leq \tfrac{n-1}{3}-r$ and an isomorphism for $i \leq \tfrac{n-4}{3}-r$, if $F$ is of degree $r$.
\end{thm}

It is well-known that $(B\beta_\infty)^+ \simeq \Omega^2 S^3$, which may be proved by group-completion \cite{Segal}. Hence, the discussion of Section \ref{sec:StabHomology} shows that the stable homology of $\beta'_n$ with constant coefficients may be described as the homology of the universal cover of $\Omega^2 S^3$ (cf.\ \cite{CohenPakianathan}), and so for example has trivial rational homology. In fact, the homology of the groups $\beta'_n$ has been completely computed \cite{Frenkel, Callegaro}. In particular $\mathrm{dim}_\bQ H_n(\beta'_{3n};\bQ)=2$ and so the slope of $\tfrac{1}{3}$ in Theorem \ref{thm:CommutatorBraid} is optimal. (Incidentally, the abelian $\beta_n$-module $\bZ[H_1(\beta_n;\bZ)] = \bZ[t,t^{-1}]$ is precisely the determinant of the Burau representation described in Example \ref{Bureauex}.)

\medskip

To prove the second part of Theorem~\ref{thm:D} for a general group $G$, just as in the case of symmetric groups one replaces the groupoid $\cB_2$ by the groupoid $\cB_2^G$ with the same objects and with $\Hom_{\cB_2^G}((S_1,I_1),(S_2,I_2))=G\wr \Hom_{\cB_2}((S_1,I_1),(S_2,I_2))$, where the wreath product is formed using the functor from $\cB_2$ to the category $\Sigma$ of finite sets taking $(S,I)$ to $\del'S$, the boundary components of $S$ not containing $I$. One then checks that the resulting complex $S_n(A,X)$ for 
$A=(S,I)$ and $X$ the cylinder identifies with the complex denoted $A(\widetilde S;\{*\},\Lambda)^G$  in \cite{HatWah10}, where $\Lambda$ the set of centres of the discs glued on $S$ to form $\widetilde S$, and where we think of $G$ as a constant labelling system in the sense of \cite[Ex.~3.3]{HatWah10}. This is a complete join complex over  $A(\widetilde S;\{*\},\Lambda)$ and hence is $(n-2)$-connected by Propositions 7.2 and 3.5 of that paper. 

\begin{rem}\label{prebraidrem}{\rm 
The pre-braiding in the subcategory of punctured discs in $U\M_2$ does not define an actual braiding,  and $U\M_2$ and this subcategory are {\em not}
braided monoidal categories. Figure \ref{nonbraid} shows that the pre-braid relation holds, and that its symmetric version does not hold, showing that the category is not braided monoidal. 
\begin{figure}[h]
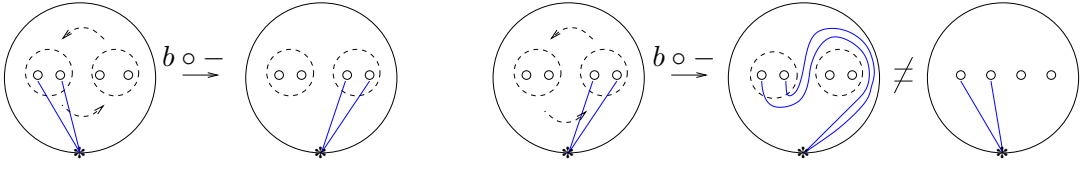

\begin{lpic}{notbraided3(0.5,0.5)}
\lbl[b]{50,26;$b\circ -$}
\lbl[b]{179,26;$b\circ - $}
\end{lpic}
\caption{Pre-braid relation and failure of the braiding on the pair $(\iota_2,id_2)$ in the homogeneous category associated to the braid groups.}\label{nonbraid}
\end{figure}
}\end{rem}

\subsubsection{Genus stabilisation for orientable surfaces}

We consider in this section the case when $X$ is the torus with one boundary, 
$$X=((S^1\x S^1)\minus \mathrm{int} D^2,I)$$ 
with 
$$A=(S,J)$$
any orientable surface. We can embed a graph $R=S^1\vee S^1\vee I$ in the torus in such a way that the complement of the graph in the surface is a disc (see Figure~\ref{surfacespin}).  Just as in the previous section, the Alexander trick then implies that an isotopy class of embedding from $X$ to $A\natural X^{\natural n}$ is determined by an isotopy class of embedding of $R$ in the surface fixed on the basepoints. 
\begin{figure}[h]
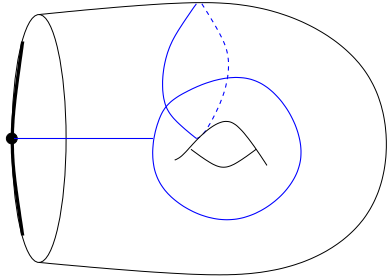

\begin{lpic}{surfacespin(0.45,0.45)}
\end{lpic}
\caption{Graph with complement is a disc in the torus with one boundary}\label{surfacespin}
\end{figure}

Moreover, a collection of $(p+1)$ such embeddings will define an embedding of $X^{\natural p+1}$ in $A\natural X^{\natural n}$ precisely when they can be made disjoint. This identifies $S_n(A,X)$ in this case with the complex of tethered chains studied by Hatcher--Vogtmann \cite{HatVoginfinity}. They compute a connectivity bound for this complex.

\begin{lem}(\cite[Prop.~5.5]{HatVoginfinity})
For  $X$ the torus with one boundary and $A$ any orientable surface, the complex $S_n(A,X)$ is $(\frac{n-3}{2})$-connected.
\end{lem} 

When $A=(S_{g,b},J)$ is an orientable surface of genus $g$ with $b\ge 1$ boundary components, and $X$ is still the torus with one boundary, we have that
$$A\natural X^{\natural n}=S_{g,b}\ \natural_n\ (S_1\x S_1)\minus \mathrm{int} D^2 \ \cong \ S_{g+n,b}$$ 
is a surface of genus $g+n$ with $b$ boundary component, and 
$$\Aut_{\U\M_2}(A\natural X^{\natural n})\cong \pi_0\Dif(S_{g+n,b}\ \textrm{rel}\ \del_0S)$$
is the mapping class group of $S_{g+n,b}$ fixing one of its boundary components pointwise. 

The most interesting choices of $A$ are those of genus 0, as connected sum with $X$ increases the genus. Applying Theorems~\ref{stabthm} and \ref{twistrange} to $U\M_2$ with the pair $(A\natural X,X)$, gives the following.

\begin{thm}
Let $X$ be the torus with one boundary, and $A=S_{0,b}$ a genus 0 orientable surface with $b\ge 1$ boundary components.  Let $F:(\U\M_2)_{A,X}\to \Z\operatorname{-Mod}$ be a coefficient system. Then the map 
$$H_i( \pi_0\Dif(S_{g,b}\ \textrm{rel}\ \del_0S);F(S_{g,b}))\rar  H_i(\pi_0\Dif(S_{g+1,b}\ \textrm{rel}\ \del_0S);F(S_{g+1,b}))$$
is: an epimorphism for $i \leq \tfrac{g-1}{2}$ and an isomorphism for $i \leq \tfrac{g-2}{2}$, if $F$ is constant; an epimorphism for $i \leq \tfrac{g-r-1}{2}$ and an isomorphism for $i \leq \tfrac{g-r-3}{2}$, if $F$ is split of degree $r$; an epimorphism for $i \leq \tfrac{g-1}{2}-r$ and an isomorphism for $i \leq \tfrac{g-3}{2}-r$, if $F$ is of degree $r$.
\end{thm}

\begin{rem}
For constant coefficients, the above theorem is part of Harer's classical stability theorem for mapping class groups \cite{Har85}. The range obtained here is better than Harer's original range, but is not the best known range, which can be found in \cite{Bol12,RW09}. For twisted stability, the result is due to Ivanov \cite[Thm.~4.1]{Iva93} for $b=1$. Ivanov's result was generalised by Cohen--Madsen to the case $b>1$ but for mapping class groups fixing all the boundary components \cite[Thm.~0.4]{CohMad09}. 
This last result can either be recovered from the above by Leray--Hochschild--Serre spectral sequence arguments, or via a modification of our framework allowing partial
monoids. 
\end{rem}

Unfortunately the mapping class group of an orientable surface of large enough genus is perfect, and so has no interesting abelian modules. Hence Theorem~\ref{abstabthm} and the second part of Theorem~\ref{twistrange} give no more information in this case.

\subsubsection{Genus stabilisation for non-orientable surfaces}
Finally we consider the case when
$$X=(M,I)$$
is a M\"{o}bius band and
$$A=(S,J)$$ 
is any object of the category $\U\M_2^-$. Again, we can characterise the maps $X^{\natural p}\to A\natural X^{\natural n}$ in terms of arcs: Note first that there is a unique isotopy class of arcs in $M$
with endpoints on the boundary whose complement in $M$ is a disc (see Figure~\ref{mobiusspin}).
\begin{figure}[h]
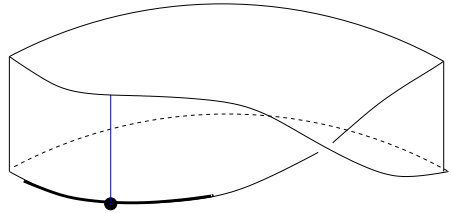

\begin{lpic}{mobius(0.4,0.4)}
\end{lpic}
\caption{Arc with complement is a disc in the M\"{o}bius band}\label{mobiusspin}
\end{figure}
Such an arc is 1-sided in the sense of \cite[Def.~2.1]{Wah08}, and again by the Alexander
trick, isotopy classes of embeddings of $M$ in $S\natural M^{\natural n}$ correspond to isotopy classes of embedded arcs. Now morphisms in $U\M_2^-$ with target a non-orientable surface were defined to be embeddings with non-orientable or genus 0 complement. This restriction corresponds to the notion of a {\em better} 1-sided arc in \cite[Sec.~3]{Wah08} as long as $n\ge 2$. (In \cite{Wah08} and  \cite{RW09}, a choice was made not to allow genus 0 complement as it did not help---or hinder. As we will see, it does not make a difference here either.) More generally, an embedding of $M\natural\cdots\natural M$ in $S\natural M^{\natural n}$ is modelled by a collection of such disjointly embedded 1-sided arcs attached next to one another at the marked interval. By sliding the start points of the arcs along the preceding arcs, one obtains an isomorphism between the $(n-2)$-skeleton of $S_n(A,X)$ and that of the complex denoted $\C_0(S)$ in  \cite{RW09}.  
A connectivity bound for $\C_0(S)$  is computed in that paper, and immediately yields the same connectivity for the relevant skeleton: 

\begin{lem}(\cite[Thm.~A.2]{RW09}) For $X$ the M\"{o}bius band and $A$ any surface, the complex $S_n(A,X)$ is $(\frac{n-5}{3})$-connected. 
\end{lem}

When $A=(F_{h,b},J)$ is a non-orientable surface of non-orientable genus $h$ with $b\ge 1$ boundary components (i.e.~$F_{h,b}$ is a connected sum of $h$ projective planes with $b$ open discs removed), we have 
$$A\natural X^{\natural n}=F_{h,b}\natural M^{\natural n} \cong F_{h+n,b}$$ 
is a non-orientable surface of genus $h+n$ with $b$ boundary components. Also, 
$$\Aut_{U\M^-_2}(A\natural X^{\natural n})\cong \pi_0\Dif(F_{h+n,b}\ \textrm{rel}\ \del_0F)$$
is the mapping class group of $F_{h+n,b}$ fixing one boundary component of $F$ pointwise.

Again, the most interesting cases are when $A$ is a genus 0 surface. Applying Theorems~\ref{stabthm} and \ref{twistrange} to $(A\natural X^{\natural 3},X)$ for $X$ the M\"{o}bius band and $A$ is a genus 0 surface with $b$ boundary component gives the
following.

\begin{thm}\label{nmcgtwist}
Let $X$ be the M\"{o}bius band, and $A=F_{0,b}$ a genus 0 orientable surface with $b\ge 1$ boundary components. Let $F:(\U\M^-_2)_{A,X}\to \Z\operatorname{-Mod}$ be a coefficient system. Then the map 
$$H_i( \pi_0\Dif(F_{n,b}\ \textrm{rel}\ \del_0F);F(F_{n,b}))\rar  H_i(\pi_0\Dif(F_{n+1,b}\ \textrm{rel}\ \del_0F);F(F_{n+1,b}))$$
is: an epimorphism for $i \leq \tfrac{n-3}{3}$ and an isomorphism for $i \leq \tfrac{n-4}{3}$, if $F$ is constant; an epimorphism for $i \leq \tfrac{n-r-3}{3}$ and an isomorphism for $i \leq \tfrac{g-r-6}{3}$, if $F$ is split of degree $r$; an epimorphism for $i \leq \tfrac{n-3}{3}-r$ and an isomorphism for $i \leq \tfrac{n-6}{3}-r$, if $F$ is of degree $r$.
\end{thm}

For constant coefficients, this recovers the first part of Theorem A in \cite{Wah08}, with almost the improved range of \cite[\S 1.4]{RW09}. The stability with twisted coefficients had not been considered before as far as we know. 

\medskip

The subgroup $\mathcal{T}_{n,1} \subset \pi_0\Dif(F_{n,1}\ \textrm{rel}\ \del_0F)$ generated by Dehn twists has index 2, and for $n \geq 7$ is perfect \cite[Cor.~6.4, Thm.~8.1]{Stu09}. 
Thus, for $n \geq 7$, these groups are the commutator subgroups of $\pi_0\Dif(F_{n,1}\ \textrm{rel}\ \del_0F)$. Theorems~\ref{abstabthm} and \ref{twistrange} give the
following.

\begin{thm}\label{nmcgtwist2}
Let $X$ and $A$ be as above. Let $F:(\U\M^-_2)_{A,X}\to \Z\operatorname{-Mod}$ be a coefficient system. Then the map 
$$H_i(\mathcal{T}_{n,1} ;F(F_{n,1}))\rar  H_i(\mathcal{T}_{n+1,1};F(F_{n+1,1}))$$
is, for $n \geq 7$: an epimorphism for $i \leq \tfrac{n-4}{3}$ and an isomorphism for $i \leq \tfrac{n-6}{3}$, if $F$ is constant; an epimorphism for $i \leq \tfrac{n-2r-4}{3}$ and an isomorphism for $i \leq \tfrac{g-2r-7}{3}$, if $F$ is split of degree $r$; an epimorphism for $i \leq \tfrac{n-4}{3}-r$ and an isomorphism for $i \leq \tfrac{n-7}{3}-r$, if $F$ is of degree $r$.
\end{thm}

Theorems \ref{nmcgtwist} and \ref{nmcgtwist2} imply Theorem \ref{thm:I}. There are few calculations of stable homology available, but the functor $H_1(-;\bZ) : (U\M^-_2)_{F_{0,1}, M} \to \Z\operatorname{-Mod}$ defines a split coefficient system of degree 1 and Stukow \cite{Stukow} has shown that
$$H_1( \pi_0\Dif(F_{n,1}\ \textrm{rel}\ \del_0F);H_1(F_{n,1};\bZ))\cong \bZ/2 \oplus \bZ/2$$
for $n \geq 7$.

\subsection{Mapping class groups of 3-manifolds}\label{3mfdex}

As we have seen in Section~\ref{genfreeprod}, mapping class groups of oriented 3-manifolds are closely related to automorphisms of free products of
groups. Homological stability with constant coefficients for mapping class groups of orientable 3-manifolds was shown to hold in great generality in \cite{HatWah10}. 
One can combine the connectivity results of that paper with Theorem~\ref{twistrange} to obtain versions of many of the stability theorems of \cite{HatWah10}
with twisted coefficients. We briefly sketch how this can be done. 

\medskip

Let $\M_3^+$ denote the groupoid with objects the pairs $(M,D)$ consisting of an oriented 3-manifold $M$ with boundary $\partial M$, and a marked disc $D\inc \del M$, and with
morphisms given by isotopy classes of orientation preserving diffeomorphisms restricting to the identity on the disc.  Just as in the case of surfaces
considered above, boundary connected sum along half-discs defines a monoidal structure $\natural$ on $\M_3^+$, with unit $(D^3,D)$. 
 The mapping class group of $D^3$ relative to its boundary, or a disc in its boundary, is trivial \cite[p.\ 3]{Cerf68}, so we can again make the unit strict by defining connected sum with $(D^3,D)$ to be the identity. 
A zero divisor would be a simply connected 3-manifold with boundary a sphere, but such a manifold is necessarily the 3-disc by the affirmed Poincar{\'e} conjecture, 
there are again no zero divisors. 
One can show that this monoidal structure is symmetric. (The braiding is a higher dimensional analogue of half a Dehn twist in a neighbourhood of the marked discs, whose square is isotopic to the identity.) 

Let $\U\M_3^+$ be the associated pre-braided category. Just as for surfaces, a morphism $[X,f]:(M_1,D_1)\to (M_2,D_2)$ in $\U\M_3^+$ is determined by the embedding $f|_{M_1}:(M_1,D_1)\inc (M_2,D_2)$, and morphisms of $\U\M_3^+$ correspond exactly to such embeddings taking $D_1^+$ to $D_2^+$ and $\del'M_1$ to $\del'M_2$, where $\del' M_i=\del M_i\minus \del_0M_i$ is the union of the components of $\del M_i$ that do not contain $D_i$. Moreover the image of $M_1$ in $M_2$ is separated from its complement by a disc inside $M_2$, namely the image of $D_1^-$,  half of whose boundary is the ``equator'' of $D_2$ and the other half in $\del_0 M_2\minus D_2$. 

In \cite{HatWah10}, two main types of stabilisations are considered: stabilisation by connected sum (denoted $\#$) and stabilisation by boundary connected sum (denoted $\natural$). As already mentioned in the introduction, connected sum with a manifold $N$ is the same as boundary connected sum with $N\minus D^3$ along a disc in $\del D^3$, so both types of connected sums actually arise from the monoidal structure considered here. The interesting stabilising directions $X$ to choose are the building blocks in the category, which are the objects $(P\minus D^3,D)$ with $D\inc \del D^3$, and $(P,D)$  with $D\inc \del_0P\not\cong S^2$. These correspond respectively to the connected sum, and boundary connected sum stabilisation of \cite{HatWah10} and we will show now that the complexes $S_n(A,X)$ identify with the complexes used to prove stability in \cite{HatWah10} (though under the assumption that $A$ is irreducible in the second case).

Let $P$ be a prime manifold, $P^0:=P\minus D^3$ and fix a disc $D\inc \del D_3\subset \del P^0$. Let $(M,D_M)$ be an object of $\U\M_3^+$, and denote by $\del_0M$ the boundary component of $M$ containing $D_M$. A vertex in $S_n((M,D_M),(P^0,D))$ is an embedding $$f:(P^0,D)\ \inc\ (M\natural (P^0)^{\natural n},D_M)\cong  (M\# P^{\# n},D_M)$$ taking  $D^+$ to $D_M^+$ and $\del'P^0$ to $\del' M$. The image of $P^0$ in $M$ is separated from the rest of $M$ by the image of $D^-$ whose boundary lies inside $D_M^+\cup f(\del_0P\minus D)\cong D^2$, and hence can be pushed to lie in $\del D_M$. Embedding such a disc is equivalent to embedding a sphere inside $M$ connected to $D_M$ by a framed arc. This shows that the vertices of $S_n((M,D_M),(P^0,D))$ for $P\neq S^1\x S^2$, identify with the vertices of the simplicial complex denoted $X^{FA}(M\# P^{\# n},P,\del_0M,x_0)$ with $x_0$ a chosen point in $D_M$. The same holds for $P=S^1\x S^2$, though in this case one uses that the embedding of $P^0$ is determined by a non-separating disc and an arc transverse to it, much in the spirit of what we did with surfaces in the previous section. 
One checks that the simplicial complexes $S_n((M,D_M),(P^0,D))$  and $X^{FA}(M\# P^{\# n},P,\del_0M,x_0)$  are  isomorphic in both cases. As Propositions 4.2 and 4.5 of \cite{HatWah10} show that $X^{FA}$ is $(\frac{n-3}{2})$-connected, this establishes LH3 with $k=2$ for the pairs $((M \# P,D_M),(P^0,D))$.

One can likewise work with the groupoid $\overline{\M}_3^+$ where we replace the isotopy classes of diffeomorphisms by such isotopy classes modulo Dehn twists along spheres and discs. The same Propositions 4.2 and 4.5 of \cite{HatWah10} apply to show that $\U\overline{\M}_3^+$ satisfies LH3 with $k=2$ with respect to all pairs $((M \# P,D_M),(P^0,D))$ as above.

Suppose now that $M$ is irreducible and $P$ is a prime manifold boundary (hence also irreducible) with $D\inc \del_0 P\not \cong S^2$. As before, a vertex in $S_n((M,D_M),(P,D))$ is an isotopy class of embedding $(P,D)\inc (M,D_M)$ taking $D^+$ to $D_M^+$ and $\del'P$ to $\del' M$, and such that the image of $P$ is separated from its complement in $M$ by a disc, the image of $D^-$, half of whose boundary lies in $D_M$. Pushing this disc a little away from $D_M$  and replacing the intersection with $D_M$ in $\del_0M$  by an arc, when $P\neq S^1\x D^2$, we get exactly a vertex of the complex denoted $Y^A(M\natural P^{\natural n},P,D_M,x_0)$ in \cite{HatWah10}, and the same holds for $P=S^1\x D^2$ though again with a slightly different description of the embeddings. The complexes  $S_n((M,D_M),(P,D))$  and $Y^A(M\natural P^{\natural n},P,D_M,x_0)$ are isomorphic and Theorems 8.3 and 8.5 in that paper show that $Y^A$ is  $(\frac{n-3}{2})$-connected, which establishes LH3 with $k=2$ for the pairs $((M \natural P,D_M),(P,D))$ in this case too.

Local cancellation holds for this groupoid for all pairs $((M,D_M),(P^0,D))$ with $D\inc \del D^3$ by the prime decomposition theorem for 3-manifolds, and it holds for all pairs $((M,D_M),(P,D))$ with $M$ and $P$ irreducible, using the existence and uniqueness of compression bodies in irreducible 3-manifolds \cite[Thm.~2.1]{Bon83}. 
The hypothesis of Theorem \ref{universal} (b)  is established in \cite[Proposition 2.3]{HatWah10} for pairs $((M,D_M),(P^0,D))$ and in Section 8 of that paper for pairs $((M,D_M),(P,D))$ with $M$ and $P$ irreducible, so this theorem shows that $\U\M_3^+$ and $\U\overline{\M}_3^+$ are locally homogeneous categories in all cases. 
 Moreover their underlying groupoids are $\M_3^+$ and $\overline \M_3^+$ by Proposition~\ref{underlying} and the classification of 3-manifolds. 

Applying Theorems~\ref{stabthm} and \ref{twistrange} to $\U\M_3^+$ and $\U\overline{\M}_3^+$ repeatedly on prime summands, we obtain the following.

\begin{thm}\label{3mfdstab}
Let $(N,D)$ and $(M,D_M)$ be objects of $\U\M_3^+$. If $D\inc \del_0N\not\cong S^2$, assume that $M$ is irreducible. Let $F:(\U\M_3^+)_{(M,D_M),(N,D)}\to \Z\operatorname{-Mod}$ be a coefficient system. Then the map
\begin{align*}
H_i( \pi_0\Dif(M\natural N^{\natural n}\ \textrm{rel}\ D_M),&F(M\natural N^{\natural n},D_M)\\
&\rar  H_i(\pi_0\Dif(M\natural N^{\natural n+1} N\ \textrm{rel}\ D_M),F(M\natural N^{\natural n+1},D_M)
\end{align*}
is: an epimorphism for $i \leq \tfrac{n-1}{2}$ and an isomorphism for $i \leq \tfrac{n-2}{2}$, if $F$ is constant; an epimorphism for $i \leq \tfrac{n-r-1}{2}$ and an isomorphism for $i \leq \tfrac{g-r-3}{2}$, if $F$ is split of degree $r$; an epimorphism for $i \leq \tfrac{n-1}{2}-r$ and an isomorphism for $i \leq \tfrac{n-3}{2}-r$, if $F$ is of degree $r$.

Similarly, assuming that $D\inc \del_0N\cong S^2$, and $\overline{F}:(\U\overline{\M}_3^+)_{(M,D_M),(N,D)}\to \Z\operatorname{-Mod}$ is a coefficent system, the map 
\begin{align*}
H_i( \pi_0\Dif(M\natural N^{\natural n}\ \textrm{rel}\ \del_0M)&/_{\textrm{twists}},\overline{F}(M\natural N^{\natural n}))\\
&\rar  H_i(\pi_0\Dif(M\natural N^{\natural n+1}\ \textrm{rel}\
\del_0M)/_{\textrm{twists}},\overline{F}(M\natural N^{\natural n+1}))
\end{align*}
is an epimorphism or isomorphism in the same range of degrees.
\end{thm}

\begin{rem}\label{symaut}
The above theorem generalises to twisted coefficients a slight variation of Theorem 1.1 (i) and well as Theorem 1.8 (i) of \cite{HatWah10}. (The variation is that, in Theorem 1.1 of \cite{HatWah10}, some boundary components of $M$ containing $\del_0M$ are assumed to be fixed by the diffeomorphisms. In the case when $\del_0M\cong S^2$ is the only fixed component, these groups are though isomorphic.) In the case when $M=D^3$ and $N=(S^1\x D^2)\minus D^3$, the above theorem gives for example a twisted stability theorem for symmetric automorphisms of free groups, the subgroup $\Sigma\Aut(F_n) \subset \Aut(F_n)$ generated by permutations of the generators, taking the inverse of a generator, and conjugating a generator by another one. This establishes Theorem \ref{thm:G} in the case $G_n = \Sigma\Aut(F_n)$.  
\end{rem}

Theorems~\ref{abstabthm} and \ref{twistrange} also give results for abelian coefficient systems for this family of groups, but little is known about their abelianisations.
Theorem~\ref{thm:J} follows from these and Theorem~\ref{3mfdstab}.

\bibliographystyle{plain}

\bibliography{biblio}

\end{document}